\documentclass[11pt]{article}
\usepackage[margin=2cm]{geometry}
\usepackage{graphicx}
\usepackage{caption}
\usepackage{subcaption}
\usepackage{hyperref}
\usepackage{amsmath}
\usepackage{amsthm}
\usepackage{amssymb}
\everymath{\displaystyle}
\usepackage{bm}
\usepackage{xcolor}
\usepackage{float}

\usepackage{tabularx}
\usepackage{booktabs}
\usepackage[labelfont=bf,format=plain,justification=raggedright,singlelinecheck=false]{caption}

\newtheorem{theorem}{Theorem}[section]
\newtheorem{definition}{Definition}[section]
\newtheorem{proposition}{Proposition}[section]
\newtheorem{remark}{Remark}[section]
\newtheorem{corollary}{Corollary}[section]
\newtheorem{lemma}{Lemma}[section]
\DeclareMathOperator*{\diag}{diag}
\usepackage{float}
\DeclareMathOperator*{\argmin}{arg\,min}

\everymath{\displaystyle}

\providecommand{\keywords}[1]
{
  \small	
  \textbf{\textit{Keywords---}} #1
}

\usepackage{authblk}
\title{Eigenvalues Distributions and Control Theory}

\author[1, 2, $\star$]{N. LAMSAHEL}
\author[1]{A. EL AKRI}
\author[1]{A. RATNANI}
\affil[1]{ {\scriptsize Vanguard Center, Mohammed VI Polytechnic University, Green City, Morocco}}
\affil[2]{ {\scriptsize  LMPA, Université du Littoral Cote d’Opale, 50 rue F. Buisson, 62228 Calais-Cedex, France}}
\affil[$\star$]{ {\scriptsize  email: noureddine.lamsahel@um6p.ma}}
\begin{document}

\maketitle
\tableofcontents

\begin{table}
\caption{The list of symbols and notations used in this paper}
\begin{tabularx}{\textwidth}{>{\centering\arraybackslash}p{0.15\textwidth}X}
\toprule
\textbf{Notation} & \textbf{Description}  \\
\midrule
  $N_j^p$ &  the $j$-th $B$-spline function of degree $p$; see \eqref{eq:definition-B-spline-1}-\eqref{eq:definition-B-spline-2} \\
  $\mathbb{S}^p$ & Schoenberg space of degree $p$; see \eqref{eq:Schoenberg-space} \\
  $\mathbb{S}_0^p$  & Schoenberg space of degree $p$ with vanishing boundary; see \eqref{eq:Schoenberg-Dirichlet-space} \\
  $\mathbf{C}_{[0,1]}$ & the space of admissible reparametrizations; see  \eqref{eq:reparametrization-space} \\
  $M_{\phi, n}^p$ & IgA mass matrix arising from the reparametrization $\phi$ of a uniform grid; see  \eqref{eq:mass-matrix-dif} \\
  $K_{\phi, n}^p$ & IgA stiffness matrix arising from the reparametrization $\phi$ of a uniform grid; see  \eqref{eq:stiffness-matrix-dif} \\
  $L_{\phi, n}^p$ & the matrix related to the IgA approximation of the Dirichlet-Laplace eigenvalue problem; $L_{\phi,n}^p=\left( M_{\phi,n}^p\right)^{-1} K_{\phi,n}^p$ \\
  $N$& the size of the matrix $L_{\phi, h}^p$, given by $N=n+p-2$\\
  $R_\omega$/$Rg(\omega)$ & the essential range of $\omega$/the range of $\omega$; see \eqref{eq:essential-range} \\
  $OUT(p,n)$ & the number of outliers \\
  $\mathcal{I}(p,n)$ & the set of indices of eigenvalues after removing the outliers:\\ &\hspace{2cm}$\mathcal{I}(p, n) = \left\{1, \cdots, N - OUT(p,n) \right\}$ \\
  $\mathcal{N}_p$ & the cardinal $B$-spline function of degree $p$; see \eqref{eq:cardinal-B-spline-1}-\eqref{eq:cardinal-B-spline-2} \\
  $\omega_{\phi}^p$ & the spectral symbol of $n^{-2} L^p_{\phi, n}$, see \eqref{eq:spectralsymbo-full-matrix} \\
  $\xi_\phi^p$ & the monotone rearrangement of $\omega_{\phi}^p$; see Definition \ref{def:monotone-rearrangement} \\
  $\widetilde{\delta}_{n}$ & the approximate discrete gap; see \eqref{eq:approximate-discrete-gap} \\
  $\delta_n^p$ &  the discrete gap; see \eqref{eq:discrete-gap} \\
\bottomrule
\end{tabularx}
\end{table}

\begin{abstract} 
This work deals with the isogeometric Galerkin discretization of the eigenvalue problem related to the Laplace operator subject to homogeneous Dirichlet boundary conditions on bounded intervals. This paper uses GLT theory to study the behavior of the gap of discrete spectra toward the uniform gap condition needed for the uniform boundary observability/controllability problems. The analysis refers to a regular $B$-spline basis and concave or convex reparametrizations. Under suitable assumptions on the reparametrization transformation, we prove that structure emerges within the distribution of the eigenvalues once we reframe the problem into  GLT-symbol analysis. We also demonstrate numerically, that the necessary average gap condition proposed in \cite{bianchi2018spectral} is not equivalent to the uniform gap condition. However, by improving the result in \cite{bianchi2021analysis} we construct sufficient criteria that guarantee the uniform gap property.
\end{abstract}

\noindent\keywords{Laplace operator, Isogeometric Galerkin discretization, Reparametrization, GLT theory, Uniform gap, Average gap, Eigenvalues, Spectral symbol. }

\section{Introduction}
%$\hspace{0.5cm}$This paper focuses on the spectral analysis of the discrete eigenvalues problem that arises when isogeometric discretization is used to approximate the Laplace operator. The primary motivation for this study stems from the critical importance of uniform eigenvalue gap problems, especially in the context of boundary control problems.

\hspace{0.5cm} We consider the one-dimensional boundary controllability problem for the wave equation: given $T>2$ and initial data $(\varphi^0,\varphi^1) \in L^2(0,1) \times H^{-1}(0,1)$, the goal is to find a control function  $v \in L^2(0, T)$ such that the solution of the system 

\begin{equation}\label{pd0}
\begin{cases}
\dfrac{\partial^2 \varphi}{\partial t^2}(t,x)-\dfrac{\partial^2 \varphi}{\partial x^2}(t,x) =0, \quad & (t,x) \in (0,T) \times (0,1), \vspace*{0.25cm}\\
\varphi(t,0)=0, \quad \varphi(t,1)=v(t),  \quad & t \in (0,T), \vspace*{0.25cm} \\
\varphi(0,t)=\varphi^0(x)  , \quad \frac{\partial \varphi}{\partial t} (0,t)= \varphi^1(x),  \quad & x \in (0,1),
\end{cases}
\end{equation}
satisfies 
$$
\varphi(T,x) = \frac{\partial \varphi}{\partial t}(T,x) =0, \quad \text{a.e } x \in (0,1).
$$

Since the wave equation \eqref{pd0} is linear and time-reversible if system \eqref{pd0} is controllable at zero, we can also drive the system at time $T$ to any desired state: $\varphi(T, \cdot)=\varphi_T^0$ and $\varphi'(T, \cdot)=\varphi_T^1$, where $(\varphi_T^0, \varphi_T^1)\in L^2(0,1) \times H^{-1}(0,1)$.

There are several approaches to tackling this problem (see, for instance, \cite{lions1988controlabilite,russell1978controllability,tucsnak2009observation}). However, the Hilbert Uniqueness Method (HUM) stands out as the only systematic and algorithmic method \cite{glowinski1995exact,lions1988controlabilite}. This method is of particular significance due to its capability to construct controls with minimal $L^2$-norm, often referred to as HUM controls. This construction is based on the creation of an abstract Hilbert space, relying on the observability of the adjoint system. The numerical implementation of the method involves discretizing the continuous problem, and then considering the possible convergence of the discrete controls towards the HUM control of the continuous system. For problems of this nature, space discretization is the most relevant and generally determines whether the discrete controls converge or diverge. 

This discrete approach requires a meticulous analysis to comprehend how the concepts of controllability and observability behave during the discretization process. To be more precise, it is widely acknowledged that almost all classical spatial discretizations of \eqref{pd0} using uniform meshes result in the divergence of the discrete controls as the discretization parameter approaches zero. This is observed, for example, in the case of finite-differences and finite-elements semi-discretizations (see \cite{glowinski1995exact,infante1999boundary}).
The phenomenon arises due to the presence of high-frequency spurious solutions that propagate with arbitrarily small velocities, destroying the uniform observability of the discrete adjoint system, which leads to the lack of uniform controllability of the discrete approximations of the wave equation. Consequently, even with regular initial conditions, there are cases where any corresponding sequence of discrete controls becomes unbounded in the $L^2$-norm. Several techniques have been developed in the literature to overcome the aforementioned high frequencies problem \cite{castro2006boundary,ervedoza2016numerical,glowinski1995exact,infante1999boundary,micu2002uniform,micu2009uniform,munch2005uniformly,negreanu2004convergence}. For a comprehensive overview of the state of the art, readers can refer to the survey article   \cite{zuazua2005propagation} and the references therein.

When dealing with the $1-d$ uniform observability problem, the usual strategy involves the utilization of Fourier series expansions. In this context, Ingham-type inequalities 
reduce the problem to the uniform observability of the discrete eigenvectors and a uniform gap condition on the eigenvalues. While proving the uniform observability of discrete eigenvectors often involves direct computation employing a multiplier technique, studying the uniform gap problem is more challenging and requires a detailed description of the discrete eigenvalues, which is usually only feasible in specific cases, for instance, in the case when the wave equation \eqref{pd0} is discretized using the finite-differences \cite{infante1999boundary} or mixed finite element \cite{castro2006boundary} methods. 

This highlights the central importance of addressing the uniform gap problem in the development of numerical approximation schemes for control problems. In this context, the pioneering work of S. Ervedoza et al. \cite{ervedoza2016numerical} proposes a solution to restore the uniform observability property involving the utilization of non-uniform meshes. The main idea involves building a variety of non-uniform finite-differences and finite-elements semi-discretizations of the wave equation through the use of an appropriate concave diffeomorphism, what we call a reparametrization, which transforms the uniform mesh into a non-uniform one.  It is worth mentioning, however, that while the property of a uniform gap has been identified through numerical analysis, rigorous analytical proof remains an open problem.

To the best of our knowledge, the first theoretical investigation of the uniform gap condition in this particular direction is attributed to Bianchi et al. \cite{bianchi2018spectral}. The main focus of that paper was the analysis of what is called the average spectral gap concerning the one-dimensional Laplace operator on non-uniform meshes. It's worth noting here that the average spectral gap condition is a necessary condition for the uniform spectral gap. The analysis in \cite{bianchi2018spectral} applies to different approximation schemes, including finite differences, Lagrangian finite elements, and a scheme coming from the Isogeometric Analysis (IgA) approximation technique \cite{hughes2005isogeometric}. In particular, the study demonstrated that the mesh introduced in \cite{ervedoza2016numerical} effectively preserves the uniform gap condition in the average sense. The main tool employed in the paper was the Generalized Locally Toeplitz (GLT) symbolic calculus \cite{garoni2017generalized,garoni2017generalized2}. The main advantage of using the GLT theory, as indicated in \cite{bianchi2018spectral}, lies in its capacity to provide a general conceptual background to address the problem of uniform gaps within a broader context. Additionally, it serves as a unifying tool for the spectral analysis of various schemes (FDs, FEs, IgA...). 

Despite preserving the average gap property, achieving uniform observability requires a uniform gap condition. In the present paper, we address the problem of the uniform gap and, under specific and possible conditions on the spectra of the discrete system after an appropriate reparametrization, we demonstrate the uniform preservation of the eigenvalue gap. Our analysis focuses on the Galerkin IgA method utilizing regular $B$-splines of degree $p$ with global regularity $C^{p-1}$. However, it's important to note that analyzing the uniform gap in this context presents significant difficulties, primarily due to the lack of information on discrete eigenvalues; unlike the case of finite differences and $\mathbb{P}_1$ finite elements on uniform meshes. Specifically, the analysis in the works \cite{castro2006boundary,infante1999boundary} cannot be straightforwardly extended. Therefore,  a careful examination of the behavior of discrete eigenvalues under the mesh reparametrization is necessary.

The authors' decision to study the average gap rather than the needed uniform gap property in \cite{bianchi2018spectral} is driven by a fundamental challenge. This challenge arises from the difficulty to construct a reparametrization that enables an approximation of  the square root of the eigenvalues up to a significant order, at least $o(h)$, where $h$ represents the discretization step when using uniform sampling of the symbol. Consequently, a thorough understanding of the connection between the set of reparametrizations and the eigenvalues is needed to tackle this problem effectively. In this context, and based on the results in \cite{bianchi2021analysis}, we have illustrated the potential to gain valuable and previously undiscovered insights into the relationship in question. Subsequently, as part of our effort to establish a sufficient condition for the uniform gap property, we have exhibited more favorable outcomes compared to those in \cite{bianchi2021analysis}, particularly when a set of reparametrizations is considered. Combining these enhanced results with our analysis of the distribution of the eigenvalues and supplementing them with numerical observations, we build a  feasible criterion on the discrete spectra to achieve the uniform gap property.

The main contributions of the paper are:

\begin{enumerate}
\item[(1)]  analysis of the distribution of the eigenvalues under the reparametrization, proving that the ordering of eigenvalues is precisely the ordering of the symbols (Theorem \ref{tdis1}). Additionally, we examine  how  the convexity of the symbol influences the  distribution of packed eigenvalues (Proposition \ref{cl2.5}). 
\item[(2)] we demonstrate a uniform version of the discrete Weyl's law, presented in Proposition \ref{gedis} and a more accurate approximation of the eigenvalues using the symbol in Corollary \ref{cl2}.
\item[(3)] examines the gap between consecutive square roots of eigenvalues, proving a uniform approximate gap condition when dealing with linear $B$-splines (Corollary \ref{lip=1}). In the more general case, we establish a uniform gap condition, provided that the sequence $(m(n)){n \geq 1}$ remains bounded (Theorem \ref{cl2.4}). Here, $m(n)$ is defined as 
$$
m(n) = \min \argmin_{1 \leq k \leq N-1} \left( \sqrt{\lambda_{k+1,h}}-\sqrt{\lambda_{k,h}} \right).
$$

\item[(4)] development of a numerical study, where the main objective is to validate our theoretical results and to construct specific examples illustrating the application of the hypothesis $(m(n))$ bounded. A summary of our numerical results is given in Table \ref{f23-f24}.
\end{enumerate}

In Section \ref{sec:preliminaries}, we introduce the notations, definitions, and preliminary results relevant to our analysis. Specifically, we start by defining the  $B$-spline functions and briefly drive the Galerkin Isogeometric Analysis (IgA) discretization for the eigenvalue problem associated with the wave equation \eqref{pd0}.  In the second subsection\ref{sec:preliminariesGLT} of this section, we provide an overview of the essential results of the abstract Generalized Locally Toeplitz (GLT) theory. At the end of this  Section, we provide the GLT symbol for our particular discretization. In Section  \ref{disIGA}, we use the GLT theory to analyze the impact of reparametrization as well as the convexity of the IgA symbol on the behavior of the discrete eigenvalues. We also demonstrate an improvement over the results in  \cite{bianchi2021analysis}, particularly showing a uniform discrete Weyl's law result. This result illustrates how simple convergence transforms into uniform convergence through a carefully selected reparametrization. Section \ref{unifgap} tackles the problem of the uniform gap. Under suitable conditions on the discrete spectra and with a well-selected reparametrization function, we demonstrate the preservation of the uniform gap. In Section \ref{sec:numerical-tests}, we provide several numerical tests to illustrate our theoretical results. Our primary objective in these numerical investigations is to analyze the equivalence between the average and strong gap conditions. Furthermore, we demonstrate the possibility of constructing reparametrizations that ensure the hypothesis of our theoretical result concerning the uniform gap. Finally, Section \ref{sec:conclusions} concludes the paper, summarizing the main results and providing some possible perspectives.

\section{Preliminaries}\label{sec:preliminaries}
\hspace{0.5cm}Consider the following one-dimensional Laplace eigenvalue problem with homogeneous Dirichlet conditions
\begin{equation}\label{pd1}
\left\{
\begin{array}{ll}
 -\partial_{xx}u=\lambda u,\quad x\in (0,1),\\
 u(0)=u(1)=0.
\end{array}
\right.
\end{equation}
It is well-established that the system (\ref{pd1}) possesses a family of exact non-trivial solutions given by $\lambda_k = \left( k\pi\right)^2$ $u_k = \sin({k\pi \cdot})$, for $k \in \mathbb{N}^*$; $\lambda_k$ is the $k$-th eigenvalue of the operator $-\partial_{xx}$ with Dirichlet boundary conditions, and $u_k$ corresponds to the associated eigenvector.  We observe that the family $\left\{ \frac{\sqrt{2}}{2} u_k \right\} $  forms an orthonormal basis of $L^{2}(0,1)$. Additionally, the eigenvalues $\lambda_k$ satisfy the uniform gap condition
\begin{equation}\label{eq:cont-gap-cond}
\sqrt{\lambda_{k+1}} - \sqrt{\lambda_{k}} = \pi, \quad \forall k\in\mathbb{N}^{*}.
\end{equation}
This uniform gap condition \eqref{eq:cont-gap-cond}, using Ingham’s inequality \cite{InghamA.Eartic}, ensures the boundary observability of the wave system \eqref{pd0} with a minimal time $T > 2$.

This paper deals with the analysis of the uniform gap property and the distribution of the IgA approximation of the eigenvalue problem \eqref{pd1}. The essential preliminary materials for our analysis are revisited in this section, starting with the IgA Galerkin discretization of \eqref{pd1} in Subsection \ref{sec:descritization}, followed by a short overview of the GLT theory in Subsection \eqref{sec:preliminariesGLT}.

\subsection{Galerkin B-spline IgA Discretization}\label{sec:descritization}
$\hspace{0.5cm}$In this part, we introduce an isogeometric Galerkin-based approach to discretize   the system \eqref{pd1}, employing $B$-spline functions \cite{de1978practical}. The method involves discretizing the weak form of problem \eqref{pd1} stated as follows: for $k \geq 1$, finds $u_k \in H^1_0(0,1)$ and $\lambda_k \in \mathbb{R}_+$ such that 
\begin{equation}\label{pd2}
A(u_k,v)=\lambda_k \;L(u_k,v),\quad \forall v \in H^1_0(0,1),
\end{equation}
where 
$$ A(u_k,v)=\displaystyle\int_0^1u_k'(x)\;v'(x) dx,\;\;\text{and}\;\; L(v)=\displaystyle\int_0^1 u_k(x)\;v(x)dx. $$
The next step involves constructing a finite-dimensional subspace that approximates the solution space $H^1_0(0,1)$. This subspace is determined by a finite set of basis functions. In the standard IgA discretization, these functions are constructed using $B$-spline functions.

We  consider non-periodic and uniform knot vectors of the form
$$
(t_j)_{0\leq j \leq 2p+n}=\left( \displaystyle\underbrace{0\cdots0}_{p+1},\dfrac{1}{n},\dfrac{2}{n},\cdots,\dfrac{n-1}{n},\displaystyle\underbrace{1\cdots1}_{p+1}   \right),
$$
where $n, p \in N^*$. The $B$-spline  functions of degree $p$ on these knots are defined recursively as follows (for instance,  see \cite{de1978practical})
\begin{equation}\label{eq:definition-B-spline-1}
N_j^p(t)=\dfrac{t-t_j}{t_{j+p}-t_{j}}N_j^{p-1}(t) +  \dfrac{t_{j+p+1}-t}{t_{j+p+1}-t_{j+1}}N_{j+1}^{p-1}(t),\quad\text{for}\;\;  0\leq j\leq p+n-1,  
\end{equation}
with
\begin{equation}\label{eq:definition-B-spline-2}
N_j^0(t)=\mathcal{X}_{[t_j,t_{j+1})}(t),\quad\text{for}\;\;0\leq j\leq p+n-1.
\end{equation}
Here, $p+n$ is the number of $B$-spline functions, and in cases where a fraction has a zero denominator, we assume it to be zero. We can then define the Schoenberg space
\begin{equation}\label{eq:Schoenberg-space}
\mathbb{S}^{p} = span \left\{ N_j^p: \; j=0, \dots,p+n-1 \right\}.
\end{equation}
In classical spline approximation theory (check \cite{de1978practical}), it is well  known that  $\mathbb{S}^{p}$ coincides with the space of splines of degree $p$ and smoothness $p-1$, namely
$$
\mathbb{S}^{p} = \left\{ s\in\mathcal{C}^{p-1}([0,1]\,\,s|_{\left[ i/n, \, (i+1)/n \right)}\in \mathbb{P}_{p},\;\; i=0,\dots,n-1 \right\},
$$ 
where $\mathbb{P}_{p}$ denotes the space of polynomials of degree at most $p$. Additionally, the B-spline basis has the following properties:
\begin{itemize}
\item Local support property:
$$
\text{supp}(N_j^p)=[t_j,t_{j+p+1}],\quad j=0,\dots,n+p-1.
$$

\item Vanishment on the boundary:
$$
N_j^p(0)=N_j^p(1)=0,\quad j=1,\dots,n+p-2.
$$

\item Nonnegative partition of the unity:
$$
N_j^p(x)\geq 0,\quad \sum_{j=0}^{n+p-1} N_j^p(x)=1, \quad x\in[0,1],\quad j=0,\dots,n+p-1.
$$
\end{itemize}
The isogeometric approximation of $H^1_0([0,1])$ is given by 
\begin{equation}\label{eq:Schoenberg-Dirichlet-space}
\mathbb{S}^{p}_0 = \left\{ s\in\mathcal{C}^{p}([0,1]\,\,s|_{\left[ i/n, \, (i+1)/n \right)}\in \mathbb{P}_{p},\;\;\text{and}\;\;s(0)=s(1)=0\;\; i=0,\dots,n-1 \right\}.
\end{equation} 
Then, our discrete solutions  $(u_{k,h}, \lambda_{k,h}) \in \mathbb{S}^{p}_0 \times \mathbb{R}_+$ satisfy the approximate weak formulation
\begin{equation}\label{eq:discrete-weak-form}
A(u_{k,h},v_h)=\lambda_{k,h} \;L(u_{k,h},v_h),\quad \forall v_h \in \mathbb{S}^{p}_0,
\end{equation}
where $h$ refers to the discretization parameter defined by $h=1/n$. We use the standard basis for $\mathbb{S}^{p}_0$ formed by the $B$-spline functions $\{N_{1}^p, \cdots, N_{p+n-2}^p\}$ that vanish at the boundary. 
Equation \eqref{eq:discrete-weak-form} can be expressed as a finite-dimensional eigenvalues problem
\begin{equation*}
\left[\left( M_n^p\right)^{-1} K_n^p \right] \mathbf{u}_{k,h} = \lambda_{k,h} \mathbf{u}_{k,h},
\end{equation*}
where $\mathbf{u}_{k,h}$ is the coefficients vector of $u_{k,h}$ with respect to the  basis $\{N_{1}^p, \cdots, N_{p+n-2}^p\}$ and $M_n^p$ and $K_n^p$ are the  mass and stiffness matrices, respectively 
\begin{equation*}
(M_n^p)_{i,j}= \int_0^1 N_i^p(x) N_j^p(x) dx, \quad (K_n^p)_{i,j} = \int_0^1 (N_i^p)'(x) (N_j^p)'(x) dx, \quad \text{for } 1 \leq i,j \leq p+n-2.
\end{equation*} 
$\hspace{0.5cm}$As mentioned in the introduction, it has been noted that uniform grids do not always ensure a uniform gap condition. However, Bianchi and Serra-Capizzano, in their work \cite{bianchi2018spectral}, demonstrated that by employing a convex or concave reparametrization of the domain $[0,1]$, it is possible to ensure the preservation of the average gap property. This average gap is a necessary condition for the needed uniform gap property.\\
In our paper, our primary focus is to analyze the distribution of eigenvalues in cases where the average gap condition is satisfied. Additionally, we endeavor to formulate a sufficient condition that guarantees the fulfillment of the uniform gap property. Following the ideas of \cite{bianchi2018spectral}, we introduce the space
\begin{equation}\label{eq:reparametrization-space}
\mathbf{C}_{[0,1]}=\left\{ \phi\in\mathcal{C}^2([0,1]):\; \phi^{'}>0,\;[\phi^{''}>0 \text{ or }\phi^{''}<0],\; \phi(0)=0 \text{ and } \phi(1)=1       \right\}.
\end{equation}
For $\phi\in \mathbf{C}_{[0,1]} $, we define the following basis functions by pullback  under transformation $\phi$
$$
B_j^p = N_j^p \circ \phi^{-1}, \quad \text{for } 1 \leq j \leq p+n-2,
$$
and aim to approximate the exact eigenpairs $(u_k, \lambda_k)$ using the standard Galerkin method, where the discrete solution space is given by $span \{B_1^p, \cdots, B_{p+n-2}^p \}$. Simple computations lead to the following expressions for the mass and stiffness matrices
\begin{equation}\label{eq:mass-matrix-dif}
(M_{\phi, n}^p)_{i,j}= \int_0^1 \left| \phi'(x) \right| N_i^p(x) N_j^p(x) dx, \quad \text{for } 1 \leq i,j \leq p+n-2.
\end{equation}
\begin{equation}\label{eq:stiffness-matrix-dif}
(K_{\phi,n}^p)_{i,j} = \int_0^1 \frac{1}{\left| \phi'(x) \right|} (N_i^p)'(x) (N_j^p)'(x)dx, \quad \text{for } 1 \leq i,j \leq p+n-2.
\end{equation}
The numerical eigenvalue problem in this case is described by:
\begin{equation*}
L_{\phi,n}^p \mathbf{u}_{k,h} = \lambda_{k,h} \mathbf{u}_{k,h},
\end{equation*}
where $L_{\phi,n}^p=\left( M_{\phi,n}^p\right)^{-1} K_{\phi,n}^p.$

\subsection{Preliminaries on GLT Sequences}\label{sec:preliminariesGLT}
$\hspace{0.5cm}$This subsection provides a brief overview of the essential background on the Generalized Locally Toeplitz (GLT) sequences theory. More details can be found in the pioneering works \cite{bianchi2021analysis,garoni2017generalized}, and the references therein. 
In what follows, $(L_{n})_{n\in\mathbb{N}^{*}}$ represents a sequence of matrices of   size denoted as $N=N(n)$, with $N \rightarrow +\infty$  as $n \rightarrow +\infty$.

In this subsection, all the definitions and results have been adjusted to better align with our particular context, including the following definition of the spectral symbol.
\begin{definition}[Spectral symbol]
Let  $\mathcal{C}_c(\mathbb{R})$ be the set of continuous functions with compact support over $\mathbb{R}$, and let $\omega :[0,1]\times[0,\pi] \longrightarrow \mathbb{R}$ be a measurable function. We say that $(L_{n})_{n\in\mathbb{N}^{*}}$ has a spectral (or eigenvalue) distribution described by $\omega$, and we write
$$
(L_{n})_{n\in\mathbb{N}^{*}}\sim_{\lambda} \omega,
$$
if for all $F \in \mathcal{C}_c(\mathbb{R})$ we have
$$
\lim_{N\longrightarrow + \infty}\dfrac{1}{N}\displaystyle\sum_{k=1}^N F(\lambda_k(L_n))=\dfrac{1}{\pi}\displaystyle\int_{[0,1]\times[0,\pi]} F(w(x,\theta)) \,dx\,d\theta,
$$
where $\lambda_k(L_n)$, $k=1, \cdots, N$ are the eigenvalues of $L_n$. In this case, $\omega$ is referred to as the spectral symbol of $(L_{n})_{n\in\mathbb{N}^*}$.
\end{definition}
With the chosen discretization in the Subsection \ref{sec:descritization}, whether or not reparametrization is employed, an issue arises concerning outliers. A few eigenvalues are poorly approximated, and their corresponding values are notably larger than the exact values. In scenarios where the reparametrization is considered to be the identity, it is observed in \cite{hiemstra2021removal} that the number of these outliers amounts to $2\left\lfloor\frac{p-1}{2}\right\rfloor$. The following definition gives the mathematical definition of these outliers.
\begin{definition}[Outliers]
Let $(L_{n})_{n\in\mathbb{N}^*}\sim_{\lambda} \omega$, and let $R_\omega$ represent the essential range of $\omega$, defined as 
\begin{equation}\label{eq:essential-range}
R_{\omega} = \left\{ y \in \mathbb{R}:\; \mu_{2}\left(\{(x,\theta)\in [0,1]\times[0,\pi],\;\; |\omega(x,\theta)-y|<\epsilon      \}\right)> 0, \; \forall \epsilon>0 \right\},
\end{equation}
where $\mu_{2}$ denotes the Lebesgue measure on $\mathbb{R}^{2}$. Let $\lambda_k(L_n)$, $k=1, \cdots, N$, be the eigenvalues of $L_n$. An eigenvalue $\lambda_k(L_n)$ is considered an outlier if $\lambda_k(L_n) \notin  R_{\omega}$.
\end{definition}
It can be proved that  $ R_{\omega}$ is a closed set [ Lemma 2.1, \cite{garoni2017generalized}]. Furthermore,  if the function $\omega$ is continuous and because the domain $[0,1]\times[0,\pi]$ is compact, we can demonstrate that the essential range of $\omega$ coincides with the image of $\omega$.\\

In general, and specifically in our case, the symbol is defined on a multidimensional domain, which complicates the mathematical and numerical study of the distribution of the eigenvalues. However, we can drive a new one-dimensional symbol from the original symbol, as explained in the following
\begin{definition}[Monotone rearrangement of the symbol]\label{def:monotone-rearrangement}
Let $\omega: [0,1] \times [0,\pi] \longrightarrow \mathbb{R}$ be a measurable function such that $(L_n)_{n\in\mathbb{N}^*} \sim_{\lambda} \omega$. We assume that the essential range of $\omega$ is bounded. The extension function $\xi: [0,1] \longrightarrow R_{\omega}$ defined by
$$
\xi(x)=\inf\left\{ y\in  R_{\omega}:\;\; \Psi(y)>\pi  x  \right\},
\quad \forall x\in(0,1),$$
where 
\begin{equation}\label{eq:function-psi}
\Psi(y)=\mu_{2}(\{  (x,\theta)\in [0,1]\times[0,\pi],\;\; \omega(x,\theta)\leq y    \}),\quad \forall y\in \mathbb{R},
\end{equation}
is called the monotone rearrangement of $\omega$.
\end{definition}

The following result \cite{bianchi2021analysis} shows that the monotone rearrangement of a symbol remains a spectral symbol for the same sequence of matrices.

\begin{proposition}
Let $(L_n)_{n\in\mathbb{N}^*} \sim_{\lambda} \omega$ with $\omega: [0,1] \times [0,\pi] \longrightarrow \mathbb{R}$ has a bounded essential range. Let $\xi$ be the monotone rearrangement of $\omega$. Then, we have
$$ 
(L_{n})_{n\in\mathbb{N}^{*}}\sim_{\lambda}\xi.
$$
\end{proposition}

We conclude this subsection by introducing discrete Weyl's law and some of its consequences. These results that describe the asymptotic behavior of eigenvalues are essential for our IgA spectral analysis given in the next section.
 
\begin{theorem} [Discrete Weyl’s law , \cite{bianchi2021analysis}]\label{th1}
Let $(L_n)_{n\in\mathbb{N}^*} \sim_{\lambda} \omega$ with $\omega: [0,1] \times [0,\pi] \longrightarrow \mathbb{R}$ has a bounded essential range. Define $\Psi:\, \mathbb{R} \longrightarrow \mathbb{R}_+$ as the function given by \eqref{eq:function-psi}. Then, at every point of continuity $y$ of $\Psi$, the eigenvalues of the matrices $L_n$ satisfy 
$$
\lim_{n\rightarrow  +\infty} \dfrac{\left| \left\{k=1,\dots, N:\; \lambda_{k}(L_n)\leq y\right\} \right|}{N}=\frac{1}{\pi}\Psi(y),
$$
where, for a generic set $A$, $|A|$ denotes the number of elements in set $A$.\\
Furthermore,  if we assume that $\Psi$ and $\xi$ are continuous. Then for every sequence $k=k(n)\in\{1,\dots,N\}$ such that $\lim_{n\rightarrow  +\infty} \frac{k(N)}{N}=x\in[0,1]$ and $(\lambda_{k(n)}(L_n))_n\subset  R_w$, we have  
$$ 
\left(\frac{k(n)}{N+1}, \, \lambda_{k(n)}(L_n)  \right)\longrightarrow (x,\, \xi(x)), \; \text{ when } n \longrightarrow +\infty.
$$
\end{theorem}
From the discrete Weyl's law,  we can deduce the following result, which demonstrates that the possible number of outliers is very small concerning $N$, specifically of the order $o(N)$.
\begin{corollary} 
Under the hypotheses of Theorem \ref{th1}, if $\Psi$ is continuous, then :
$$
\lim_{n\rightarrow  +\infty}\dfrac{\left| \left\{k=1,\dots, N:\; \lambda_{k}(L_n)\notin R_{\omega} \right\}\right|}{N}=0.
$$
Moreover, for all $t\in R_{\omega}$ we have 
$$
\lim_{n\rightarrow  +\infty}\dfrac{\left| \left\{k=1,\dots, N:\; \lambda_{k}(L_n)\leq t,\;\lambda_{k}(L_n)\in R_{\omega} \right\}\right|}{N}=\frac{1}{\pi}\Psi(t).
$$
\end{corollary}

In the outlier-free context, we can approximate all the eigenvalues of matrices $L_n$ through uniform sampling of the monotone rearrangement, as shown in the following result \cite{bianchi2021analysis}.

\begin{corollary}\label{cl1}
Under the hypotheses of Theorem \ref{th1}, and assuming additionally that $\Psi$ and $\xi$ are  continuous, and that $R_{\omega}$ is bounded. Then in the presence of no outliers, the error between the uniform sampling of $\xi$ and the eigenvalues of $L_{n}$, tends to $0$ when $n$ tends to infinity, namely
$$
\lim_{ n \rightarrow + \infty} \sup_{1 \leq k \leq N}\left\{ \left| \lambda_{k}(L_n) - \xi\left(\frac{k}{N+1}\right)\right|  \right\} = 0,  
$$
where $\xi$ represents the monotone rearrangement of $\omega$.
\end{corollary}

\begin{remark}
In both  Theorem \ref{th1} and Corollary \ref{cl1}, an essential assumption is the absence of outliers.  Nevertheless, as illustrated in \cite[Chapter 5.1.2, p. 153]{cottrell2009isogeometric}, there is substantial numerical evidence pointing to the existence of outliers in Isogeometric Analysis (IgA) when B-splines of degree $p$ are employed. Specifically, it has been observed that the number of outliers is solely dependent on $p$ and does not vary with the discretization step $h=1/n$. 

\end{remark}

In the subsequent sections of the paper, we will employ the notation $OUT(p, n)$ to represent the number of outliers, and $\mathcal{I}(p, n) = \left \{1, \cdots, N - OUT(p,n) \right \}$ to denote the set of indices of eigenvalues after removing the outliers.

\subsection{The IgA GLT-Symbol}\label{symbol}
$\hspace{0.52cm}$Before presenting the theorem that provides the spectral symbol of the matrix $L^{p}_{\phi,n}$, it is essential to introduce what we call the cardinal B-spline \cite{de1978practical}. Let $\mathcal{N}_p: \mathbb{R}\longrightarrow\mathbb{R} $ be the cardinal B-spline of degree $p$ recursively defined as follows

\begin{equation}\label{eq:cardinal-B-spline-1}
\mathcal{N}_0(x)=\left\{
\begin{array}{ll}
1,\;\;x\in[0,1],\\
0,\;\;\text{otherwise},
\end{array}
\right.
\vspace{-0.25cm}  
\end{equation}
and
\begin{equation}\label{eq:cardinal-B-spline-2}
\mathcal{N}_p(x)=\frac{x}{p}\mathcal{N}_{p-1}(x)+\frac{p+1-x}{p}\mathcal{N}_{p-1}(x-1),\;\;x\in\mathbb{R},\;\;p\geq 1.
\end{equation}
It was proved in\cite{de1978practical} ( see also \cite{chui1992introduction}) that $\mathcal{N}_p\in C^{p-1}(\mathbb{R})$ and 
$$supp(\mathcal{N}_p)=[0,p+1].$$

For $p\in\mathbb{N}$, let 
$$f_p:[0,\pi]\longrightarrow\mathbb{R}, \quad f_p(\theta)=-\mathcal{N}^{''}_{2p+1}(p+1)-2\displaystyle\sum_{k=1}^p\mathcal{N}^{''}_{2p+1}(p+1-k)\cos(k\theta),\;\;p\geq 1,$$
$$g_p:[0,\pi]\longrightarrow\mathbb{R}, \quad g_p(\theta)=\mathcal{N}_{2p+1}(p+1)+2\displaystyle\sum_{k=1}^p\mathcal{N}_{2p+1}(p+1-k)\cos(k\theta),\;\;p\geq 0,$$
$$e_p:[0,\pi]\longrightarrow\mathbb{R},\quad e_p(\theta)=\dfrac{f_p(\theta)}{g_p(\theta)},\;\;p\geq 1.$$
It is known from \cite{garoni2014spectrum} that
$$f_p(\theta)=(2-2\cos(\theta))g_{p-1}(\theta),\;\;\theta\in[0,\pi],\;\; p\geq 1.$$
$$\left( \frac{4}{\pi^2}\right)^{p+1}\leq g_p(\theta)\leq g_p(0)=1,\;\; \theta\in[0,\pi],\;\; p\geq 0.$$
We can observe also from \cite{donatelli2016spectral} ( see also \cite{bianchi2018spectral}) that for every $p\geq 2$, it holds that
\begin{equation}\label{eq:preliminaries-1}
e_p(\theta)=(2-2\cos(\theta))\frac{g_{p-2}(\theta)}{g_{p}(\theta)},\;\; \left( \frac{2}{\pi} \right)^{p-1}\leq\frac{g_{p-2}(\theta)}{g_{p}(\theta)}\leq \left( \frac{\pi}{2} \right)^{p+1},\;\; \theta\in[0,\pi].
\end{equation}
$\hspace{0.52cm}$Using the inner product Lemma on the cardinal B-splines, it has been demonstrated that both $M_{\phi,n}^p$ and $K_{\phi,n}^p$ are small rank perturbations of Toeplitz matrices. A complete proof of this result can be found in \cite{garoni2017generalized}. As a result of the GLT theory, we obtain the following theorem:
\begin{theorem}[\cite{garoni2017generalized}, IgA GLT symbol]
Let $p\geq 1$ and $\phi\in \mathbf{C}_{[0,1]}$. Then $$n^{-2}L^{p}_{\phi,n}\sim_{\lambda}\omega^p_{\phi},$$
where  
\begin{equation}\label{eq:spectralsymbo-full-matrix}
\omega^p_{\phi}(x,\theta)=\dfrac{e_p(\theta)}{\left(\phi^{'}(x)\right)^2},\quad\forall (x,\theta)\in [0,1]\times[0,\pi].
\end{equation}

\end{theorem}
The following corollary illustrates the regularity and some properties of the symbol   $e_p$ (see \cite{bianchi2021analysis,everitt1999boundary}).
\begin{corollary} 
Let $p\geq 1$. The function $e_p$ is differentiable, nonnegative, and monotonically increasing on the interval $[0,\pi]$. Additionally, it satisfies the following properties:
$$
e_p(\theta)\sim_{0^{+}} \theta^2\;\; \text{as},\quad \lim_{p\rightarrow +\infty }\sup_{\theta\in[0,\pi]}|e_p(\theta)-\theta^2|=0.
$$
\end{corollary}

In what follows, our focus will be on the spectral analysis of the matrix $L^{p}_{\phi,n}$. To simplify the notation, we refer to the size of  $L^{p}_{\phi,n}$ by $N$, such that $N=n+p-2$.
\section{IgA  Eigenvalues Distribution Analysis}\label{disIGA}
$\hspace{0.52cm}$In this section,  our objective is to apply the Generalized Locally Toeplitz (GLT) theory to investigate the behavior of the eigenvalues resulting from the discretization  introduced in Subsection \ref{sec:descritization}. 
We will demonstrate that by employing a reparametrization $\phi\in\mathbf{C}_{[0,1]}$, we can improve upon the results provided by Bianchi in \cite{bianchi2021analysis} and presented with alight to our context in the previous Subsection  \ref{sec:preliminariesGLT}, particularly concerning the Discrete Weyl's law. Additionally, we introduce a  novel approach for analyzing the behavior of the eigenvalues.

We proved that it is possible to build an order among the eigenvalues associated with different reparametrizations, and we illustrate that valuable and unknown information on the distribution of the eigenvalues can be derived through the use of the GLT theory, specifically when the average gap condition is satisfied. \\ 

Before delving into the study of the gap between successive square roots of eigenvalues, it is important to understand  how the distribution of eigenvalues is influenced by the chosen reparameterization. With this objective in mind, we introduce the following notations.\\
For all $\phi\in \mathbf{C}_{[0,1]}$, we note 
$$
\Psi_{\phi}^{\sqrt{\;}}(y)=\mu_{2} \left( \left\{  (x,\theta)\in [0,1]\times[0,\pi]:\;\; \sqrt{\omega_{\phi}^p(x,\theta)}\leq y    \right\} \right),\;\;\forall y\in R_{\sqrt{\omega_{\phi}^p}}.
$$
With this notation, we have 
$$
\sqrt{\xi_{\phi}^p(x)}=\inf\left\{ y\in  R_{\sqrt{\omega_{\phi}^p}}:\;\; \Psi_{\phi}^{\sqrt{\;}}(y)>\pi  x  \right\},\;\;\forall x\in(0,1).
$$
where $\xi_{\phi}^p$ denotes the monotone rearrangement of $\omega_\phi^p$.  Note that in our case the  $\sqrt{\omega_{\phi}^p}$ is a continuous function, which implies that the essential range of $\sqrt{\omega_{\phi}^p}$ coincides with its regular range. Furthermore, it is apparent in this scenario that $\Psi_{\phi}^{\sqrt{\;}}$ is strictly increasing over $Rg\left(\sqrt{\omega_{\phi}^p}\right)$.

We make use of the following result, which provides sufficient conditions on the reparametrization to satisfy the average gap condition. For the proof see \cite{bianchi2018spectral}.

\begin{proposition}\label{stefa}
For all  $\phi\in \mathbf{C}_{[0,1]}$, we have 
$$
\Psi_{\phi}^{\sqrt{\;}}\in \mathcal{C}^1((0,+\infty[)\;\;\text{and} \;\; \sup_{y \geq 0} \left(\Psi^{\sqrt{\;}}_{\phi}(y)\right)'<\infty.
$$ 

\end{proposition}

The following theorem provides insight into the distribution of eigenvalue families. It demonstrates that no family of eigenvalues can be obtained from two different reparametrizations.

\begin{theorem}[Distribution of eigenvalues]\label{tdis1}
Let $\phi_1, \phi_2 \in \mathbf{C}_{[0,1]}$, and let $I$ be a closed interval in  $Rg(\omega_{\phi_1}^p) \cap Rg(\omega_{\phi_2}^p)$. If, for all $y \in I$, it holds
$$ 
\forall y\in I, \quad \Psi_{\phi_1}^{\sqrt{\;}}(\sqrt{y})>\Psi_{\phi_2}^{\sqrt{\;}}(\sqrt{y}).
$$
Then, there exists $  n_0\in\mathbb{N}^{*}$,  such that for all $n\geq n_0$  and $k\in \mathbb{N}^{*}$ we have: if $n^{-2}\lambda_{k,h}^{\phi_1},n^{-2}\lambda_{k,h}^{\phi_2}\in I$, then
$$ 
n^{-2}\lambda_{k,h}^{\phi_1}<n^{-2}\lambda_{k,h}^{\phi_2}.
$$ 
\end{theorem}

\begin{proof}
We will begin by proving the following lemma:
\begin{lemma}\label{theo:lemma-ditribution-eig-1}
For all $n \in \mathbb{N}^*$ and $k \in \mathcal{I}(n,p)$, there exist unique $z_{k,h}^{\phi^1}$ and $z_{k,h}^{\phi^2}$ in the interval $[0, \pi]$ such that:
$$\sqrt{n^{-2}\lambda_{k,h}^{\phi_1}}= \left(\Psi^{\sqrt{\;}}_{\phi_1}\right)^{-1}(z_{k,h}^{\phi_1}),\; \sqrt{n^{-2}\lambda_{k,h}^{\phi_1}}= \left(\Psi^{\sqrt{\;}}_{\phi_1}\right)^{-1}(z_{k,h}^{\phi_1})$$
and 
$$
\lim_{n\rightarrow +\infty}\; \max_{k \in \mathcal{I}(p,n)} \left| z_{k,h}^{\phi_1}-   z_{k,h}^{\phi_2}  \right|=0.
$$
\end{lemma}

\begin{proof} (of Lemma \ref{theo:lemma-ditribution-eig-1}).
Let $n \in \mathbb{N}^*$ and $k \in \mathcal{I}(n,p)$. Since $n^{-2}\lambda_{k,h}^{\phi_1} \in Rg (\omega_{\phi_1}^p )$, $n^{-2}\lambda_{k,h}^{\phi_2} \in Rg (\omega_{\phi_2}^p )$ and for $i=1,2$, $\left(\Psi^{\sqrt{\;}}_{\phi_i}\right)^{-1}$ is strictly increasing , there exist unique $z_{k,h}^{\phi^1}$ and $z_{k,h}^{\phi^2}$ in the interval $[0, \pi]$ such that:
$$
\sqrt{n^{-2}\lambda_{k,h}^{\phi_1}}= \left(\Psi^{\sqrt{\;}}_{\phi_1}\right)^{-1}(z_{k,h}^{\phi_1}),
$$
and 
$$
\sqrt{n^{-2}\lambda_{k,h}^{\phi_2}}= \left(\Psi^{\sqrt{\;}}_{\phi_2}\right)^{-1}(z_{k,h}^{\phi_2}).
$$
Owing to Discrete Weyl's law \ref{cl1} together with Proposition \ref{stefa} and the fact that 
$$
Rg\left(\sqrt{\omega_{\phi_i}^p}\right)=\left[0,\max\left\{ \frac{\sqrt{e_p(\pi)}}{\phi'_i(0)},\frac{\sqrt{e_p(\pi)}}{\phi'_i(1)}\right\}\right], \quad i \in \{1, 2 \},
$$
we can immediately conclude
\begin{equation}\label{eq:lemma-ditribution-eig-1}
\lim_{n \rightarrow +\infty}\;\max_{k \in \mathcal{I}(p,n)}\left|    \sqrt{n^{-2}\lambda_{k,h}^{\phi_i}}   -  \left(\Psi^{\sqrt{\;}}_{\phi_i}\right)^{-1}\left(\frac{k\pi}{N+1}\right)\right|= 0,\quad i \in \{1,2 \}.
\end{equation}
Using the mean value theorem and Proposition \ref{stefa}, we can further deduce
$$
\max_{k \in \mathcal{I}(p,n)}\left|  z_{k,h}^{\phi_i} -\frac{k\pi}{N+1} \right |\leq \left| \sup_{y \geq 0} \left(\Psi^{\sqrt{\;}}_{\phi_i}\right)'(y)\right| \max_{k \in \mathcal{I}(p,n)}\left|  \sqrt{n^{-2}\lambda_{k,h}^{\phi_i}}   -  \left(\Psi^{\sqrt{\;}}_{\phi_i}\right)^{-1}\left(\frac{k\pi}{N+1}\right) \right |, \quad i \in \{1,2 \},
$$
which concludes the proof of Lemma \ref{theo:lemma-ditribution-eig-1} using \eqref{eq:lemma-ditribution-eig-1}.
\end{proof}
Now, we can proceed to the proof of the theorem. Let $(z_{k,h}^{\phi^1})$ and $(z_{k,h}^{\phi^2})$ be the two sequences constructed in Lemma \ref{theo:lemma-ditribution-eig-1}, and let $\varepsilon = \min_{y\in I} \left( \Psi_{\phi_1}^{\sqrt{\;}}(\sqrt{y})- \Psi_{\phi_2}^{\sqrt{\;}}(\sqrt{y})\right)$. Note that the continuity of $y \mapsto \Psi_{\phi_i}^{\sqrt{\;}}(y)$, $i=1,2$, ensures that $\varepsilon > 0$, hence there exists $n^0 \in \mathbb{N}^*$, such that:
\begin{equation}\label{eq:lemma-ditribution-eig-2}
\max_{k \in \mathcal{I}(p,n)} \left| z_{k,h}^{\phi_1}-   z_{k,h}^{\phi_2}  \right|< \frac{\varepsilon}{2}, \quad \forall n \geq n^0.
\end{equation}
Now, let $n \geq n^0$ and $k \in \mathbb{N}^*$ such that $n^{-2}\lambda_{k,h}^{\phi_1},n^{-2}\lambda_{k,h}^{\phi_2}\in I$. To conclude the proof, it is sufficient to prove that
$$
n^{-2}\lambda_{k,h}^{\phi_1}<n^{-2}\lambda_{k,h}^{\phi_2}.
$$
In fact, if we suppose the contrary $\left(n^{-2}\lambda_{k,h}^{\phi_1}\geq n^{-2}\lambda_{k,h}^{\phi_2}\right)$, we obtain
$$ 
\left| z_{k,h}^{\phi_1}-   z_{k,h}^{\phi_2}  \right|= \left| \Psi_{\phi_1}^{\sqrt{\;}}\left(\sqrt{n^{-2}\lambda_{k,h}^{\phi_1} }\right)-\Psi_{\phi_2}^{\sqrt{\;}}\left(\sqrt{n^{-2}\lambda_{k,h}^{\phi_2}} \right)\right|<\frac{\varepsilon}{2},
$$
and
$$
\Psi_{\phi_1}^{\sqrt{\;}}\left(\sqrt{n^{-2}\lambda_{k,h}^{\phi_1}} \right)\geq \Psi_{\phi_1}^{\sqrt{\;}}\left(\sqrt{n^{-2}\lambda_{k,h}^{\phi_2}} \right)\geq \Psi_{\phi_2}^{\sqrt{\;}}\left(\sqrt{n^{-2}\lambda_{k,h}^{\phi_2}} \right), \quad \Psi_{\phi_2}^{\sqrt{\;}}\left(\sqrt{n^{-2}\lambda_{k,h}^{\phi_2}} \right)\leq \Psi_{\phi_2}^{\sqrt{\;}}\left(\sqrt{n^{-2}\lambda_{k,h}^{\phi_1}} \right).
$$
This yields
$$
\left| z_{k,h}^{\phi_1}-   z_{k,h}^{\phi_2}  \right|\geq \Psi_{\phi_1}^{\sqrt{\;}}\left(\sqrt{n^{-2}\lambda_{k,h}^{\phi_1}} \right)-\Psi_{\phi_2}^{\sqrt{\;}}\left(\sqrt{n^{-2}\lambda_{k,h}^{\phi_1}} \right)\geq  \varepsilon, 
$$
which contradicts \eqref{eq:lemma-ditribution-eig-2}. This concludes the proof.
\end{proof}

The previous Theorem \ref{tdis1} established a crucial relationship: it demonstrated that the ordering of eigenvalues is precisely the inverse of the ordering of the functions $\Psi_{\phi_1}^{\sqrt{\;}}$ and $\Psi_{\phi_2}^{\sqrt{\;}}$. However, it's important to note that accessing and manipulating the functions $\Psi_{\phi_1}^{\sqrt{\;}} $ and $\Psi_{\phi_2}^{\sqrt{\;}} $ can be complex and challenging in practice. Therefore, there is a need for a more general relationship that connects the order of reparametrizations to the order of the associated eigenvalues.

The following theorem provides a solution to this problem. It demonstrates that it is indeed possible to exert control over the distribution of eigenvalues through reparametrization, offering a valuable tool for managing the behavior of eigenvalues in a more tractable manner.

\begin{theorem} \label{dis1}
Let $\phi_1, \phi_2 \in \mathbf{C}_{[0,1]}$ that are strictly convex, and  $\phi_1'(0) = \phi_2'(0)$. Suppose that the first zero, denoted as $x_0 \in (0, 1)$, of the function $\phi_1' - \phi_2'$ is not an accumulation point. Additionally, assume that $\phi_1'(x) \geq \phi_2'(x)$ for all $x \in [0, x_0]$.

Then, for any $y_0$ and $y_1$ in the open interval $\left(\frac{e_p(\pi)}{(\phi_1'(x_0))^2}, \frac{e_p(\pi)}{(\phi_1'(0))^2}\right)$ with $y_0 < y_1$, it holds
$$ 
\Psi_{\phi_1}^{\sqrt{\;}}(\sqrt{y}) > \Psi_{\phi_2}^{\sqrt{\;}}(\sqrt{y}), \quad \forall y \in [y_0, y_1].
$$
\end{theorem}

\begin{proof}
First, we observe that
$$
Rg(\omega^p_{\phi_1})=Rg(\omega^p_{\phi_2})=\left[0,\frac{e_p(\pi)}{(\phi_1^{'}(0))^2}\right],
$$
which implies that the restrictions of $\Psi_{\phi_1}^{\sqrt{\;}}$ and $\Psi_{\phi_2}^{\sqrt{\;}}$ on the interval $\left(\frac{\sqrt{e_p(\pi)}}{\phi_1'(x_0)}, \frac{\sqrt{e_p(\pi)}}{\phi_1'(0)}\right)$ are well-defined. On the other hand, Roll's theorem ensures the existence of the zero $x_0$ in the open interval $(0,1)$ since $\phi_1(0) - \phi_2(0) = \phi_1(1) - \phi_2(1)$.

Now, let's assume that there exists $y\in [y_0, y_1]$ such that $\Psi_{\phi_1}^{\sqrt{\;}}(\sqrt{y})\leq \Psi_{\phi_2}^{\sqrt{\;}}(\sqrt{y})$. Following an argument similar to that in the reference \cite{bianchi2018spectral}, we can derive the equality
\begin{equation}\label{eq:therem-dist1-eq1}  
\begin{split}
\pi&-\displaystyle\int_{S_1(y)} \left(\phi_1^{'}\right)^{-1}\left(\sqrt{\dfrac{e_p(\theta)}{y}} \right)\;d\theta -\mu_{1}(A_1)
\\
&\leq\pi-\displaystyle\int_{S_2(y)}\left(\phi_2^{'}\right)^{-1}\left(\sqrt{\dfrac{e_p(\theta)}{y}} \right)\;d\theta -\mu_{1}(A_2),
\end{split}
\end{equation}
where
$$
S_i(y)=\left\{ \theta\in[0,\pi]:\quad  \phi_i'(0)\leq \sqrt{\dfrac{e_p(\theta)}{y}} \leq\phi_i'(1)    \right\},
$$
and
$$
A_i(y)=\left\{ \theta\in[0,\pi]:\;\;    \sqrt{\dfrac{e_p(\theta)}{y}}> \phi_i'(1)   \right\},\quad i \in \{1,2\}.
$$

We  claim that
\begin{equation}\label{eq:therem-dist1-eq2}  
S_1(y) = S_2(y) = \left[    e_p^{-1}\left( y\left( \phi_1^{'}(0)   \right)^2 \right),\pi   \right], \quad \text{and } A_1=A_2=\emptyset.
\end{equation}
Indeed, considering that $y\in \left(\frac{e_p(\pi)}{(\phi_1^{'}(x_0))^2},\frac{e_p(\pi)}{(\phi_1^{'}(0))^2}\right)$ together with $\phi_1'(0)=\phi_2'(0)$ and $\phi_1'(x_0)=\phi_2'(x_0)$, we conclude
$$
S_1(y) = S_2(y) = \left\{ \theta\in[0,\pi]:\quad  \phi_1^{'}(0)\leq \sqrt{\dfrac{e_p(\theta)}{y}} \leq\phi_1^{'}(x_0) \right\},
$$
and 
$$
A_1=A_2=\emptyset.
$$
Moreover, since $y(\phi_1'(x_0))^2 > e_p(\pi)$, it is straightforward to see that
\begin{align*}
S_1(y) & = \left\{ \theta\in [0,\pi]: \quad  \phi_1'(0) \leq \sqrt{\dfrac{e_p(\theta)}{y}} \leq\phi_1'(x_0) \right\}\\
& = \left\{ \theta\in [0,\pi]: \quad y\left( \phi_1'(0)   \right)^2 \leq e_p(\theta) \leq y\left( \phi_1'(x_0)   \right)^2 \right\}\\
& = \left[    e_p^{-1}\left( y\left( \phi_1^{'}(0)   \right)^2 \right),\pi   \right],
\end{align*} 
which prove the claim \eqref{eq:therem-dist1-eq2}. In particular, since $0 < y(\phi_1'(0))^2 < e_p(\pi)$, we have $\mu_1\left(S_1(y)\right)>0$, and \eqref{eq:therem-dist1-eq1} together with the fact that $\phi_1'(x) \geq \phi_2'(x)$ for all $x \in [0, x_0]$, yield
\begin{equation}\label{eq:therem-dist1-eq3}
\left(\phi_1'\right)^{-1}\left(\sqrt{\dfrac{e_p(\theta)}{y}} \right) =\left(\phi_2'\right)^{-1}\left(\sqrt{\dfrac{e_p(\theta)}{y}} \right), \quad \forall \theta \in S_1(y).
\end{equation}

To end the proof, it is enough to demonstrate that \eqref{eq:therem-dist1-eq3} contradicts the fact that $x_0 \in (0,1)$ is the first zero of the function $\phi_1'-\phi_2'$. In fact, by the intermediate-value theorem and  \eqref{eq:therem-dist1-eq2}, one can find $x_1 \in (0,x_0)$ and $\theta_0 \in S_1(y)$ such that
$$
\phi_1'(x_1) = \sqrt{\frac{e_p(\theta_0)}{y}},
$$
and, using \eqref{eq:therem-dist1-eq3}, we obtain 
$$
x_1=\left(\phi_1'\right)^{-1}\left( \sqrt{\frac{e_p(\theta_0)}{y}}\right) = \left(\phi_2'\right)^{-1}\left( \sqrt{\frac{e_p(\theta_0)}{y}}\right) = \left(\phi_2'\right)^{-1}\left(\phi_1'(x_1)\right).
$$
Consequently, $\phi_2'(x_1) = \phi_1'(x_1)$, signifying that $x_1$ is a zero for the function $\phi_1'-\phi_2'$. This concludes the proof as $x_1<x_0$.
\end{proof}

 Similarly, the next theorem addresses strictly concave reparametrizations. The proof follows a similar approach to that of Theorem \ref{dis1} and is thus omitted.

\begin{theorem}\label{dis2}
Let $\phi_1, \phi_2 \in \mathbf{C}_{[0,1]}$ that are strictly concave, and  $\phi_1'(1) = \phi_2'(1)$. Suppose that the last zero $x_0 \in (0, 1)$ of the function $\phi_1' - \phi_2'$ is not an accumulation point. Additionally, assume that $\phi_1'(x) \geq \phi_2'(x)$ for all $x \in [x_0, 1]$.

Then, for any $y_0$ and $y_1$ in the open interval $\left(\frac{e_p(\pi)}{(\phi_1'(x_0))^2}, \frac{e_p(\pi)}{(\phi_1'(1))^2}\right)$ with $y_0 < y_1$, it holds
$$ 
\Psi_{\phi_1}^{\sqrt{\;}}(\sqrt{y}) > \Psi_{\phi_2}^{\sqrt{\;}}(\sqrt{y}), \quad \forall y \in [y_0, y_1].
$$
\end{theorem}

A family of examples of such convex reparametrization functions is 
\begin{equation}\label{eq:reparametrization-convex-example}
\phi_{a,b}=e^{ax+b}-e^b+(\gamma-ae^b)x,
\end{equation}
where $a >0$, $0<\gamma<1$ and $b$ is given by 
$$
b=-\ln\left(\frac{e^a-(a+1)}{1-\gamma}\right).
$$
Notably, these reparametrization functions satisfy the conditions specified in Theorem \ref{tdis1}, including the property $\phi_{a,b}'(0) = \gamma$.  

In the case of strictly concave reparametrization, the family \eqref{eq:reparametrization-convex-example} is replaced by
\begin{equation}\label{eq:reparametrization-concave-example}
\phi_{a,b}=\ln(ax+b)-\ln(b)+\left(\gamma-\frac{a}{a+b}\right)x.
\end{equation}
Here, the parameters are similar: $a > 0$, $0 < \gamma < 1$. However, the value of $b$ is calculated as $b:=a/x^*$, where $x^* \in (0,1)$ represents the unique solution to the equation
$$
\gamma=1-\left(\ln(x^*+1)-\frac{x^*}{x^*+1}\right).
$$ 
Once again, these reparametrization functions satisfy the hypotheses of Theorem \ref{dis2}, and they specifically verify $\phi_{a,b}'(1)=\gamma$.

At this stage, we have demonstrated the feasibility of ordering the families of eigenvalues based on the corresponding reparametrizations order. In the following proposition, our focus will be on the distribution of packed eigenvalues and how the convexity of the symbol influences the behavior of these eigenvalues. More precisely,  we will show that the number of eigenvalues can indeed be ordered, and this order depends  on the convexity of the symbol. In the case of a strictly convex symbol, the spectrum shifts to the right, whereas a strictly concave symbol results in a leftward shift.

It is important to mention that this proposition does not necessitate the assumption of the average gap condition, and the results are presented within the context of a reparametrization in $\mathbf{C}_{[0,1]}$ to align with the scope of our paper.

\begin{proposition}[Distribution of pack-eigenvalues]\label{cl2.5}
Let $p \geq 1$ and $\phi \in \mathbf{C}_{[0,1]}$. Suppose there exists a closed interval $I$ in which $\left( \sqrt{\xi_{\phi}^p} \right)^{-1} $ is of class $\mathcal{C}^2(I)$. Let $\{y_0, y_1, \ldots, y_r\}$ be a uniform discretization of $I$. For sufficiently large $n$, the following hold
\begin{itemize}
\item[1.] If $\left( \sqrt{\xi_{\phi}^p} \right)^{-1}$ is strictly convex on $I$, then for all $i \in \{1, 2, \ldots, r-1\}$:
\begin{equation*}
\left|\left\{ k=1,\dots,N: \quad y_{i-1}< \sqrt{n^{-2}  \lambda_{k,h}}\leq y_{i}     \right\}\right|< \left|\left\{ k=1,\dots,N: \quad y_{i}< \sqrt{n^{-2}\lambda_{k,h}}\leq y_{i+1}     \right\}\right|.
\end{equation*}

\item[2.] If $\left( \sqrt{\xi_{\phi}^p} \right)^{-1}$ is strictly concave on $I$, then for all $i \in \{1, 2, \ldots, r-1\}$:
\begin{equation*}
\left|\left\{ k=1,\dots,N: \quad y_{i-1}< \sqrt{n^{-2}\lambda_{k,h}}\leq y_{i}     \right\}\right|> \left|\left\{ k=1,\dots,N: \quad y_{i}<\sqrt{ n^{-2}\lambda_{k,h}}\leq y_{i+1}     \right\}\right|.
\end{equation*}
\end{itemize}
\end{proposition}

\begin{proof}
We will focus on proving assertion 1, as the proof of assertion 2. follows a similar argument. For all $i \in \{1, \cdots, r\}$, we apply Discrete Weyl's law \ref{th1} to obtain
\begin{equation} \label{eq:propo-distribution-pack-eigenvalues-1}
\lim_{n\rightarrow  +\infty}\dfrac{\left|\left\{k=1,\dots, N:\quad y_{i-1}<\sqrt{n^{-2} \lambda_{k,h}}\leq y_{i} \right\}\right|}{n}=\left( \sqrt{\xi_{\phi}^p}\right)^{-1}(y_{i})-\left( \sqrt{\xi_{\phi}^p}\right)^{-1}(y_{i-1}),
\end{equation}
and by the mean value theorem, there exist $x_i \in (y_{i-1}, y_i)$ such that
\begin{equation} \label{eq:propo-distribution-pack-eigenvalues-2}
\left( \sqrt{\xi_{\phi}^p}\right)^{-1}(y_{i})-\left( \sqrt{\xi_{\phi}^p}\right)^{-1}(y_{i-1}) = (y_{i}-y_{i-1}) \left[\left(  \sqrt{\xi_{\phi}^p} \right)^{-1}\right]'(x_i) = \frac{y_r-y_0}{r}\left[\left(  \sqrt{\xi_{\phi}^p} \right)^{-1}\right]'(x_i).
\end{equation}
Now, for $i \in \{1, \cdots, r-1\}$, let's fix $\varepsilon$ such that
$$
0 < \varepsilon < \frac{(y_r-y_0) \delta}{2r} \inf_{I} \left[\left( \sqrt{\xi_{\phi}^p}  \right)^{-1}   \right]'', \quad \text{with } \delta:= \inf_{1 \leq i \leq r-1} (x_{i+1} -x_{i} ). 
$$
Using equations \eqref{eq:propo-distribution-pack-eigenvalues-1} and \eqref{eq:propo-distribution-pack-eigenvalues-2}, we deduce that for sufficiently large $n$:
\begin{align*}
& \dfrac{1}{n}\left(\left|\left\{ y_{i-1}<\sqrt{n^{-2} \lambda_{k,h}}\leq y_i       \right\}\right| -\left|\left\{ y_i<\sqrt{n^{-2}  \lambda_{k,h}}\leq y_{i+1}\right\}\right|       \right) \\
&\leq 2\epsilon +\left[  \left( \sqrt{\xi_{\phi}^p}  \right)^{-1}(y_{i})-\left( \sqrt{\xi_{\phi}^p}  \right)^{-1}(y_{i-1})  \right] -  \left[ \left( \sqrt{\xi_{\phi}^p}  \right)^{-1}(y_{i+1})-(\sqrt{\xi_{\phi}^p})^{-1}(y_{i})\right]\\
&=2\epsilon+ \frac{y_r-y_0}{r}\left(   \left[\left( \sqrt{\xi_{\phi}^p}  \right)^{-1}\right]^{'}(x_i) -\left[\left( \sqrt{\xi_{\phi}^p} \right)^{-1}\right]^{'}(x_{i+1})         \right)\\
&\leq 2\epsilon -\frac{y_r-y_0}{r}\inf_{I} \left[\left( \sqrt{\xi_{\phi}^p}  \right)^{-1}\right]'' \;(x_{i+1}-x_i)\\
&\leq 2\epsilon -\frac{(y_r-y_0) \delta}{r} \inf_{I} \left[\left( \sqrt{\xi_{\phi}^p}  \right)^{-1}\right]'' <0.
\end{align*}
This concludes the proof.
\end{proof}

In all the previous findings within this section,  we have solely used the classical Discrete Weyl's law \ref{th1} to analyze the impact of reparametrization on eigenvalues when the average gap condition is satisfied. In the following result, we will prove that the  simple convergence in \ref{th1} becomes uniform when we select a reparametrization within $\mathbf{C}_{[0,1]}$.

\begin{proposition}[Uniform Discrete Weyl’s law]\label{gedis}
Let $\phi\in\mathbf{C}_{[0,1]}$, such that  $n^{-2}L_{\phi,n}^p\sim_{\lambda} \xi^p_{\phi}$. Then the sequence of  functions 
$$
G_n^p(y)= \dfrac{\left| \left\{k=1,\dots, N:\quad \sqrt{n^{-2}\lambda_{k,h}}\leq y\right\}\right|}{N+1},\quad \forall y\in Rg\left(\sqrt{\omega_\phi^p}\right),
$$
converges uniformly, as $n \rightarrow +\infty$, to $\left(\sqrt{\xi_{\phi}^p}\right)^{-1}$. In addition, it holds
\begin{equation}\label{eq:gedis-0}
\lim_{n\longrightarrow +\infty}\dfrac{1}{N-1}\displaystyle\sum_{k=1}^{N-1} k\left( \sqrt{ n^{-2}\lambda_{k+1,h}}-\sqrt{n^{-2}\lambda_{k,h}}\right)=\displaystyle\int_{Rg \left(\sqrt{\omega_\phi^p} \right)} \left(\sqrt{\xi_{\phi}^p}\right)^{-1}(y)\; dy.
\end{equation}
\end{proposition}

\begin{proof}
Let $n \in \mathbb{N}^*$. We observe that the range of the function $\sqrt{\omega_\phi^p}$ is decomposed as 
$$
Rg\left(\sqrt{\omega_\phi^p}\right)=\left[0,\sqrt{n^{-2}\lambda_{1,h}}\right)\cup\left(\bigcup_{k=1}^{N-OUT(p,n)-1}\left[\sqrt{n^{-2}\lambda_{k,h}},\sqrt{n^{-2}\lambda_{k+1,h}}\right)\right)\cup\left[\sqrt{n^{-2}\lambda_{N-OUT(p,n),h}},\max Rg\left(\sqrt{\omega_\phi^p}\right)\right].
$$
Hence,
\begin{equation}\label{eq:gedis-1}
\max_{y\in Rg\left(\sqrt{\omega_\phi^p}\right)}\left| N_n^p(y) - \left(\sqrt{\xi_{\phi}^p}\right)^{-1}(y)    \right|=\max \left \{T_1,T_2,T_3 \right\},
\end{equation}
where 
$$
\begin{cases}
T_1=\max\left\{\left(\sqrt{\xi_{\phi}^p}\right)^{-1}(y): \quad 0\leq y<  \sqrt{n^{-2}\lambda_{1,h} } \right\},\vspace*{0.3cm}\\
T_2=\max \left\{ \left|\dfrac{k}{N+1}-  \left(\sqrt{\xi_{\phi}^p}\right)^{-1}(y)  \right|: \quad 1 \leq k \leq N-OUT(p,n)-1 \text{ and } \sqrt{n^{-2}\lambda_{k,h}}\leq y < \sqrt{n^{-2}\lambda_{k+1,h}} \right\},\vspace*{0.3cm}\\
T_3=\max\left\{\left|\dfrac{N-OUT(p,n)}{N+1} -\left(\sqrt{\xi_{\phi}^p}\right)^{-1}(y)  \right|: \quad  \sqrt{n^{-2}\lambda_{N-OUT(p,n),h}}\leq y\leq \max Rg\left(\sqrt{\omega_\phi^p}\right)\right\}.
\end{cases}
$$
Now, we'll estimate each of the three terms $T_1$, $T_2$, and $T_3$. For $T_1$, using the monotonicity of the function $y \mapsto \left(\sqrt{\xi_{\phi}^p}\right)^{-1}$, we have:
\begin{equation}\label{eq:gedis-2}
T_1\leq \left(\sqrt{\xi_{\phi}^p}\right)^{-1} \left( \sqrt{n^{-2}\lambda_{1,h}} \right).
\end{equation}
Applying the mean value theorem and Proposition \ref{stefa}, we obtain:
\begin{equation}\label{eq:gedis-3}
T_2\leq C \max_{1\leq k\leq  N-1-OUT(p,n)}\left(\left|\sqrt{\xi_{\phi}^p}\left(\dfrac{k}{N+1}\right)- \sqrt{n^{-2}\lambda_{k,h}}   \right|+ \left|  \sqrt{ n^{-2}\lambda_{k+1,h}}-\sqrt{n^{-2}\lambda_{k,h}}\right|\right),
\end{equation}
and
\begin{equation}\label{eq:gedis-4}
T_3\leq  C \left(\left|\sqrt{\xi_{\phi}^p}\left(\dfrac{N-OUT(p,n)}{N+1}\right)- \sqrt{n^{-2}\lambda_{N-OUT(p,n),h}}   \right|+ \left|  \sqrt{ n^{-2}\lambda_{N-OUT(p,n),h}}-\max Rg\left(\sqrt{\omega_\phi^p}\right)\right|\right),
\end{equation}
where $C:=\max\left( \left(\sqrt{\xi_{\phi}^p}\right)^{-1}\right)'$. We obtain then the uniform convergence of $G_n^p$ by taking the limit in \eqref{eq:gedis-1}–\eqref{eq:gedis-4} using:
$$
\lim_{n\longrightarrow +\infty}\left(\sqrt{\xi_{\phi}^p}\right)^{-1} \left(\sqrt{n^{-2}\lambda_{1,h}}\right)=0, \quad \lim_{n\longrightarrow +\infty} \max_{k\in\mathcal{I}(p,n)}\left|\sqrt{\xi_{\phi}^p}\left(\dfrac{k}{N+1}\right)- \sqrt{n^{-2}\lambda_{k,h}}   \right|=0,
$$
along with
$$
\max Rg\left(\sqrt{\omega_\phi^p}\right)=\sqrt{\xi_{\phi}^p}(1), \quad \lim_{n\longrightarrow +\infty} \max_{1\leq k\leq  N-1-OUT(p,n)} \left|   \sqrt{ n^{-2}\lambda_{k+1,h}}-\sqrt{n^{-2}\lambda_{k,h}}\right| =0.
$$

The approximation property \eqref{eq:gedis-0} follows from the uniform convergence of $G_n^p$ and the fact that we are working on a compact domain.
\end{proof}

\section{Uniform Gap Condition}\label{unifgap}
\hspace{0.5cm}In this section, we analyze the behavior of the gap between successive square roots of eigenvalues. As mentioned, we aim to establish a sufficient condition for the uniform gap property. The previous section is pivotal to achieving our objective because it is essential to understand the impact of reparametrization and the convexity of the symbol on the distribution of eigenvalues before we can analyze the gaps between them. To construct this sufficient condition, we initiate our analysis by examining the case where $p=1$. In this scenario, we can explicitly calculate the symbol, and as presented in the following subsection, we can identify intervals where the symbol is strictly convex, which allows for a direct application of Corollary \ref{cl2.5}.
\subsection{The Case of Linear B-Splines}
\hspace{0.5cm}In the sequel, we shall examine the symbol $\sqrt{\xi_{\phi}^p}$ defined in the previous section for the case when $p=1$. This analysis will provide valuable insights for constructing a sufficient condition for the uniform gap property in the subsequent subsection. Additionally, we will study the approximate gap, which is defined as follows
\begin{definition}[Approximate discrete gap]
Let  $\phi\in\mathbf{C}_{[0,1]}$, such that $n^{-2}L^{1}_{n,\phi}\sim_{\lambda} \xi_{\phi}^1$. The approximate discrete gap is defined by:
\begin{equation}\label{eq:approximate-discrete-gap}
\widetilde{\delta}_{n}=n \displaystyle\inf_{1\leq i \leq n-2} \sqrt{\xi_{\phi}^1\left( \dfrac{i+1}{n}\right)}- \sqrt{\xi_{\phi}^1\left( \dfrac{i}{n}\right)},\quad\forall n\in\mathbb{N}^{*}.
\end{equation}
\end{definition}
In this section, we assume that the reparametrization $\phi\in\mathbf{C}_{[0,1]}$ is strictly convex, as we obtain equivalent results when the reparametrization is strictly concave. The following proposition illustrates the regularity and the convexity of the function $\Psi_{\phi}^{\sqrt{\;}}$.

\begin{proposition}\label{propc}

Let $\phi$ be a strictly convex reparametrization of the interval $[0,1]$. Then $\Psi_{\phi}^{\sqrt{\;}}$ is of class $C^{1}\left(Rg\left(\sqrt{\omega^1_{\phi_1}}\right)\right)$. Moreover, for every $\epsilon_1,\epsilon_2\in(0,1)$, the function $\Psi_{\phi}^{\sqrt{\;}}$ is $C^{2}\left(\left(0,\epsilon_1\frac{\sqrt{12}}{\phi^{'}(1)}  \right)\cup\left(\frac{\sqrt{12}}{\phi^{'}(1)},\epsilon_2\frac{\sqrt{12}}{\phi^{'}(0)}\right) \right)$. In addition, $\Psi_{\phi}^{\sqrt{\;}}$ is strictly concave over $\left(0,\frac{\sqrt{6}}{\phi^{'}(1)}\right)\cup \left(\frac{\sqrt{12}}{\phi^{'}(1)},\epsilon_2\frac{\sqrt{12}}{\phi^{'}(0)}\right)$.
\end{proposition}

\begin{proof}
Let $\phi\in\mathbf{C}_{[0,1]}$ and $y\in \mathbb{R}_+$. As a first step, we obtain an explicit expression for $\Psi_{\phi}^{\sqrt{\;}}(y)$. We proceed as follows
\begin{align}
\Psi_{\phi}^{\sqrt{\;}}(y)= \mu_2\left\{ \sqrt{\omega_{\phi}^1}\leq y\right\}& =\mu_2\left\{ (x,\theta)\in [0,1]\times[0,\pi],\;\;\; \omega_{\phi}^1(x,\theta)\leq y^2 \right\} \nonumber \\
 & =\mu_2\left\{ (x,\theta)\in [0,1]\times[0,\pi],\;\;\; \dfrac{6(1-\cos{\theta})}{2+\cos{\theta}}\leq \left(y \phi^{'}(x)\right)^2 \right\} \nonumber\\
 &= \mu_2\left\{ (x,\theta)\in [0,1]\times[0,\pi],\;\;\; 6(1-\cos{\theta})\leq 2 \left(y \phi^{'}(x)\right)^2 + \left(y \phi^{'}(x)\right)^2 \cos{\theta} \right\}\nonumber\\
 &= \mu_2\left\{ (x,\theta)\in [0,1]\times[0,\pi],\;\;\; \cos{\theta}\geq \beta_x \right\}.
\label{eq:propc-eq1}
\end{align}
Here, $\beta_x$ is defined as
$$
\beta_x = \dfrac{6-2 \left(y\phi^{'}(x)\right)^2}{6+\left(y \phi^{'}(x)\right)^2}, \quad 0 \leq x \leq 1. 
$$
To evaluate the last integral \eqref{eq:propc-eq1}, we employ the property that $\cos:\, [0, \pi] \longrightarrow [-1,1]$ is invertible, $\arccos$ being its inverse function. This requires characterizing the conditions under which $\beta_x$ is in $[-1, 1]$. In fact, it's easy to see that $\beta_x \leq 1$ for all $x \in [0,1]$, and $\beta_x \geq -1$ if, and only if
$$
\begin{cases}
0 \leq x \leq 1 \vspace*{0.25cm}\\
0 \leq y \leq \frac{\sqrt{12}}{\phi'(1)}, 
\end{cases} \quad \text{or} \quad 
\begin{cases}
0 \leq x \leq \left(\phi' \right)^{-1} \left( \frac{\sqrt{12}}{y} \right) \vspace*{0.25cm}\\
\frac{\sqrt{12}}{\phi'(1)} \leq y \leq \frac{\sqrt{12}}{\phi'(0)}.
\end{cases}
$$ 

Therefore, from \eqref{eq:propc-eq1}, we deduce that

\begin{equation}\label{psiequa}
\Psi_{\phi}^{\sqrt{\;}}(y)=\left\{
\begin{array}{rrrrr}
\displaystyle\int_{0}^{1}\arccos\left(  \dfrac{6-2\left( y\phi'(x)\right)^2}{6+\left( y\phi'(x)\right)^2}\right)\;dx, & y\in \bar{J_1},\\\\

\displaystyle\int_{0}^{\left( \phi'\right)^{-1}\left(\dfrac{\sqrt{12}}{y}\right)}\arccos\left(  \dfrac{6-2\left( y\phi'(x)\right)^2}{6+\left( y\phi'(x)\right)^2}\right)\;dx &+\pi\left(1-\left( \phi'\right)^{-1}\left(\dfrac{\sqrt{12}}{y}\right)\right), \vspace*{0.25cm}\\ 
& y\in \bar{J_2}.
\end{array}
\right.
\end{equation}
where $J_1:= \left(0,\dfrac{\sqrt{12}}{\phi'(1)}\right)$ and $J_2:=  \left(\dfrac{\sqrt{12}}{\phi'(1)},\dfrac{\sqrt{12}}{\phi'(0)}\right)$. As a second step, we show that $\Psi_{\phi}^{\sqrt{\;}}(y)$ is of class $C^1 \left( \bar{J_1} \cup \bar{J_2} \right)$. We analyze the cases where $y \in \bar{J_1}$ and $y \in \bar{J_2}$ separately.

For $y \in \bar{J_1}$, we define
$$  
\varphi(x,y)= \arccos\left(  \dfrac{6-2\left( y\phi'(x)\right)^2}{6+\left( y\phi'(x)\right)^2}\right),\;\; \forall x \in [0,1],\quad y \in \bar{J_1}.
$$
For almost every $x \in [0,1]$ and all $y \in \bar{J_1}$, we have
\begin{align*}
\dfrac{\partial \varphi}{\partial y}(x,y)&=\dfrac{-1}{\sqrt{1-\left[ \dfrac{6-2\left( y\phi'(x)\right)^2}{6+\left( y\phi'(x)\right)^2}\right]^2}}\;\dfrac{-36 y \left(\phi'(x)\right)^2}{\left[ 6+\left( y\phi'(x)\right)^2\right]^2}\vspace*{0.25cm}\\
&=\dfrac{36 y \left(\phi'(x)\right)^2}{\sqrt{ \left(12-\left( y\phi'(x)\right)^2\right) \left[ 3\left( y\phi'(x)\right)^2\right]    }}\;\dfrac{1}{ 6+\left( y\phi'(x)\right)^2}\vspace*{0.25cm}\\
&=\dfrac{\phi'(x)}{\sqrt{1-\left(\dfrac{y\phi'(x)}{\sqrt{12}}\right)^2}}\; \dfrac{6}{6+\left( y\phi'(x)\right)^2} \vspace*{0.25cm}\\
& \leq \dfrac{\phi'(x)}{\sqrt{1-\left(\dfrac{\phi'(x)}{\phi'(1)}\right)^2}}  = \dfrac{\phi'(x) \sqrt{\phi'(1)}}{\sqrt{1+\dfrac{\phi'(x)}{\phi'(1)}}} \;\dfrac{1}{\sqrt{\phi'(1)-\phi'(x)}} \vspace*{0.25cm}\\
& \leq \frac{\left(\phi'(1)\right)^{3/2}}{\sqrt{1+\frac{\phi'(0)}{\phi'(1)}}} \;\dfrac{1}{\sqrt{\phi'(1)-\phi'(x)}},
\end{align*}
and
$$
\int_0^1 \dfrac{dx}{\sqrt{\phi'(1)-\phi'(x)}}    \leq \frac{1}{\displaystyle\inf_{[0,1]} \left\{\phi''\right\}} \int_{\phi'(0)}^{\phi'(1)} \dfrac{dx}{\sqrt{\phi'(1)-x}}  < \infty.
$$
Thus, by the Lebesgue Dominated Convergence Theorem, $\Psi_{\phi}^{\sqrt{\;}} \in C^1(\bar{J_1})$, and
\begin{equation}\label{PsiJ1}
\left(\Psi_{\phi}^{\sqrt{\;}}\right)'(y)=\displaystyle\int_0^1 \dfrac{\partial \varphi}{\partial y}(x,y)\;dx,\quad \forall y\in \bar{J_1}.
\end{equation}
Similar considerations are applied for $y \in \bar{J_2}$, where we obtain
\begin{align*}
\int_{0}^{\left( \phi'\right)^{-1}\left(\dfrac{\sqrt{12}}{y}\right)} \dfrac{\phi'(x)}{\sqrt{1-\left(\dfrac{y\phi'(x)}{\sqrt{12}}\right)^2}}\; \dfrac{6}{6+\left( y\phi'(x)\right)^2} \, dx  = \left(\dfrac{\phi'(t)}{\phi'(1)}\right)^2\;\displaystyle\int_{z_0}^{1} \dfrac{\phi'(z) g(z)}{\sqrt{1-\left(\dfrac{\phi'(z)}{\phi'(1)}\right)^2}} \, \frac{1}{1+2 \left(\frac{\phi'(z)}{\phi'(1)} \right)^2} dz,
\end{align*}
where
\begin{equation}\label{changePsi2}
t:=\left(\phi'\right)^{-1}\left(\dfrac{\sqrt{12}}{y}\right), \; z:=\left(\phi'\right)^{-1}\left(\dfrac{\phi'(1)}{\phi'(t)} \phi'(x)\right), \;  z_0:=\left(\phi'\right)^{-1}\left(\dfrac{\phi'(1)}{\phi'(t)} \phi'(0)\right), \ g(z):=\dfrac{\phi''(z)}{\phi''\left[  \left(\phi^{'}\right)^{-1}\left(\dfrac{\phi'(t)}{\phi'(1)} \phi'(z)\right) \right]}. 
\end{equation}
Given that
$$
\int_{z_0}^1 \dfrac{\phi'(z) g(z)}{\sqrt{1-\left(\dfrac{\phi'(z)}{\phi'(1)}\right)^2}} \, dz \leq  \phi^{'}(1)  \displaystyle\max_{[0,1]}\{g\} \;\displaystyle\int_{0}^{1} \dfrac{1}{\sqrt{1- \left(\dfrac{\phi'(z)}{\phi'(1)}\right)^2}} \,dz<\infty, 
$$
we conclude that $\Psi_{\phi}^{\sqrt{\;}} \in C^1(\bar{J_2})$, with
\begin{equation}\label{PsiJ2}
\left(\Psi_{\phi}^{\sqrt{\;}}\right)'(y) = \left(\dfrac{\phi'(t)}{\phi'(1)}\right)^2\;\displaystyle\int_{z_0}^{1} \dfrac{\phi'(z) g(z)}{\sqrt{1-\left(\dfrac{\phi'(z)}{\phi'(1)}\right)^2}} \, \frac{1}{1+2 \left(\frac{\phi'(z)}{\phi'(1)} \right)^2} dz, \quad \forall y \in \bar{J_2}.
\end{equation}

For the third step, we fix $\epsilon_1$ and $\epsilon_2$ in $(0,1)$ and we aim to demonstrate  that $\left(\Psi_{\phi}^{\sqrt{\;}}\right)'(y)$ is of class $C^{1}\left(J_1^{\epsilon}\cup J_2^{\epsilon}\right)$, where $J_1^{\epsilon}=\left(0,\epsilon_1\frac{\sqrt{12}}{\phi^{'}(1)}  \right)$ and $J_2^{\epsilon}=\left(\frac{\sqrt{12}}{\phi^{'}(1)},\epsilon_2\frac{\sqrt{12}}{\phi^{'}(0)}\right)$.\\
Let us start with the case $y\in J_2^{\epsilon}$. From (\ref{changePsi2}) we can check that for all $y\in J_2^{\epsilon}$ we have $$ 0\leq z_0\leq \left(\phi^{'}\right)^{-1}\left(\epsilon_2\phi^{'}(1)\right)<1.$$
Let denote $\eta=\left(\phi^{'}\right)^{-1}\left(\epsilon_2\phi^{'}(1)\right)$ and define the function $\varphi_1$ over $[0,\eta]$ by
$$\varphi_1(z)=\dfrac{\phi'(z) g(z)}{\sqrt{1-\left(\dfrac{\phi'(z)}{\phi'(1)}\right)^2}} \, \frac{1}{1+2 \left(\frac{\phi'(z)}{\phi'(1)} \right)^2}.$$
From (\ref{PsiJ2}), we obtain 
$$\left(\Psi_{\phi}^{\sqrt{\;}}\right)'(y) = \left(\dfrac{\phi'(t)}{\phi'(1)}\right)^2\left[\displaystyle\int_{0}^{1} \varphi_1(z) dz-\displaystyle\int_{0}^{z_0} \varphi_1(z) dz\right], \quad \forall y \in J_2^{\epsilon}.$$
On the other hand, the function $\varphi_1$ is continuous over $[0,\eta]$ and $y\rightarrow z_0(y)$ is $C^1(J_2^{\epsilon})$, using Leibniz’s Integral Rule we obtain $\Psi_{\phi}^{\sqrt{\;}}\in C^2(J_2^{\epsilon})$, with
\begin{equation}\label{Psi2}
\left(\Psi_{\phi}^{\sqrt{\;}}\right)^{''}(y)=\frac{\partial}{\partial y}\left(\dfrac{\phi'(t)}{\phi'(1)}\right)^2\left(\displaystyle\int_{z_0}^{1} \varphi_1(z) dz\right)+\left(\dfrac{\phi'(t)}{\phi'(1)}\right)^2\left(-\varphi_1(z_0)\dfrac{\partial z_0}{\partial y}\right),\quad \forall y\in J_2^{\epsilon}. 
\end{equation}
Now let $y\in J_1^{\epsilon}$. From (\ref{PsiJ1}), we observe that for every $x\in [0,1]$ the function $\frac{\partial\varphi}{\partial y}$ is differentiable with respect to $y$ and we have 
\begin{align*}
    \frac{\partial^2\varphi}{\partial^2 y}(x,y)&=-6\phi^{'}(x)\dfrac{\frac{-\left(\frac{\phi^{'}(x)}{\sqrt{12}}\right)^2 2y}{2\sqrt{1-\left(\frac{y\phi^{'}(x)}{\sqrt{12}}\right)^2}}\left(6+\left(y\phi^{'}(x)\right)^2\right)+\sqrt{1-\left(\frac{y\phi^{'}(x)}{\sqrt{12}}\right)^2}\;2y\left(\phi^{'}(x)\right)^2}{\left(1-\left(\frac{y\phi^{'}(x)}{\sqrt{12}}\right)^2\right)\left(6+\left(y\phi^{'}(x)\right)^2\right)^2}\\
    &=-6y\left(\phi^{'}(x)\right)^3\dfrac{-\frac{1}{12}\left(6+\left(y\phi^{'}(x)\right)^2\right)+2\left(1-\left(\frac{y\phi^{'}(x)}{\sqrt{12}}\right)^2\right)}{\left(1-\left(\frac{y\phi^{'}(x)}{\sqrt{12}}\right)^2\right)^{\frac{3}{2}}\left(6+\left(y\phi^{'}(x)\right)^2\right)^2}\\
    &=\dfrac{-9y\left(\phi^{'}(x)\right)^3 \left(1-2\left(\frac{y\phi^{'}(x)}{\sqrt{12}}\right)^2\right)}{\left(1-\left(\frac{y\phi^{'}(x)}{\sqrt{12}}\right)^2\right)^{\frac{3}{2}}\left(6+\left(y\phi^{'}(x)\right)^2\right)^2}.
\end{align*}
Observe that $\frac{\partial^2\varphi}{\partial^2 y}$ is continuous over $[0,1]\times J_1^{\epsilon} $, then  $\Psi_{\phi}^{\sqrt{\;}}\in C^2(J_1^{\epsilon})$, with
\begin{equation}\label{PsipJ1}
\left(\Psi_{\phi}^{\sqrt{\;}}\right)^{''}(y)=\displaystyle\int_0^1 \dfrac{\partial^2 \varphi}{\partial^2 y}(x,y)\;dx,\quad\forall y\in J_1^{\epsilon}.
\end{equation}

To conclude the proof, let $y\in \left(0,\frac{\sqrt{6}}{\phi^{'}(1)}\right)$. Then, there exists an $\epsilon$ in $(0,1)$ such that 
\begin{align*}
    y&\leq \frac{\sqrt{6}}{\phi^{'}(1)}< \epsilon\frac{\sqrt{12}}{\phi^{'}(1)}\\
    &\Rightarrow y\in J_1^{\epsilon}\;\;\text{ and}\;\; 1-2\left(\frac{y\phi^{'}(x)}{\sqrt{12}}\right)^2\geq 1-\left(\frac{\phi^{'}(x)}{\phi^{'}(1)}\right)^2\geq 0.
\end{align*}
Hence, using the (\ref{PsipJ1}), we obtain 
$$\left(\Psi_{\phi}^{\sqrt{\;}}\right)^{''}(y)<0,\;\;\forall y\in \left(0,\frac{\sqrt{6}}{\phi^{'}(1)}\right). $$
Now, let us consider the case $y\in J_2^{\epsilon}$. From (\ref{changePsi2}), it is not difficult to see that 
$$\frac{\partial z_0}{\partial y}>0,\;\; \frac{\partial \phi^{'}(t)}{\partial y}<0\;\;\text{and}\;\; \varphi_1(z_0)>0.$$
Then, using (\ref{Psi2}), we obtain
$$\left(\Psi_{\phi}^{\sqrt{\;}}\right)^{''}(y)<0,\;\;\forall y\in J_2^{\epsilon}. $$
Finally, the function $\Psi_{\phi}^{\sqrt{\;}}$ is strictly concave over $\left(0,\frac{\sqrt{6}}{\phi^{'}(1)}\right)\cup \left(\frac{\sqrt{12}}{\phi^{'}(1)},\epsilon_2\frac{\sqrt{12}}{\phi^{'}(0)}\right)$.
\end{proof}

\begin{remark}
    In the case of a strictly concave reparametrization $\phi\in\mathbf{C}_{[0,1]}$, it is sufficient to interchange $\phi(0)$ and  $\phi(1)$ in Proposition \ref{propc}.
\end{remark}

 With the help of the previous Proposition \ref{propc} and by applying Corollary  \ref{cl2.5}, we can establish a decreasing order of the number  of eigenvalues over $\left(0,\frac{\sqrt{6}}{\phi^{'}(1)}\right)\cup \left(\frac{\sqrt{12}}{\phi^{'}(1)},\epsilon_2\frac{\sqrt{12}}{\phi^{'}(0)}\right)$.  It's important to note that the natural behavior of the square roots of eigenvalues, or the simple one that provides us the order of the case where $\left( \sqrt{\xi_{\phi}^p} \right)^{-1}=\frac{1}{\pi}\Psi_{\phi}^{\sqrt{\;}}$ is strictly concave in Corollary \ref{cl2.5} occurs when the distance between each successive square root of eigenvalue becomes larger. A formal proof of this observation is currently unavailable.\\
 
In the following corollary, we establish that the symbol $\sqrt{\xi_{\phi}^1}$ behaves linearly in the neighborhood of $0$. Furthermore, it confirms that the approximate gap indeed satisfies the uniform gap condition.

\begin{corollary}\label{lip=1}
Let $\phi\in\mathbf{C}_{[0,1]}$ be strictly convex. Then we have
$$ 
\widetilde{\delta}_{n}\geq \dfrac{1}{\displaystyle\sup_{y\in \left(0,\dfrac{\sqrt{12}}{\phi^{'}(0)}\right) }\left(\Psi_{\phi}^{\sqrt{\;}}\right)'(y)}\pi, \quad \text{and} \quad \sqrt{\xi_{\phi}^1}(x)\sim_{0^+} \pi\; x,
$$
where $\widetilde{\delta}_{n}$ stands for the approximate discrete gap defined by \eqref{eq:approximate-discrete-gap}.
\end{corollary}

\begin{proof}
Let us start by recalling that
\begin{equation}\label{eq:lip=1-1}
\sqrt{\xi_{\phi}^1}(x)=\left(\Psi_{\phi}^{\sqrt{\;}}\right)^{-1}(\pi x),\quad \forall x\in[0,1].
\end{equation}
Hence,  for each $i \in \{1, \cdots, n-1 \}$, there exists a unique $y_i \in \left( 0, \frac{\sqrt{12}}{\phi'(0)} \right)$ such that
$$
\sqrt{\xi_{\phi}^1\left( \dfrac{i}{n}\right)}=y_{i},\quad \text{and} \quad \Psi_{\phi}^{\sqrt{\;}}(y_{i})=\dfrac{i}{n}\pi.
$$
Moreover, for $i \in \{1, \cdots, n-2 \}$, by the mean value theorem and  using the fact that $\left(\Psi_{\phi}^{\sqrt{\;}}\right)'>0$ in $\left(0,\dfrac{\sqrt{12}}{\phi'(0)}\right)$, there exists $z_i \in (y_i, y_{i+1})$ such that
$$
\sqrt{\xi_{\phi}^1\left( \dfrac{i+1}{n}\right)}- \sqrt{\xi_{\phi}^1\left( \dfrac{i}{n}\right)} = \dfrac{1}{\left(\Psi_{\phi}^{\sqrt{\;}}\right)'(z_i)}\dfrac{\pi}{n},
$$ 
which implies
\begin{equation*} 
\sqrt{\xi_{\phi}^1\left( \dfrac{i+1}{n}\right)}- \sqrt{\xi_{\phi}^1\left( \dfrac{i}{n}\right)}  \geq \displaystyle\inf_{y\in \left(0,\dfrac{\sqrt{12}}{\phi'(0)}\right) } \dfrac{1}{\left(\Psi_{\phi}^{\sqrt{\;}}\right)'(y)} \;\;\dfrac{\pi}{n } =  \dfrac{1}{\displaystyle\sup_{y\in \left(0,\dfrac{\sqrt{12}}{\phi^{'}(0)}\right) }\left(\Psi_{\phi}^{\sqrt{\;}}\right)'(y)}\;\; \dfrac{\pi}{n } .
\end{equation*}
Consequently, we can conclude that
$$
\widetilde{\delta}_{n} = n \displaystyle\inf_{1\leq i \leq n-2} \sqrt{\xi_{\phi}^1\left( \dfrac{i+1}{n}\right)}- \sqrt{\xi_{\phi}^1\left( \dfrac{i}{n}\right)} \geq \dfrac{1}{\displaystyle\sup_{y\in \left(0,\dfrac{\sqrt{12}}{\phi^{'}(0)}\right) }\left(\Psi_{\phi}^{\sqrt{\;}}\right)'(y)} \pi,
$$
which demonstrates the approximate gap property. 

Regarding the asymptotic behavior of the symbol, it suffices to prove
\begin{equation}\label{eq:lip=1-2}
\left(\sqrt{\xi_{\phi}^1}\right)'(0) =\pi.
\end{equation}
Using \eqref{eq:lip=1-1} again, we get
$$
\left(\sqrt{\xi_{\phi}^1}\right)^{'}(0)=\pi \dfrac{1}{\left(\Psi_{\phi}^{\sqrt{\;}}\right)^{'}\left(\left(\Psi_{\phi}^{\sqrt{\;}}\right)^{-1}(0)\right)}.
$$
Given $\Psi_{\phi}^{\sqrt{\;}}(0)=0$, and $\left(\Psi_{\phi}^{\sqrt{\;}}\right)'(0)=1$, this yields \eqref{eq:lip=1-2} and conclude the proof.
\end{proof}

\subsection{Sufficient Condition}
\hspace{0.5cm}In this subsection, we establish a condition sufficient to ensure a uniform gap between the discrete eigenvalues of the matrix $n^{-2}L^{p}_{\phi,n}$. The derivation of this condition is guided by the spectral analysis done in Section \ref{cl2.5}. 

The proof of the main result in this subsection relies on the observation that, similar to the case when $p=1$, the symbol $\sqrt{\xi_\phi^p}$ exhibits linear behavior near to $0$ for all values of $p$; as indicated by the following result

\begin{proposition}\label{linb}
Consider $p\in \mathbb{N}^{*}$ and $\phi\in \mathbf{C}_{[0,1]}$ a   reparameterization  of $[0,1]$. Then, it holds 
$$
\sqrt{\xi_{\phi}^p}(x)\sim_{0^+} \gamma x,
$$
where $\gamma:= \frac{\pi}{\left(\Psi_{\phi}^{\sqrt{\;}}\right)'(0)}$.
\end{proposition}

\begin{proof}
The result for the case $p=1$ is provided by Corollary \ref{lip=1}.  Now, let's focus on the case $p\geq 2$. Using relations \eqref{eq:preliminaries-1} from Subsection \ref{symbol}, we derive the following inequalities for all $\theta\in[0,\pi]$
$$
2\sin\left(\frac{\theta}{2}\right)\left(\frac{2}{\pi}\right)^{(p-1)/2}\leq\sqrt{e_p(\theta)}\leq 2\sin\left(\frac{\theta}{2}\right)\left(\frac{\pi}{2}\right)^{(p+1)/2},\quad \forall \theta\in[0,\pi].
$$
Given that for all $y\in Rg\left(\sqrt{\omega^p_{\phi}}\right) $, we have
$$
\Psi_{\phi}^{\sqrt{\;}}(y)=\mu_{2} \left( \left\{  (x,\theta)\in [0,1]\times[0,\pi]:\quad \frac{\sqrt{e_p(\theta})}{\phi^{'}(x)}\leq y    \right\} \right),
$$
we deduce 
\begin{equation}\label{eq:linb-1}
\Psi_{1}(y)\leq \Psi_{\phi}^{\sqrt{\;}}(y)\leq \Psi_{2}(y),
\end{equation}
where
$$
\Psi_{1}(y)=\mu_{2} \left( \left\{  (x,\theta)\in [0,1]\times[0,\pi]:\quad \frac{2\sin\left(\frac{\theta}{2}\right)\left(\frac{\pi}{2}\right)^{(p+1)/2}}{\phi^{'}(x)}\leq y    \right\} \right),
$$
and 
$$
\Psi_{2}(y)=\mu_{2} \left( \left\{  (x,\theta)\in [0,1]\times[0,\pi]:\quad \frac{2\sin\left(\frac{\theta}{2}\right)\left(\frac{2}{\pi}\right)^{(p-1)/2}}{\phi^{'}(x)}\leq y    \right\} \right).
$$
According to \cite[Proposition 1]{bianchi2018spectral}, $\Psi_1$ and $\Psi_2$ are right-differentiables at $0$, and using relation (37) in the proof \cite[Proposition 1]{bianchi2018spectral}, we obtain 
$$
\Psi_1'(0) = \left( \frac{2}{\pi} \right)^{(p+1)/2}, \quad \text{and} \quad \Psi_2'(0) = \left( \frac{\pi}{2} \right)^{(p-1)/2}.
$$
Hence, by dividing both sides of \eqref{eq:linb-1} by $y$ and taking the limit as $y \rightarrow 0^+$, we get
$$
0<\left(\frac{2}{\pi}\right)^{(p+1)/2}\leq \left(\Psi_{\phi}^{\sqrt{\;}}\right)'(0)\leq\left(\frac{\pi}{2}\right)^{(p-1)/2}.
$$
To conclude the proof of the proposition, we remark that
$$
\left(\sqrt{\xi_{\phi}^p}\right)'(0)=\pi \dfrac{1}{\left(\Psi_{\phi}^{\sqrt{\;}}\right)'\left(\left(\Psi_{\phi}^{\sqrt{\;}}\right)^{-1}(0)\right)},
$$ 
where we have used that $\sqrt{\xi_{\phi}^p}(x)=\left(\Psi_{\phi}^{\sqrt{\;}}\right)^{-1}(\pi x)$, for all $x \in [0,1]$.
\end{proof}

In the previous Section \ref{disIGA}, we established the possibility of achieving uniform convergence in the discrete Wely's law, as presented in the Corollary \ref{gedis} by employing a reparametrization $\phi\in\mathbf{C}_{[0,1]}$ that satisfies the average gap condition.
In the subsequent result, we aim to improve the estimation of the square root of the eigenvalues discussed previously in Section \ref{sec:preliminaries} and used throughout our previous analysis of the  eigenvalue distribution. This improvement is based on the fact  that our symbol $\sqrt{\xi_{\phi}^p}$ exhibits linearity in the neighborhood of $0$.

\begin{corollary}\label{cl2}
Let $p\geq 1$  and $\phi\in\mathbf{C}_{[0,1]}$, such that  $n^{-2}L^{p}_{\phi,n}\sim_{\lambda} \xi_{\phi}^p.$ Then:
\begin{equation}\label{eq:cl2}
\lim_{n \rightarrow +\infty} \sup_{k \in \mathcal{I}(p,n)}\left\{ \dfrac{n}{k}\left| \sqrt{n^{-2}\lambda_{k,h}} - \sqrt{\xi_{\phi}^p}\left(\frac{k}{N+1}\right)\right|  \right\} = 0.
\end{equation}
\end{corollary}

\begin{proof}
Assume that the condition \eqref{eq:cl2} is not satisfied. This implies the existence of a constant $C > 0$ and a sequence $(k_n)_n$ of elements in $\mathcal{I}(p, n)$ such that
\begin{equation}\label{eq:cl2-eq1}
\dfrac{n}{k_n}\left|\sqrt{n^{-2}\lambda_{k_n,h}}-\sqrt{\xi^p_{\phi}}\left( \dfrac{k_n}{N_n+1}\right) \right|\geq C>0.
\end{equation}
As the sequence $\left(\dfrac{k_n}{N_n+1}\right)_n$ is bounded in $[0,1]$, it possesses a convergent subsequence $\left(\dfrac{k_{j_n}}{N_{j_n}+1}\right)_n$. Let $x \in [0,1]$ be its limit.  In particular, from \eqref{eq:cl2-eq1}, we deduce that
\begin{equation}\label{eq:cl2-eq3}
\left|\sqrt{j_n^{-2}\lambda_{k_{j_n},h}}-\sqrt{\xi^p_{\phi}}\left( \dfrac{k_{j_n}}{N_{j_n}+1}\right) \right|\geq C \dfrac{k_{j_n}}{j_n}.
\end{equation}

We will discuss the cases $x \in (0, 1]$ and $x=0$ separately.

If $x \in (0, 1]$, taking the limit in \eqref{eq:cl2-eq1} as $n \longrightarrow 0$, we obtain
$$
\lim_{n \rightarrow 0} \left|\sqrt{j_n^{-2}\lambda_{k_{j_n},h}}-\sqrt{\xi^p_{\phi}}\left( \dfrac{k_{j_n}}{N_{j_n}+1}\right) \right|\geq C x,
$$
and by Discrete Weyl’s Law (Theorem \ref{th1}), we have $C x \leq 0$, which contradicts the assumption $x > 0$.

Now, suppose $x=0$. First, observe that
$$
\left(\diag_{k=1,\dots,N}\left[ \sqrt{\xi_{\phi}^p}\left(\frac{k}{N+1}\right) \right]\right)^{-1}\; \sqrt{n^{-2}L^p_{\phi,n}}\sim_{\lambda}1,
$$
thus, using Theorem \ref{th1}, we deduce
$$
\sqrt{\xi_{\phi}^p}\left( \dfrac{k}{N+1}\right)\sim_{\infty} \sqrt{n^{-2}\lambda_{k,h}}, \;\; \forall k\in \mathcal{I}(p,n).
$$
In particular, we have
\begin{equation}\label{eq:cl2-eq2}
\sqrt{\xi_{\phi}^p}\left( \dfrac{k_{j_n}}{N_{j_n}+1}\right)\sim_{\infty} \sqrt{j_n^{-2}\lambda_{k_{j_n},h}},
\end{equation}
and by Proposition \ref{linb}, since $\lim_{n \rightarrow +\infty} \dfrac{k_{j_n}}{N_{j_n}+1} = 0$, we obtain
$$ 
\sqrt{\xi_{\phi}^p}\left( \dfrac{k_{j_n}}{N_{j_n}+1}\right)\sim_{\infty} \gamma \;  \dfrac{k_{j_n}}{N_{j_n}+1},
$$
where $\gamma$ is the positive constant given in Proposition  \ref{linb}. This relation, together with \eqref{eq:cl2-eq2}, implies
$$
\sqrt{j_n^{-2}\lambda_{k_{j_n},h}}\sim_{\infty} \gamma\;  \dfrac{k_{j_n}}{N_{j_n}+1}.
$$
Dividing both sides of (\ref{eq:cl2-eq3}) by $\dfrac{k_{j_n}}{N_{j_n}+1}$ and letting $n$ tend to infinity yields $0<C\leq 0$, which is imposible. 
\end{proof}

At this point, we have  the necessary ingredients to deal with the uniform gap problem. We recall that the gap between the roots of consecutive eigenvalues associated with discrete spectra is defined as follows
\begin{equation}\label{eq:discrete-gap}
\delta_n^p = \inf_{1 \leq k \leq N-1} \left( \sqrt{\lambda_{k+1,h}}-\sqrt{\lambda_{k,h}} \right).
\end{equation}
Let us introduce the notation
\begin{equation*}\label{eq:m(n)}
m(n) = \min \argmin_{1 \leq k \leq N-1} \left( \sqrt{\lambda_{k+1,h}}-\sqrt{\lambda_{k,h}} \right),
\end{equation*}
and evidently, we have
\begin{equation*}
\delta_n^p =  \sqrt{\lambda_{m(n)+1,h}}-\sqrt{\lambda_{m(n),h}}.
\end{equation*}

The following theorem shows that, assuming the sequence $(m(n))_n$ is bounded, the gap in \eqref{eq:discrete-gap} is uniformly bounded from below with respect to $h$.

\begin{theorem}[Gap and symbol]\label{cl2.4}
Let $p \geq 1$ and $\phi \in \mathbf{C}_{[0,1]}$ be such that $n^{-2}L_{\phi,n}^p \sim_{\lambda} \xi^p_{\phi}$. If the sequence $(m(n))_{n>0}$ is bounded, then we have
$$
\liminf_{n \rightarrow +\infty} {\delta_n^p} \geq \gamma,
$$
where $\gamma$ is a positive constant as defined in Proposition \ref{linb}. 
\end{theorem}

\begin{proof}
Assume that $(m(n))_{n>0}$ is bounded. First, let's note that, as $n \rightarrow +\infty$, we have $1 \leq m(n) \leq N-1-OUT(p,n)$. In particular, for $n$ large enough, we obtain
$$
\delta_n^p=\inf_{1 \leq k \leq N-1-OUT(p,n)} \left(\sqrt{\lambda_{k+1,h}}-\sqrt{\lambda_{k,h}}\right).
$$

On the other hand, for all $n \in \mathbb{N}^*$, we have
\begin{equation*} 
\begin{split}
\delta_n^p=\sqrt{\lambda_{m(n)+1,h}}-\sqrt{\lambda_{m(n),h}}&=\frac{n}{m(n)+1}\left(\sqrt{n^{-2}\lambda_{m(n)+1,h}} -\sqrt{\xi_{\phi}^p}\left(\frac{m(n)+1}{N+1}\right) \right)(m(n)+1)\\
&-\frac{n}{m(n)}\left(\sqrt{n^{-2}\lambda_{m(n),h}} -\sqrt{\xi_{\phi}^p}\left(\frac{m(n)}{N+1}\right) \right)m(n)\\
&+n\left[  \sqrt{\xi_{\phi}^p}\left(\frac{m(n)+1}{N+1}\right)-\sqrt{\xi_{\phi}^p}\left(\frac{m(n)}{N+1}\right)\right].\\
\end{split}
\end{equation*}
This yields, as $n$ large enough,
\begin{equation}\label{eq:cl2.4-1}
\delta_n^p\geq -(2 m(n)+1)S_{n,p}+\frac{n}{N+1}\left[  m(n)\left(             \frac{ \sqrt{\xi_{\phi}^p}\left(\frac{m(n)+1}{N+1}\right)}{\frac{m(n)+1}{N+1}}   -    \frac{ \sqrt{\xi_{\phi}^p}\left(\frac{m(n)}{N+1}\right)}{\frac{m(n)}{N+1}}         \right)+  \frac{ \sqrt{\xi_{\phi}^p}\left(\frac{m(n)+1}{N+1}\right)}{\frac{m(n)+1}{N+1}} \right],
\end{equation}
with notation
$$
S_{n,p} = \sup_{k \in \mathcal{I}(p,n)} \left\{ \dfrac{n}{k}\left| \sqrt{n^{-2}\lambda_{k,h}} - \sqrt{\xi_{\phi}^p}\left(\frac{k}{N+1}\right)\right| \right\}.
$$
By Corollary \ref{cl2}, $\lim_{n \rightarrow + \infty}S_{n,p}=0$, and by Proposition \ref{linb}, $\sqrt{\xi_{\phi}^p}(x)\sim_{0^+} \gamma\, x$. Hence, as the sequence $(m(n))_n$ is bounded, by taking the limit in \eqref{eq:cl2.4-1}, we obtain
$$
\liminf{\delta_n^p}\geq\gamma.
$$
This concludes the proof of the theorem.
\end{proof}

In Theorem \ref{cl2.4}, the key factor for the proof is the linearity of the symbol $\sqrt{\xi^p_{\phi}}$  near zero. This property  is a result of our  symbol analysis when the average gap condition is met. However, it's worth noting that the condition of having a bounded sequence $(m(n))$  prevents us from handling outliers. This is because, if  $(m(n))$ extends in the  outliers index, we cannot  approximate these eigenvalues using the GLT symbol.
Additionally, it should be mentioned that the condition of having $(m(n))$ bounded is motivated by both our eigenvalues distribution analysis as presented in Section \ref{disIGA} and numerical evidence, as we will see in the next section. Moreover, when we do not have the optimal gap, the sequence becomes unbounded, indicating a numerical equivalence between the boundedness of $(m(n))$ and the optimal gap condition.

Before providing numerical illustrations to support these observations, we intend to present the following proposition, offering an alternative sufficient condition, applicable when we can employ a reparametrization $\phi$ (not necessary  in $\mathbf{C}_{[0,1]}$) in such a way that the approximation of the square root of the eigenvalues of the matrix $n^{-2}L_{\phi,n}^p$ can be achieved through uniform sampling of the symbol $\sqrt{\xi_{\phi}^p}$ up to the order of $o(1/n)$. In this scenario, we can establish the uniform gap property for non-outlier eigenvalues without the need to assume the boundedness of $(m(n))$.

\begin{proposition}\label{prop4.3}
Consider $p \geq 1$ and $\phi \in \mathbf{C}_{[0,1]}$, such that $n^{-2}L_{\phi,n}^p \sim_{\lambda} \xi^p_{\phi}$. We define the gap between successive non-outlier square root eigenvalues as follows:
$$
\delta_n^{p,out}=\inf_{1 \leq k \leq N-1-OUT(p,n)} \left( \sqrt{\lambda_{k+1,h}}-\sqrt{\lambda_{k,h}} \right), \quad n \geq 1.
$$ 
Assuming
\begin{equation}\label{eq:prop4.3-assumption}
\lim_{n \rightarrow +\infty} \max_{k \in \mathcal{I}(p,n)}\left|\sqrt{\lambda_{k,h}} - n\sqrt{\xi_{\phi}^p}\left(\frac{k}
{N+1}\right)\right| = 0,
\end{equation}
and for all $y \in \mathring{Rg}\left(\sqrt{\omega_{\phi}^p}\right)$, we have $(\Psi^{\sqrt{\;}}_{\phi})^{'}(y) \neq 0$, then it holds
$$
\liminf_{n \rightarrow +\infty} \, \delta_n^{p,out}\geq \widetilde{\gamma},
$$
where  $\widetilde{\gamma}=\inf_{x\in(0,1)}\left(\sqrt{\xi_{\phi}^p}\right)'(x)$.
\end{proposition}

{\begin{proof}
Let 
$\widetilde{m}(n) = \min \argmin_{1 \leq k \leq N-1-OUT(p,n)} \left( \sqrt{\lambda_{k+1,h}}-\sqrt{\lambda_{k,h}} \right).$ We then compute 
\begin{align}\label{eq:prop4.3-proof-1}
\delta_n^{p,out}&= \sqrt{\lambda_{\widetilde{m}(n)+1,h}}-\sqrt{\lambda_{\widetilde{m}(n),h}} \nonumber\\
   &=\sqrt{\lambda_{\widetilde{m}(n)+1,h}}-n\sqrt{\xi_{\phi}^p}\left(\frac{\widetilde{m}(n)+1}{N+1}\right)+n\sqrt{\xi_{\phi}^p}\left(\frac{\widetilde{m}(n)}{N+1}\right)-\sqrt{\lambda_{\widetilde{m}(n),h}}\nonumber\\
   &\quad\quad +n\left[ \sqrt{\xi_{\phi}^p}\left(\frac{\widetilde{m}(n)+1}{N+1}\right)-  \sqrt{\xi_{\phi}^p}\left(\frac{\widetilde{m}(n)}{N+1}\right) \right] \nonumber\\
   & \geq -2\max_{k \in \mathcal{I}(p,n)} \left|\sqrt{\lambda_{k,h}} - n\sqrt{\xi_{\phi}^p}\left(\frac{k}
{N+1}\right)\right|+n\left[ \sqrt{\xi_{\phi}^p}\left(\frac{\widetilde{m}(n)+1}{N+1}\right)-  \sqrt{\xi_{\phi}^p}\left(\frac{\widetilde{m}(n)}{N+1}\right) \right].
\end{align}
By assumption \eqref{eq:prop4.3-assumption}, we get 
\begin{equation}\label{eq:prop4.3-proof-2}
\lim_{n \rightarrow +\infty} -2\max_{k \in \mathcal{I}(p,n)} \left|\sqrt{\lambda_{k,h}} - n\sqrt{\xi_{\phi}^p}\left(\frac{k}
{N+1}\right)\right|=0.
\end{equation}
To control the second term in \eqref{eq:prop4.3-proof-1}, we recall that by assumption, $(\Psi^{\sqrt{\;}}_{\phi})^{'}(y) \neq 0$ for all $y \in \mathring{Rg}\left(\sqrt{\omega_{\phi}^p}\right)$. Hence, $\sqrt{\xi_{\phi}^p}$ is differentiable in $(0,1)$ with
$$
\left(\sqrt{\xi_{\phi}^p}\right)'(x)=\frac{\pi}{\left(\Psi^{\sqrt{\;}}_{\phi}\right)' \left(\sqrt{\xi^p_{\phi}}( x)\right)}, \quad \forall x\in(0,1).
$$
In particular, we obtain $\widetilde{\gamma} = \inf_{x\in(0,1)}\left(\sqrt{\xi_{\phi}^p}\right)'(x)>0.$  Hence, by the mean value theorem, it follows that
\begin{equation}\label{eq:prop4.3-proof-3}
n\left[ \sqrt{\xi_{\phi}^p}\left(\frac{\widetilde{m}(n)+1}{N+1}\right)-  \sqrt{\xi_{\phi}^p}\left(\frac{\widetilde{m}(n)}{N+1}\right) \right] \geq \frac{n}{N+1} \widetilde{\gamma}, 
\end{equation}
and we conclude the proof using \eqref{eq:prop4.3-proof-1}--\eqref{eq:prop4.3-proof-3}.
\end{proof}

In Proposition \ref{prop4.3}, it is also possible to assume that $(m(n))$ is bounded, which results in $\delta_n^{p,out}=\delta_n^{p}$. Consequently, we can establish the proposition directly using  the assumed $o(1/n)$ eigenvalues estimation achieved through uniform sampling of the symbol. The main challenge lies in constructing  a reparametrization that satisfies the required hypotheses outlined in  Proposition \ref{prop4.3}.

In the next section, our objective is to apply Theorem \ref{cl2.4} and provide numerical evidence that supports the validity of the condition $(m(n))$ being bounded for a variety of different reparametrizations.

\section{Numerical tests}\label{sec:numerical-tests}

\hspace{0.5cm}In this section, we numerically analyze the behavior of the gap. Our objectives are to illustrate that the average gap condition, outlined in \cite{bianchi2018spectral} is not equivalent to the uniform gap property and to demonstrate the feasibility of constructing reparametrizations that satisfy the condition $(m(n))$ being bounded, as stated in Theorem \ref{cl2.4}. This section is divided into four subsections.

In the first subsection: analyze the gap when $p=1$. We present numerical evidence supporting the boundedness of $(m(n))$ and the convergence of the gap toward the optimal gap condition. Additionally, we demonstrate numerically that the gap between successive frequencies follows the convexity of the symbol. As a  final test, we illustrate the convergence of the normalized square root eigenvalues to the symbol.

From the second subsection to the last subsection, our focus shifts to a detailed exploration of the gap's behavior for $p\geq 2$. Based on numerical tests, it becomes apparent that when $p\geq3$, the sequence $(m(n))$ is unbonded and, in general, we lose the uniform gap condition, resulting in a non-equivalence between the average gap condition and the uniform gap condition. To solve this problem, we investigate the possibility of constructing a reparametrization that ensures the boundedness of $(m(n))$ and also satisfies the optimal gap condition. In conclusion, we summarize the numerical findings obtained from the analysis. 

\subsection{The Case of Linear $B$-splines}
\hspace{0.5cm}The case  $p=1$ is special for  several reasons. Firstly, the discretization corresponds to the classical finite element method that utilizes Lagrange polynomials, resulting in $OUT(p,n)=0$. Secondly, this specific case has been studied in \cite{ervedoza2016numerical}, where the authors demonstrated the uniform boundary observability of the wave equation under the assumption of strictly concave reparametrizations. In their proof, they employed the classical multiplication method due to a lack of information about the eigenvalues. However, this technique is not easily generalized to all values of $p$ due to the increasing complexity of the scheme when $p$ becomes large.\\

\noindent  {\bf \em Test 1: study of the gap condition ($\mathbf{p=1}$):} \vspace*{0.25cm}

\noindent \hspace{0.5cm}In this test, our objective is to leverage the gap condition proved in Theorem \ref{cl2.4}. We recall that the theorem's proof relies on the assumption that the sequence $(m(n))_{n \geq 1}$ is bounded, a hypothesis that has not been proved analytically. To address this, we initially developed a preliminary test, presenting numerical evidence affirming the boundedness of $(m(n))_{n \geq 1}$. The results are reported in Table \ref{t1}, where we show the values of $m(n)$ corresponding to various choices of $n$ and reparametrization functions $\phi$. Specifically, we considered the following reparametrizations
\begin{equation}\label{eq:reparametrizations-functions-tests}
\phi_1(x)=\dfrac{\ln(x+1)}{\ln(2)},\quad \phi_2(x)=\dfrac{e^x-1}{e-1},\quad \phi_3(x)=\sqrt{(2\theta+1)x+\theta^2}- \theta,\;\theta>0.
\end{equation}

\bigskip
  
We observe that the sequence $(m(n))_{n \geq 1}$ is not only bounded but also stationary. Specifically,  $(m(n))_{n \geq 1}$ is found to be equal to $1$ for all $n$, regardless of the chosen reparametrization.  This consistent behavior holds across various reparametrizations and is not influenced by the convexity of the reparametrization.

Now, we turn our attention to the examination of the gap condition. For this purpose, we select $\phi = \phi_1$ as the reparametrization function. In Figure \ref{f2} we plot the gap $\delta_n^1$ defined in \eqref{eq:discrete-gap}. As expected, this figure illustrates the convergence of the gap to the optimal value $\pi$.  \\
\begin{minipage}{\textwidth}
  \begin{minipage}[b]{0.49\textwidth}
    \centering
    \begin{tabular}{c|ccc}\hline
      $n$    & $\phi_1$ & $\phi_2$ & $\phi_3$ ($\theta=0.01$)\\ \hline
      $50$   & $1$      & $1$      & $1$\\
      $99$   & $1$      & $1$      & $1$\\
      $200$  & $1$      & $1$      & $1$\\
      $300$  & $1$      & $1$      & $1$\\
      $400$  & $1$      & $1$      & $1$\\
      $500$  & $1$      & $1$      & $1$\\
      $600$  & $1$      & $1$      & $1$\\
      $700$  & $1$      & $1$      & $1$\\
      $800$  & $1$      & $1$      & $1$\\
      $900$  & $1$      & $1$      & $1$\\
      $1600$ & $1$      & $1$      & $1$\\ \hline
      \end{tabular}
      \captionof{table}{$m(n)$ for different values of $n$ and reparametrization functions  ($p=1$)}
      \label{t1}
  \end{minipage}
  \hfill
  \begin{minipage}[b]{0.49\textwidth}
    \centering
    \includegraphics[width=\linewidth]{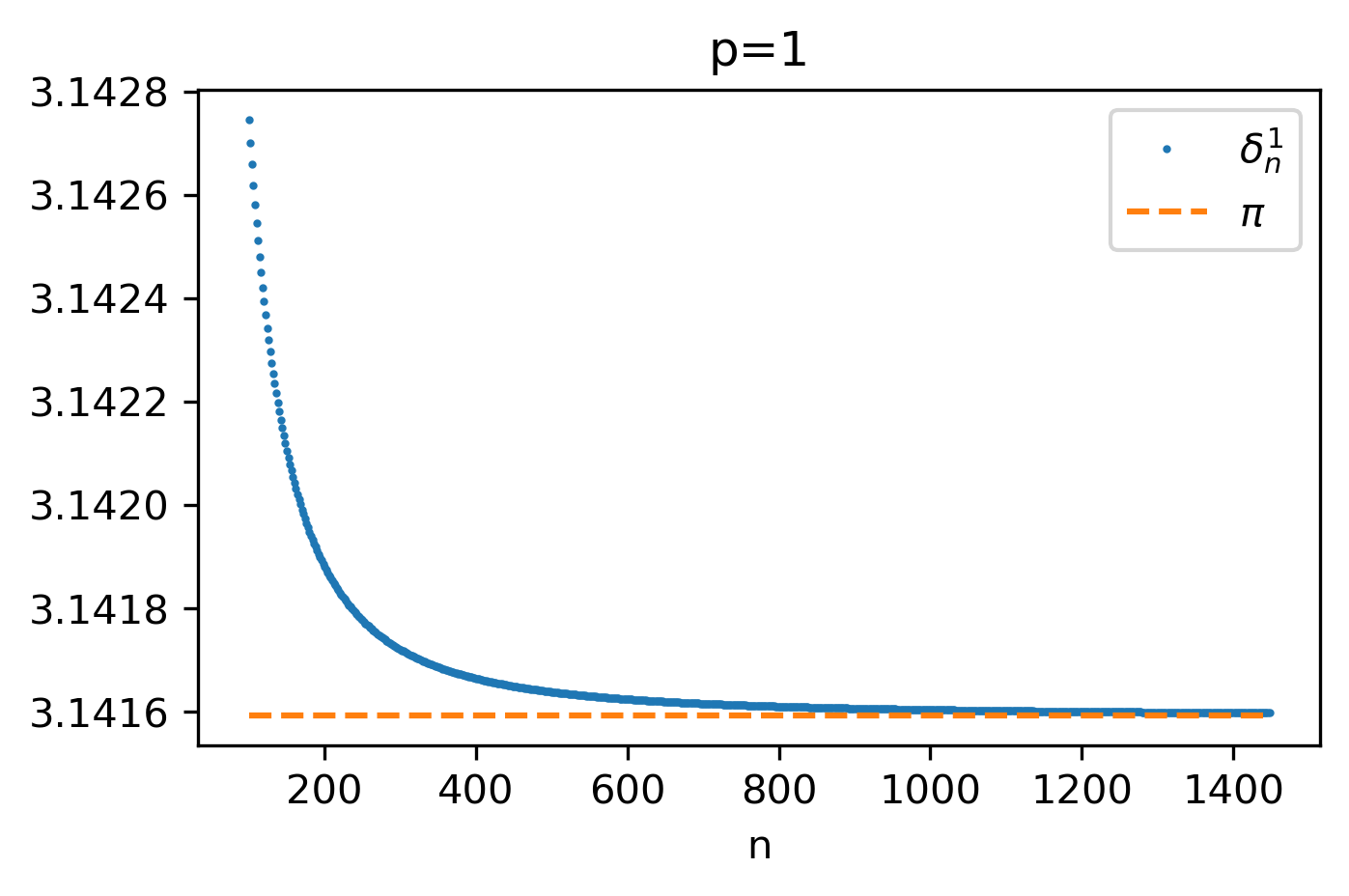}
    \captionof{figure}{The gap in function of $n$ ($p=1$).}
    \label{f2}
    \end{minipage}
  \end{minipage}\vspace*{0.25cm}
\noindent {\bf \em Test 2: study of the eigenvalues  distribution ($\mathbf{p=1}$):} \vspace*{0.25cm}

\noindent \hspace{0.5cm} The analysis of the distribution of our eigenvalues in the case $\phi_1$ is presented in Figure \ref{fgap}, where
$$\delta_n^k=\sqrt{\lambda_{k+1,h}}-\sqrt{\lambda_{k,h}},\;\; \forall n\in\mathbb{N}^{\star},\;\forall k\in\{1,\dots,n-2\}$$
It is apparent that the distance between square root eigenvalues exhibits a pattern of increase, decrease, and subsequent increase. As per Proposition (\ref{propc}), it is established that the symbol starts convex and ends convex. Moreover, upon examining  Figure \ref{f3}, it becomes evident that there exists a segment where the symbol is concave.

A noteworthy observation is that the distance between square root eigenvalues tends to increase when the symbol is convex and decreases when it is concave, resulting in the fact that the presence of segments in which the symbol is concave implies the non-uniform gap issue. However, in the specific case under consideration, the smallest gap is found between the square roots of the initial eigenvalues. This is attributed to the fact that the concave part of the symbol is confined to a small domain.

It is important to note that the results presented in this paper allow us to establish only the decreasing or increasing number of eigenvalues within each interval, as detailed in  Corollary \ref{cl2.5}.
\begin{figure}[H]
\minipage{0.5\textwidth}
  \includegraphics[width=\linewidth]{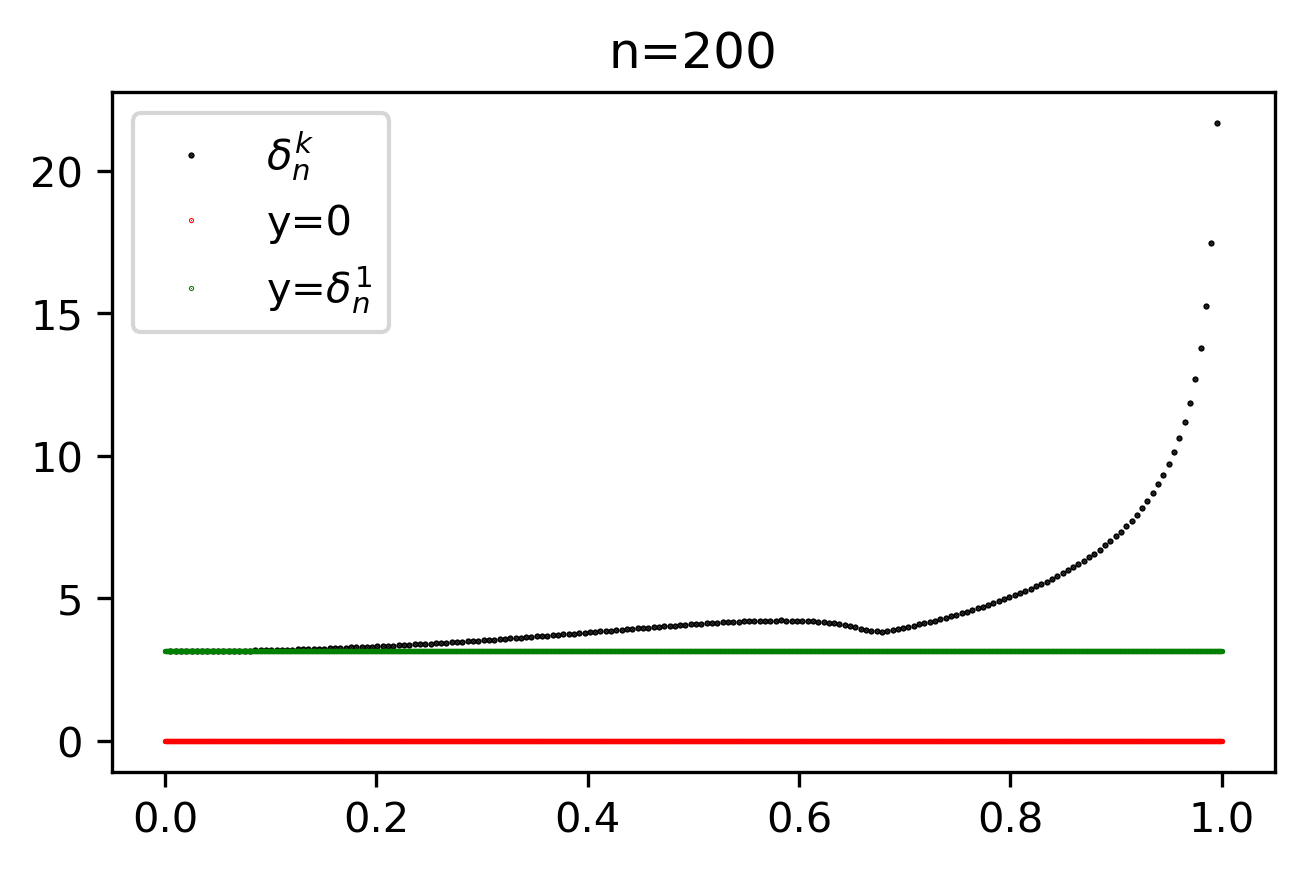}
\endminipage\hfill
\minipage{0.5\textwidth}
  \includegraphics[width=\linewidth]{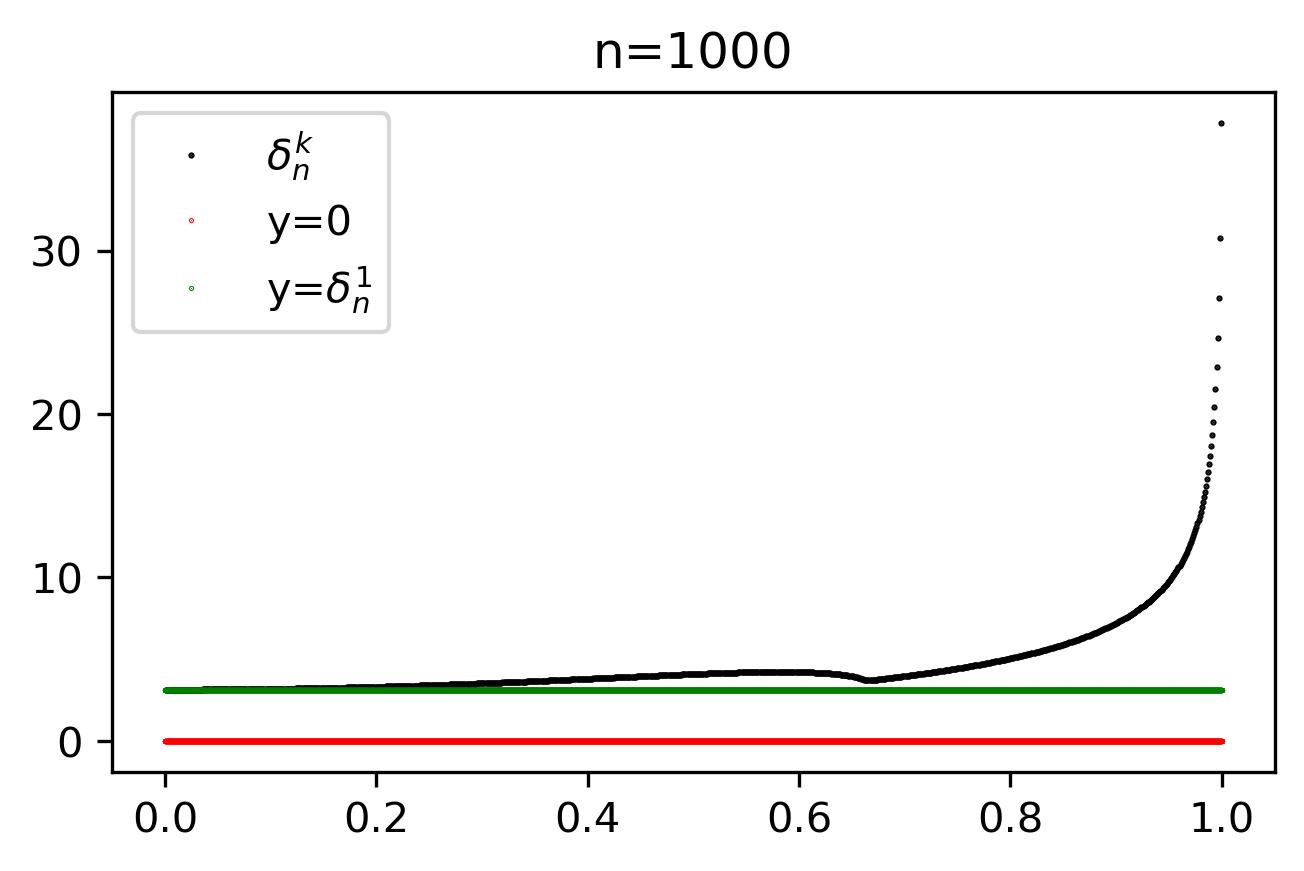}
\endminipage
 \caption{The distribution of the distance between the square root eigenvalues for different values of $n$ and $\phi_1$.}\label{fgap}
\end{figure}\vspace*{0.25cm}
\noindent {\bf \em Test 3: convergence study of the symbol ($\mathbf{p=1}$):} \vspace*{0.25cm}

\noindent  \hspace{0.5cm} In Figures \ref{f3}, we provide a plot of the normalized square root eigenvalues and the approximate symbol for various values of $n$ with a reparametrization function $\phi_1$. The algorithm detailed in \cite{garoni2019symbol}, \cite{garoni2017generalized}, and \cite{bianchi2021analysis}) was employed to derive the approximation of the symbol. Evidently, there is a clear convergence of the discrete eigenvalues towards the symbol.

In summary, in the case of $p=1$, we have observed the optimal gap condition for a large set of reparametrizations $\phi\in\mathbf{C}_{[0,1]}$. In addition, the square root of eigenvalues demonstrates a monotone order, not just in terms of quantity but also in their individual distribution.
\begin{figure}[H]
\minipage{0.5\textwidth}
  \includegraphics[width=\linewidth]{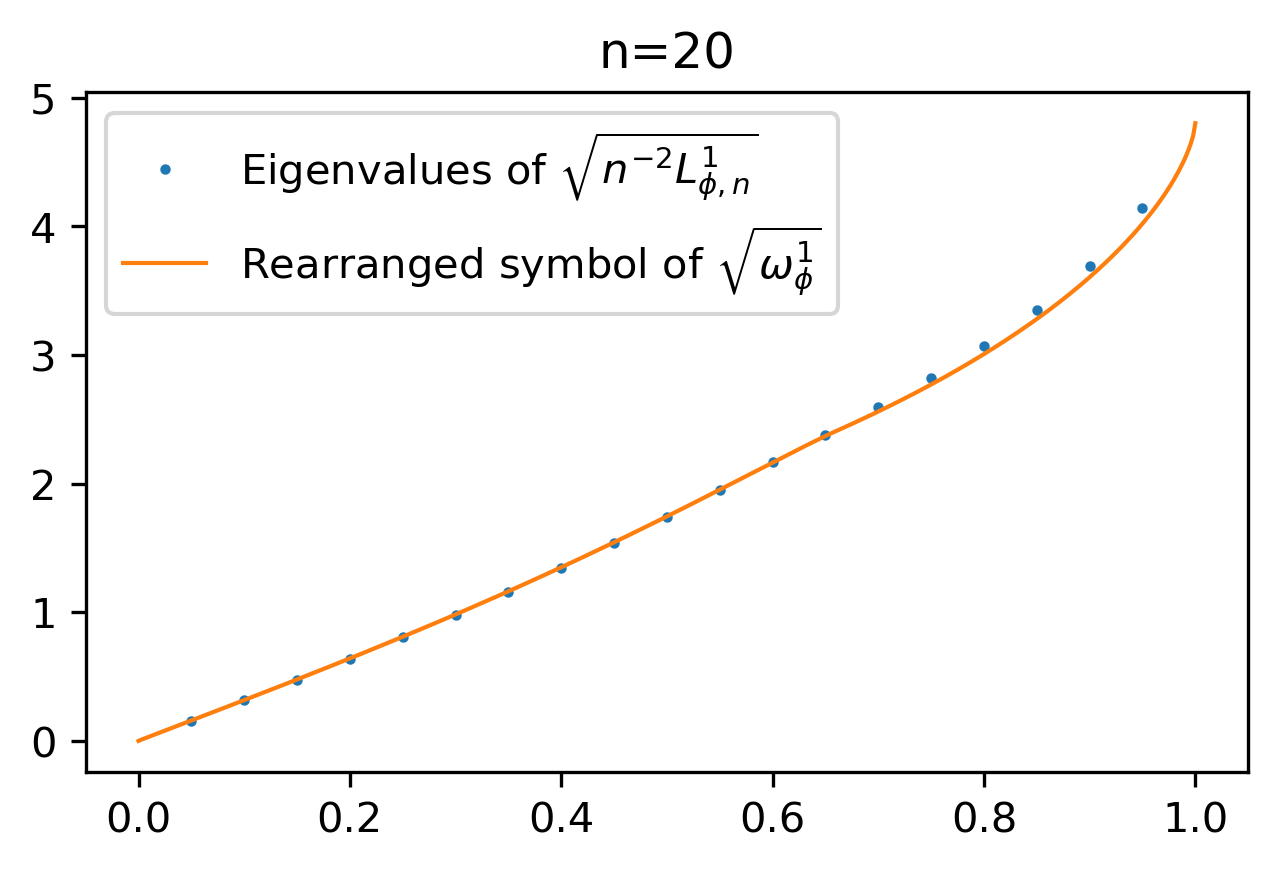}
\endminipage\hfill
\minipage{0.5\textwidth}
  \includegraphics[width=\linewidth]{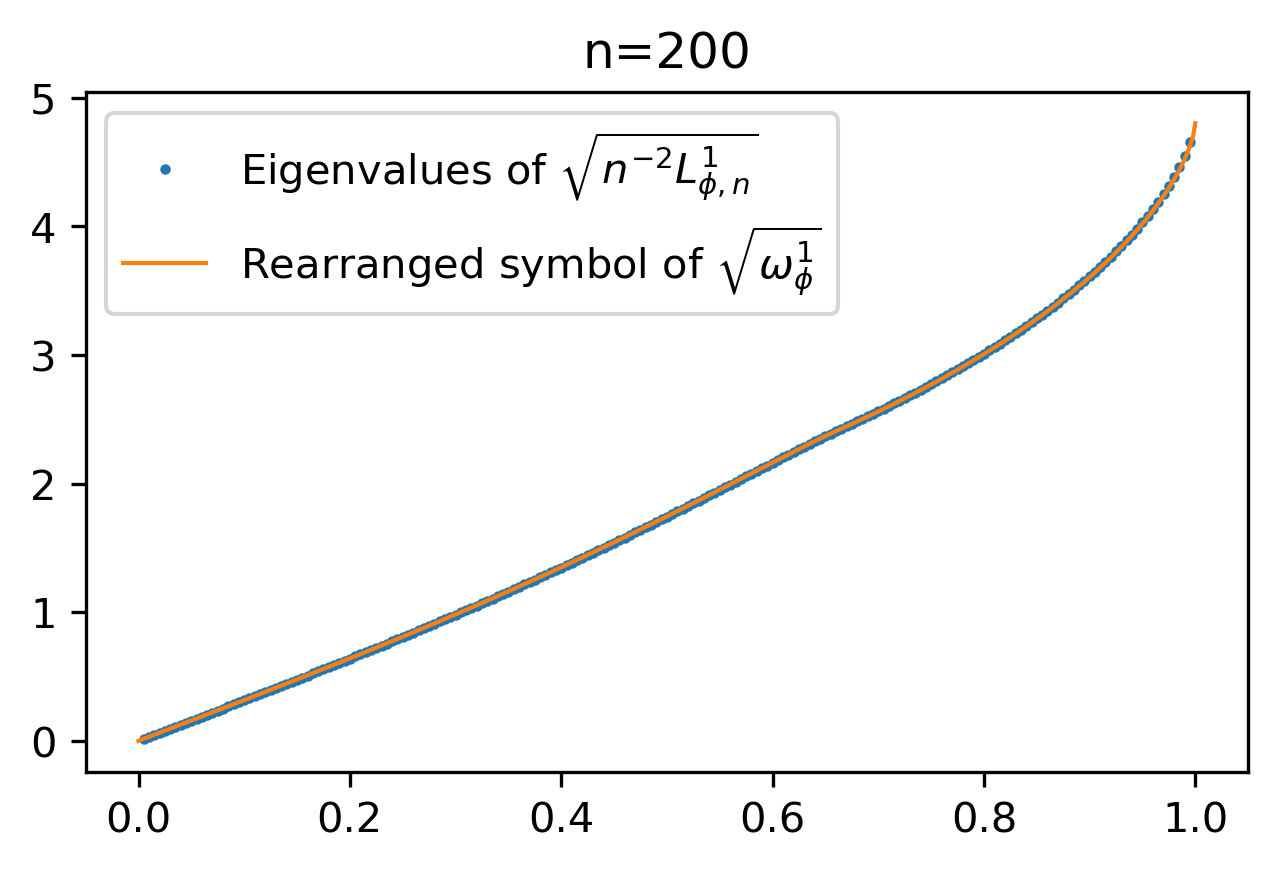}
\endminipage
 \caption{The normalized eigenvalues and the symbol for different values of $n$ and for $\phi_1$.}\label{f3}
\end{figure}
 
\subsection{The Case of Quadratic $B$-splines}
 \hspace{0.5cm} The case $p=2$ is also special, because $OUT(p,n)=0$. Even in this case, we are no longer in the context of classical Lagrange finite elements, yet we observe a uniform gap condition. However, unlike linear $B$-spline, the optimality of the gap is observed only for some specific choice of the reparametrization functions $\phi$. \vspace*{0.25cm}

\noindent  {\bf \em Test 4: study of the gap condition ($\mathbf{p=2}$, and $\phi=\phi_1$):} \vspace*{0.25cm}

\noindent  \hspace{0.5cm}  We begin by employing a reparametrization $\phi = \phi_1$, given by \eqref{eq:reparametrizations-functions-tests}. Following the same approach as in the case of linear $B$-splines, we examine the boundedness of the sequence $(m(n))_{n \geq 1}$ and then the uniform gap condition.  The corresponding plots are presented in Figure \ref{f4-f5}. The figure shows that the sequence $(m(n))_{n \geq 1}$ may not be  bounded, but it  equals $1$ across the majority of indices. However, It is not immediately evident whether a specific value $\gamma$ exists for which $\inf \delta_n^2 \geq \gamma$.  Nonetheless, for numerical testing and practical applications, we can assume that such a $\gamma$ does indeed exist. 
\begin{figure}[H]
\centering
\begin{subfigure}{.5\textwidth}
\centering
\includegraphics[width=\linewidth]{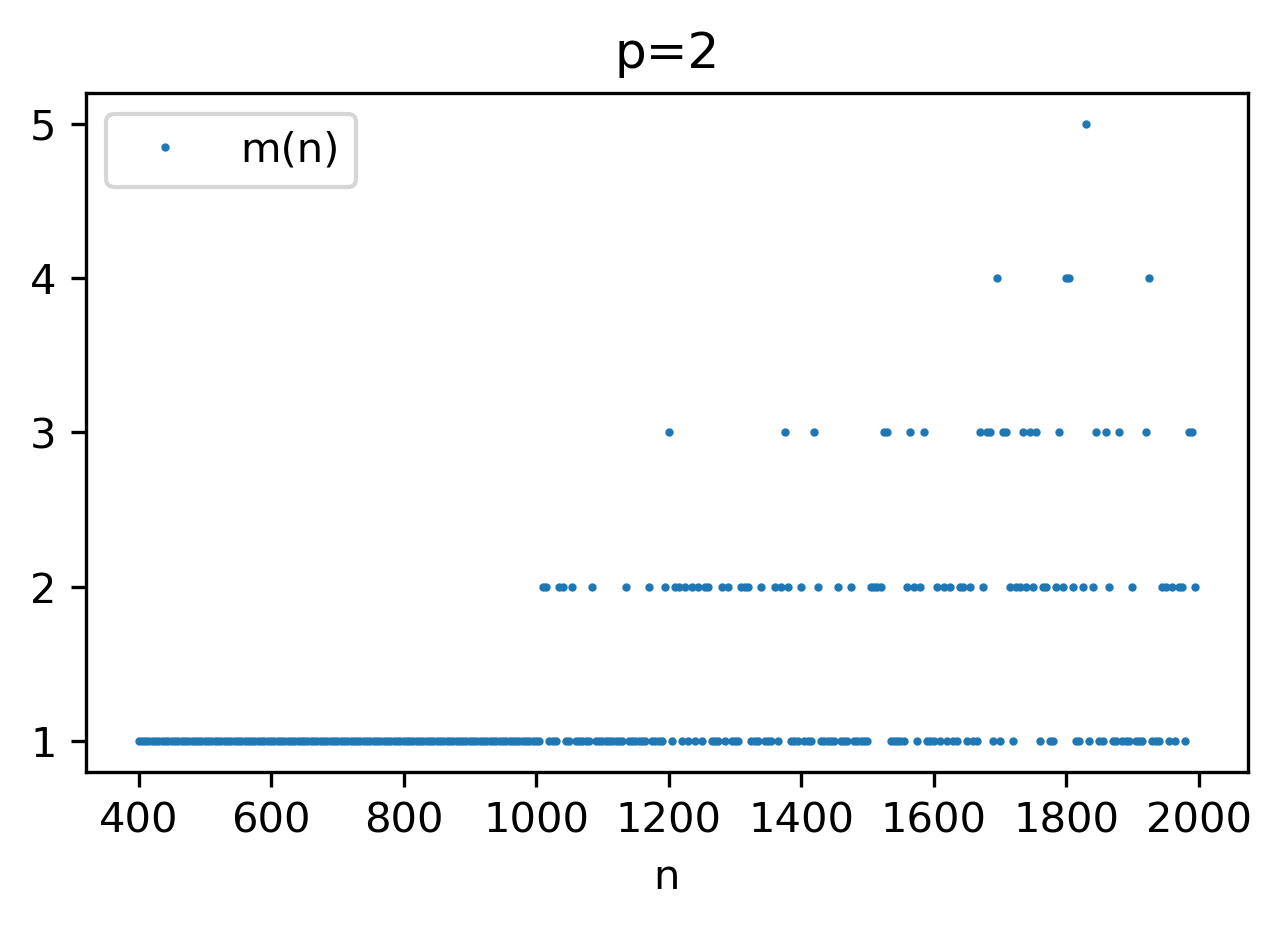}
\caption{The sequence $(m(n))_{n \geq 1}$}
\end{subfigure}%
\begin{subfigure}{.5\textwidth}
\centering
\includegraphics[width=\linewidth]{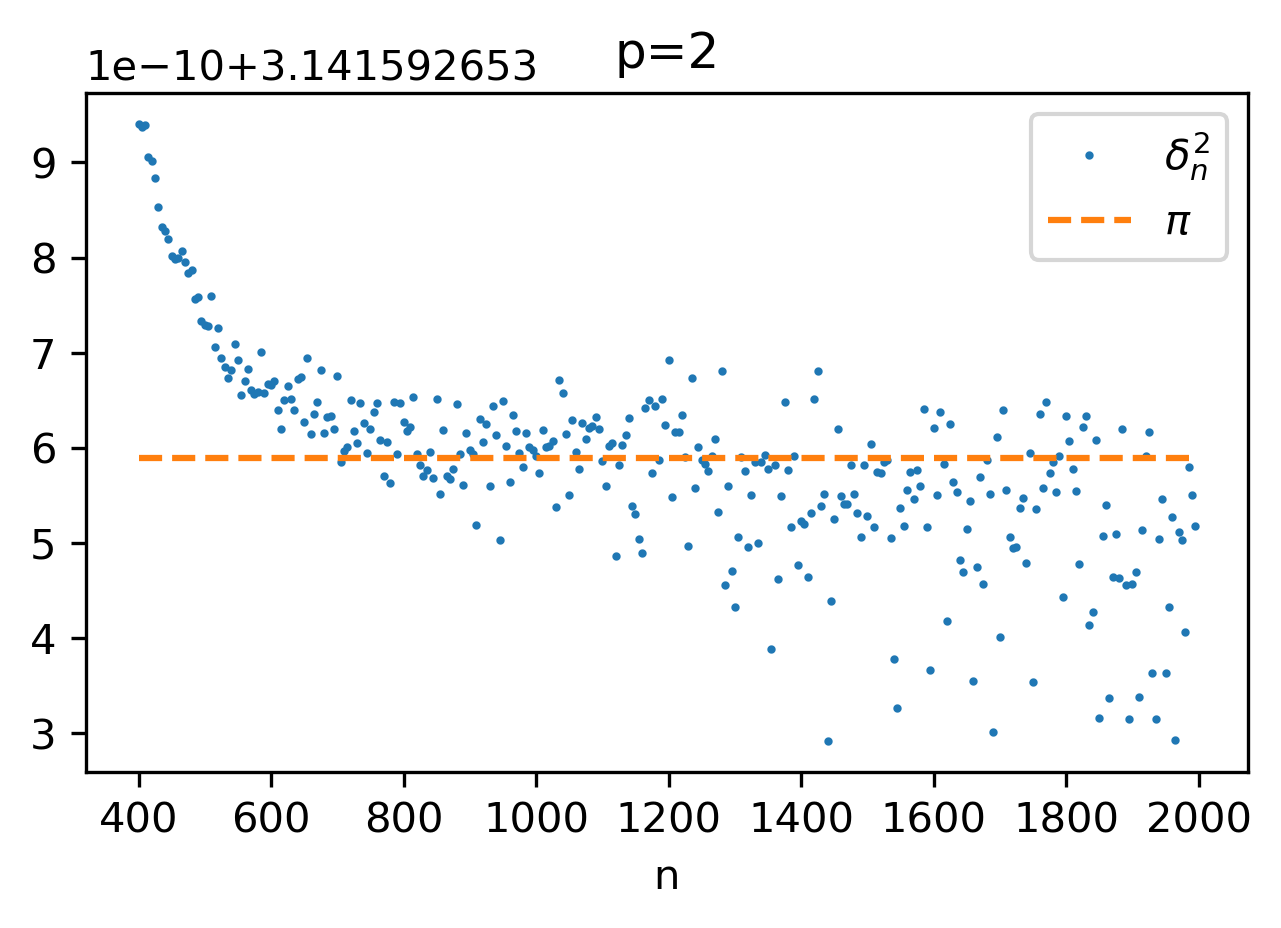}
\caption{The gap $\delta_n^p$}
\end{subfigure}
\caption{The graph of $(m(n))_{n \geq 1}$ (left) and the gap in function of $n$ (right). Parameter values $p=2$, and $\phi=\phi_1$}
\label{f4-f5}
\end{figure}
\noindent  {\bf \em Test 5: study of the gap condition ($\mathbf{p=2}$, and $\phi=\phi_2$):} \vspace*{0.25cm}

\noindent  \hspace{0.5cm} We follow the same rationale as the last test, employing the reparametrization $\phi_2$ as defined in \eqref{eq:reparametrizations-functions-tests}. The results are shown in Figure \ref{f6-f7}. We remark that, contrary to the case $\phi=\phi_1$, the index $m(n)$ follows a linear pattern and that now it is more clear that the gap $\delta_n^2$ is uniformly bounded. It is important to mention here that  the major difference between $\phi_1$ and $\phi_2$ lies in their convexity.
\begin{figure}[H]
\centering
\begin{subfigure}{.5\textwidth}
\centering
\includegraphics[width=\linewidth]{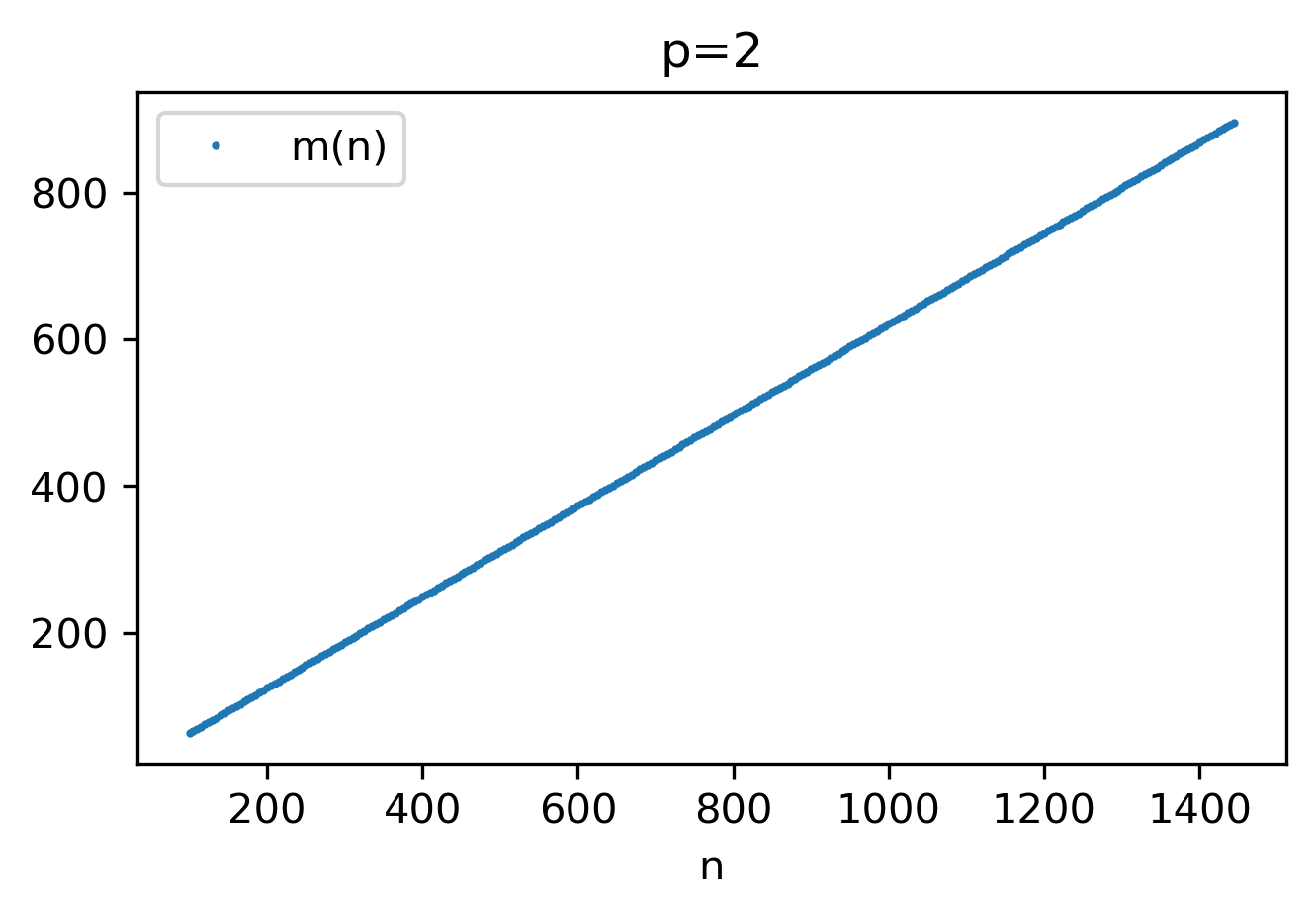}
\caption{The sequence $(m(n))_{n \geq 1}$}
\end{subfigure}%
\begin{subfigure}{.5\textwidth}
\centering
\includegraphics[width=\linewidth]{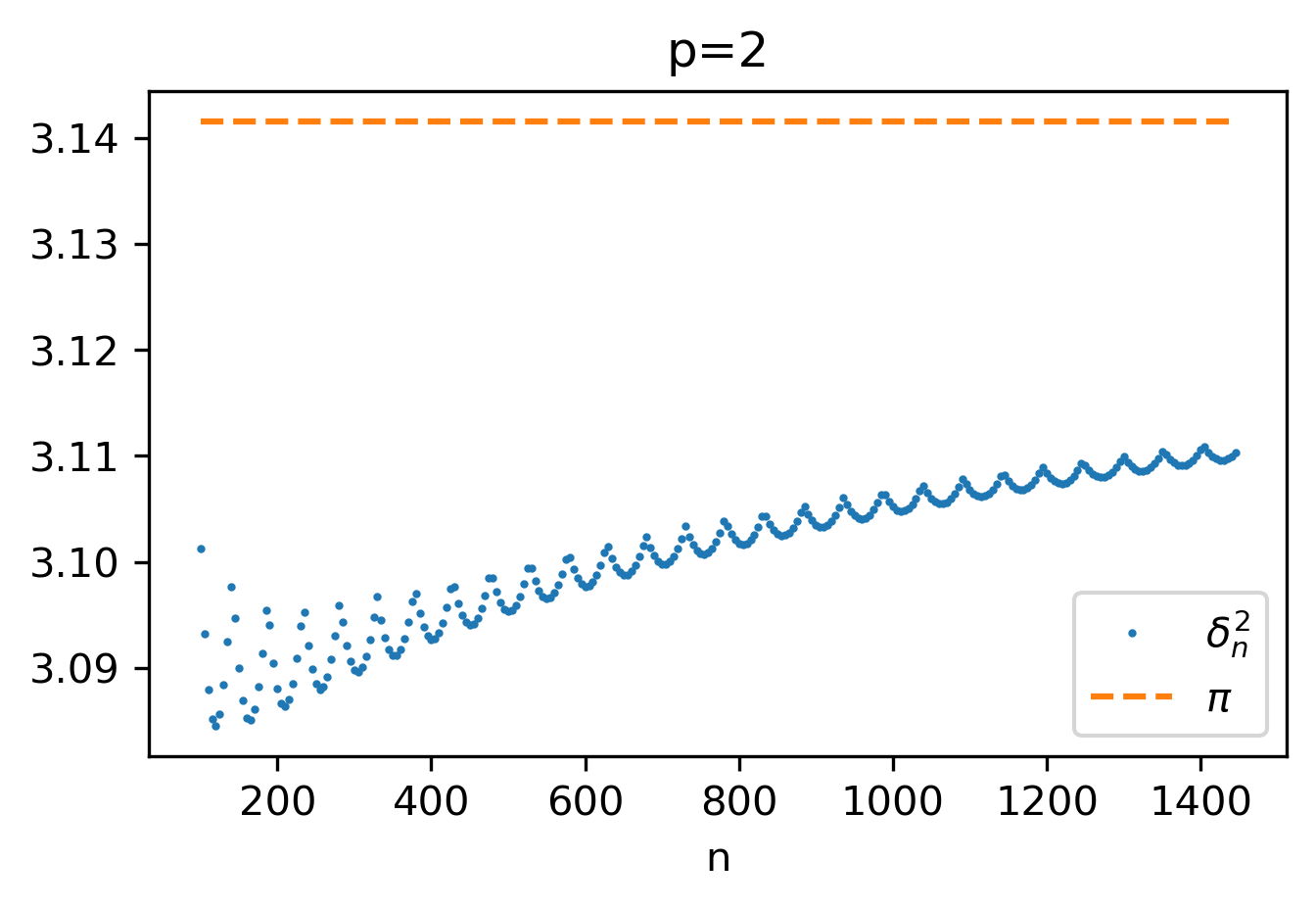}
\caption{The gap $\delta_n^p$}
\end{subfigure}
\caption{The graph of $(m(n))_{n \geq 1}$ (left) and the gap in function of $n$ (right). Parameter values $p=2$, and $\phi=\phi_2$}
\label{f6-f7}
\end{figure}
\noindent  {\bf \em Test 6: study of the optimal gap condition ($\mathbf{p=2}$, and $\phi=\phi_3$):} \vspace*{0.25cm}

\noindent  \hspace{0.5cm}  In the last test, we use the strictly concave function $\phi_3$, as illustrated in the top plots of Figure \ref{f8-f11}. Remarkably, in this case, we observe the optimal gap condition. In addition, we remark that the behavior of the sequence $(m(n))_{n \geq 1}$ and the gap is similar to that of the case when $p=1$. However, if we modify $\theta$, as shown in the bottom plots of Figure \ref{f8-f11}, both $(m(n))_{n \geq 1}$ and the gap exhibit behavior similar to the case when $\phi=\phi_1$.\\

In summary, based on multiple numerical tests, we observed that concave reparametrizations prove to be more effective in restoring the optimal gap condition when $p=2$. On the contrary, employing convex reparametrizations leads to losing the optimal gap, even though a uniform gap is observed. Additionally, the sequence $(m(n))_{n \geq 1}$ is unbounded when convex reparametrizations are employed.
\begin{figure}[H]
\centering
\begin{subfigure}{.5\textwidth}
\centering
\includegraphics[width=\linewidth]{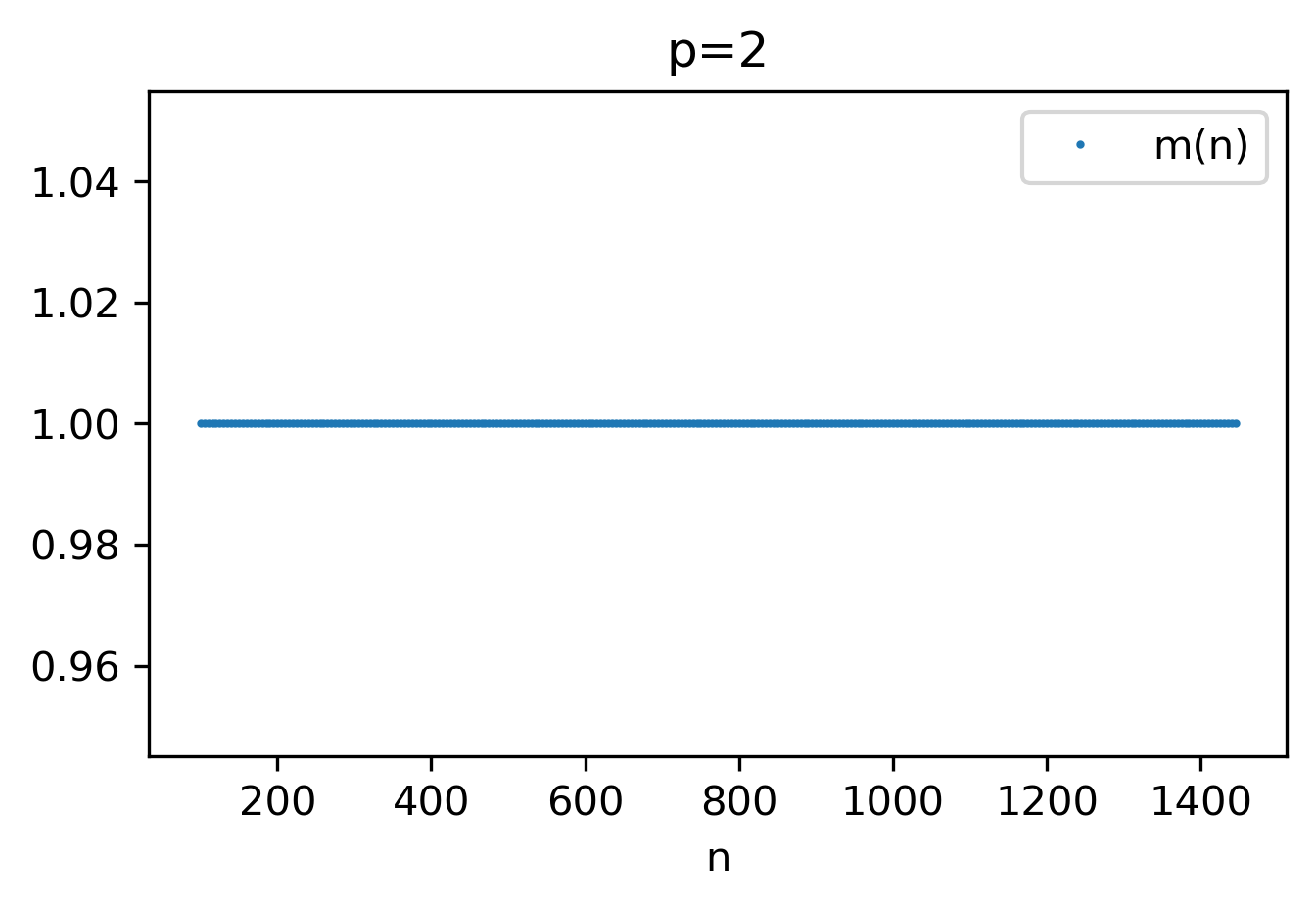}
\caption{The sequence $(m(n))_{n \geq 1}$ ($\theta=0.01$)}
\end{subfigure}%
\begin{subfigure}{.5\textwidth}
\centering
\includegraphics[width=\linewidth]{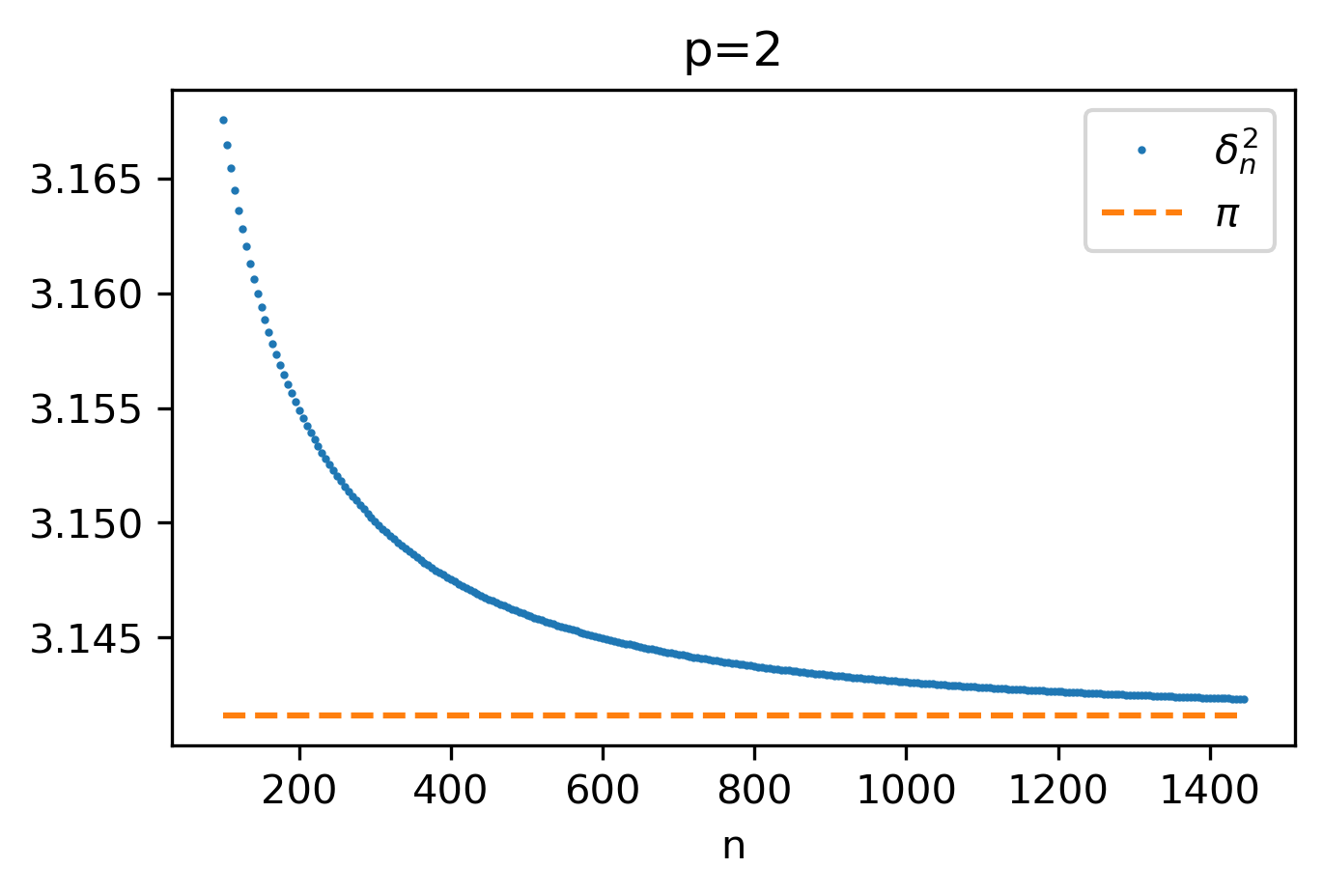}
\caption{The gap $\delta_n^p$ ($\theta=0.01$)}
\end{subfigure}
\begin{subfigure}{.5\textwidth}
\centering
\includegraphics[width=\linewidth]{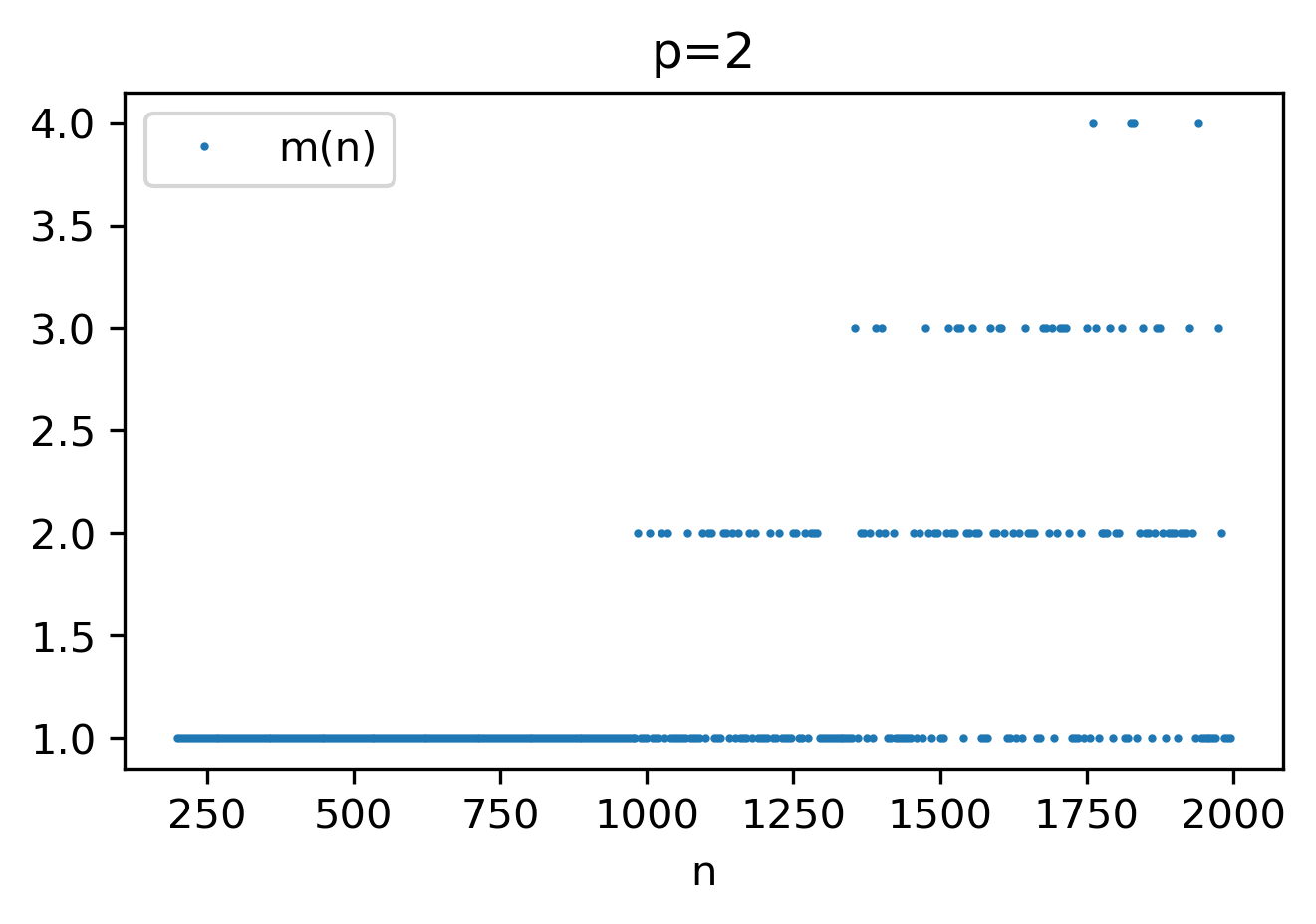}
\caption{The sequence $(m(n))_{n \geq 1}$ ($\theta=1$)}
\end{subfigure}%
\begin{subfigure}{.5\textwidth}
\centering
\includegraphics[width=\linewidth]{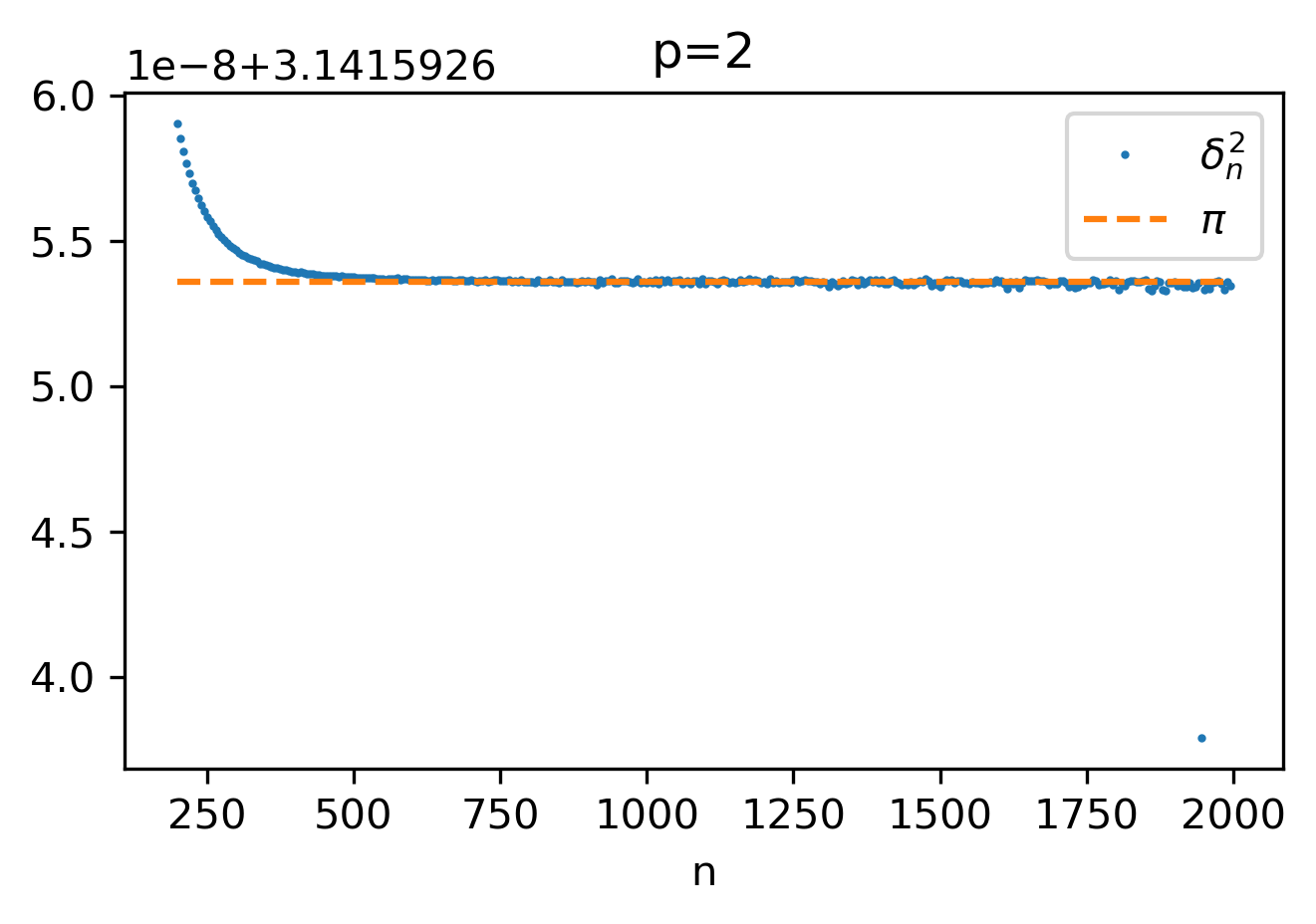}
\caption{The gap $\delta_n^p$ ($\theta=1$)}
\end{subfigure}
\caption{The graph of $(m(n))_{n \geq 1}$ (left) and the gap in function of $n$ (right). Parameter values $p=2$, and $\phi=\phi_3$}
\label{f8-f11}
\end{figure}

\subsection{The Case when $p \geq 3$}

 \hspace{0.5cm} Now we examine the gap behavior when $p \geq 3$. We aim to determine whether it exhibits similar characteristics to the case when $p=2$. \vspace*{0.25cm} 

\noindent  {\bf \em Test 7: study of the gap condition ($\mathbf{p \geq 3}$, and different values of $\phi$):} \vspace*{0.25cm} 

\noindent  \hspace{0.5cm}  Our analysis begins with the examination of the reparametrizations $\phi_1$ and $\phi_2$, and the corresponding results are given in  Figures \ref{f12-f13-}--\ref{f14-f15-}. These choices of reparametrizations yield a similar behavior of $(m(n))$ to that of $p=2$ with $\phi_2$ (Fig. \ref{f6-f7}). Namely, the sequence $(m(n))_{n \geq 1}$ is unbounded. The major difference is that for $p=3$ the uniform gap condition is lost.

However, a distinct behavior in terms of gap and index $m(n)$ compared to the case $p=2$ is observed when we use the concave reparametrization $\phi_3$, as depicted in Figure  \ref{f16-f17-}. In this case, it appears that the sequence $(m(n))_{n\geq 1}$ does not seem to be bounded, and notably, it is not equal to $1$ across the majority of indices. In addition, even though the use of $\phi_3$ restores the uniform gap condition, it is not the optimal one, as was the case for $p=2$.\\

After analyzing the behavior of the gap $\delta^p_n$ and the sequence $(m(n))_{n \geq 1}$ for values of $p$ in $\{1,2,3\}$, we now turn our attention to the case where $p \geq 4$. In this case, we developed several tests (not detailed here) and they showed that the behaviors of $\delta^p_n$ and $m(n)$ are similar to the case when $p=3$. In particular, when we reparametrize using either $\phi_1$ or $\phi_2$, we lose the property of a uniform gap and the condition on $(m(n))_{n \geq 1}$. However, when we use $\phi_3$ with $\theta=0.01$, an interesting observation emerges. We observe that when the values of $p$ increase, the point at which the gap becomes less than $\pi$ approaches zero. These observations raise the question of whether there exists a reparametrization that can yield the optimal gap condition. It is crucial to numerically verify if such a reparametrization can be found and determine if the required hypothesis on $(m(n))_{n \geq 1}$ holds. \vspace*{0.25cm}
\begin{figure}[H]
\centering
\begin{subfigure}{.5\textwidth}
\centering
\includegraphics[width=\linewidth]{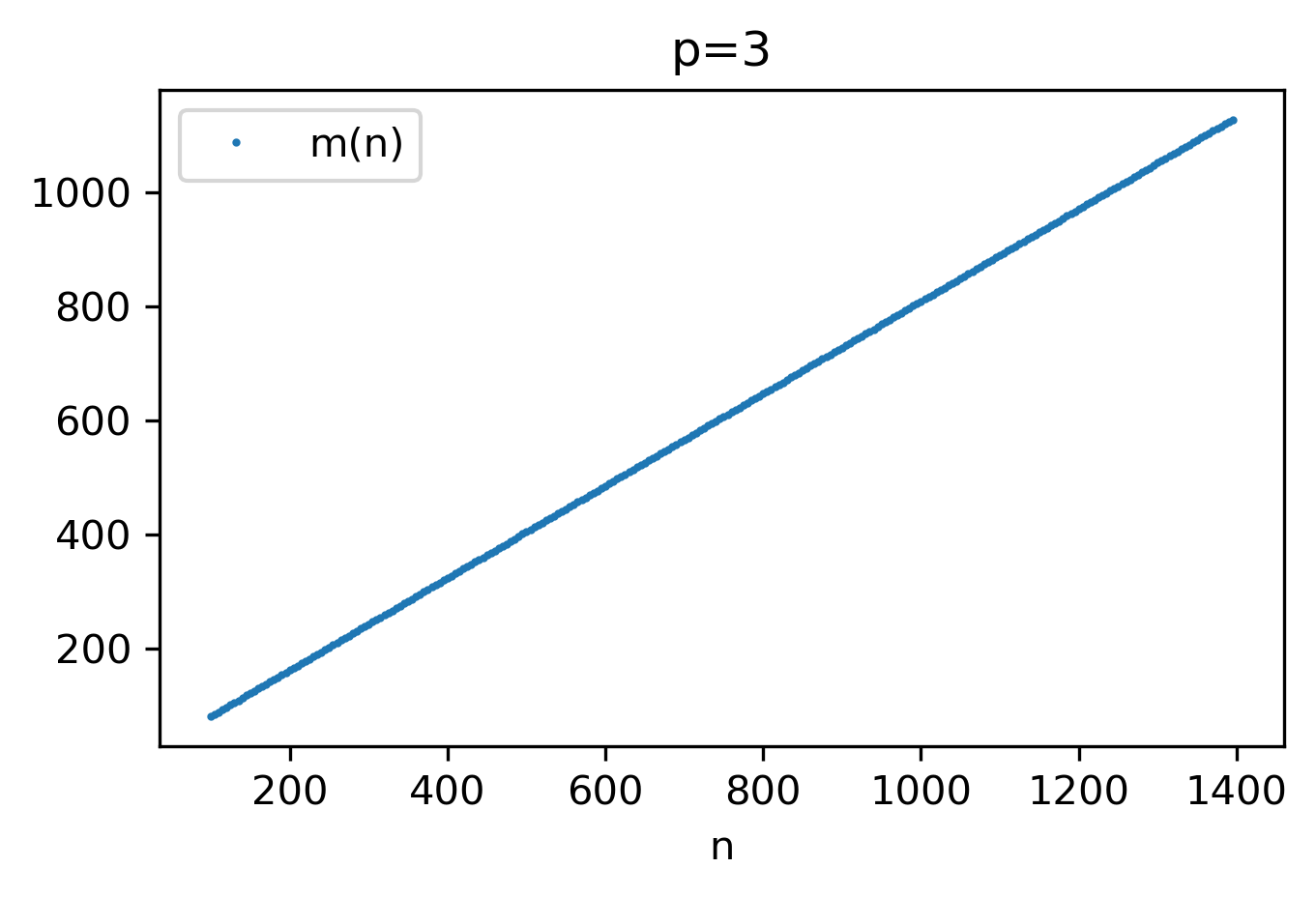}
\caption{The sequence $(m(n))_{n \geq 1}$.}
\end{subfigure}%
\begin{subfigure}{.5\textwidth}
\centering
\includegraphics[width=\linewidth]{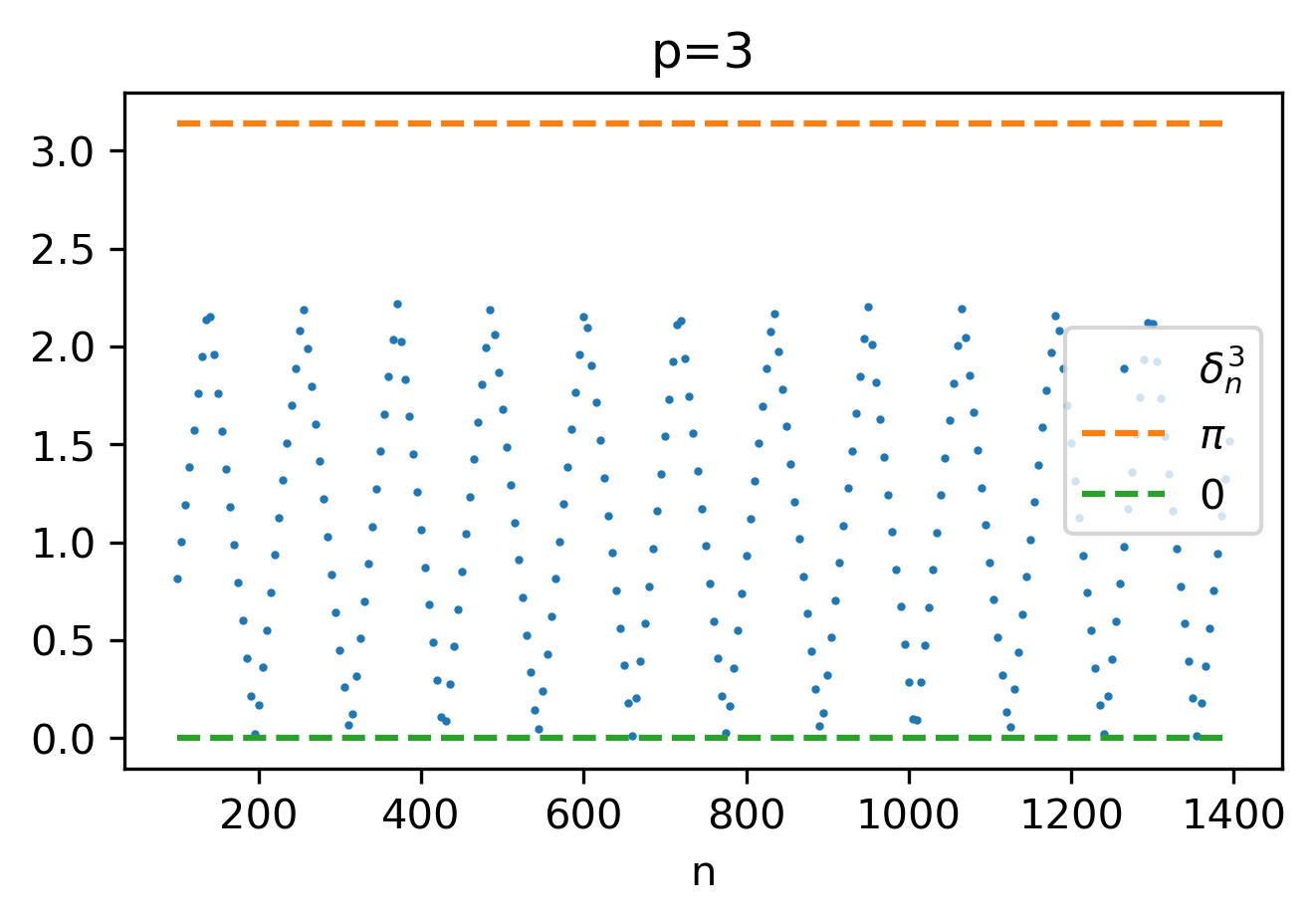}
\caption{The gap $\delta_n^p$}
\end{subfigure}
\caption{The graph of $(m(n))_{n \geq 1}$ (left) and the gap in function of $n$ (right). Parameter values $p=3$, and $\phi=\phi_1$}
\label{f12-f13-}
\end{figure}

\begin{figure}[H]
\centering
\begin{subfigure}{.5\textwidth}
\centering
\includegraphics[width=\linewidth]{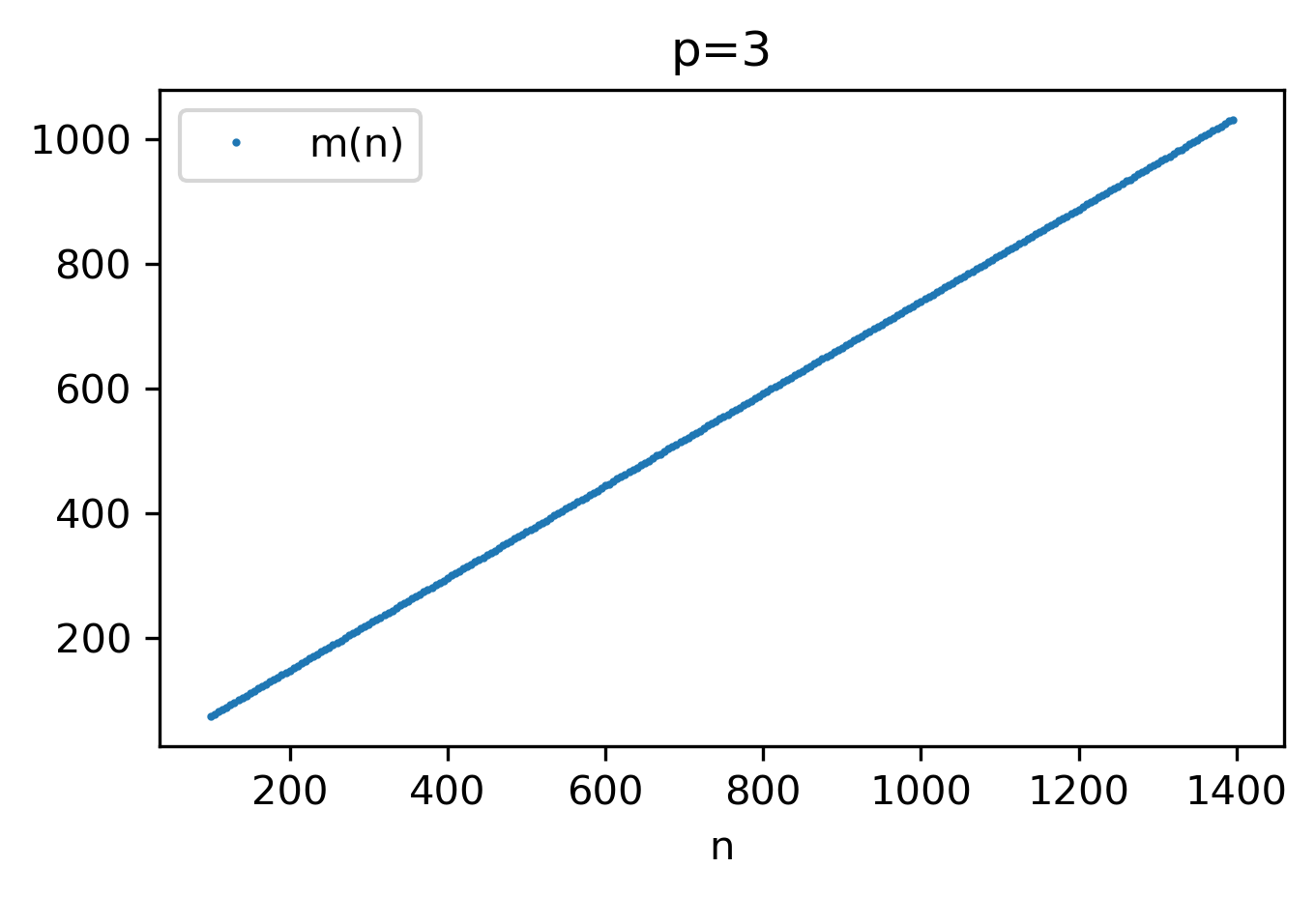}
\caption{The sequence $(m(n))_{n \geq 1}$}
\end{subfigure}%
\begin{subfigure}{.5\textwidth}
\centering
\includegraphics[width=\linewidth]{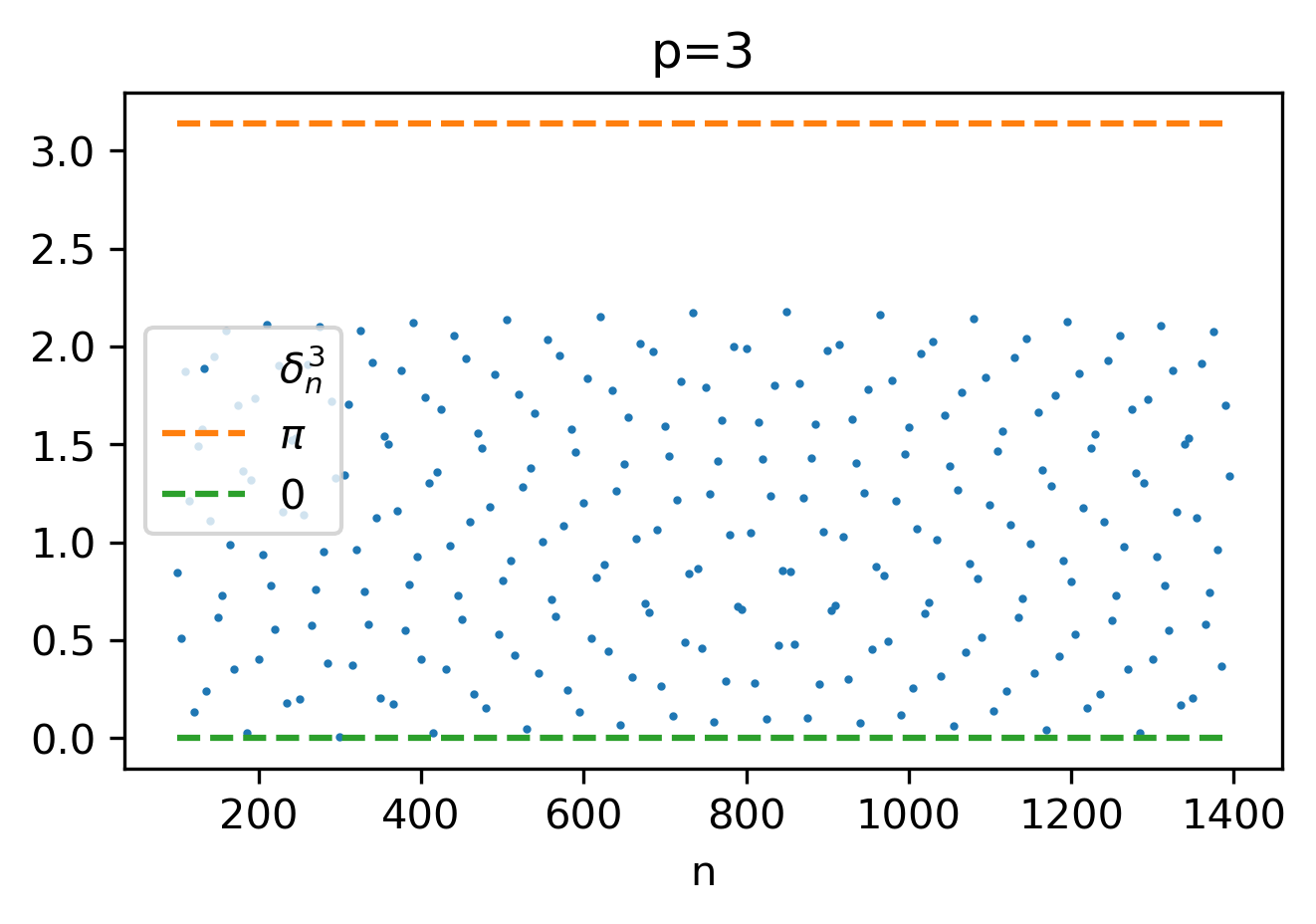}
\caption{The gap $\delta_n^p$}
\end{subfigure}
\caption{The graph of $(m(n))_{n \geq 1}$ (left) and the gap in function of $n$ (right). Parameter values $p=3$, and $\phi=\phi_2$}
\label{f14-f15-}
\end{figure}

\begin{figure}[H]
\centering
\begin{subfigure}{.5\textwidth}
\centering
\includegraphics[width=\linewidth]{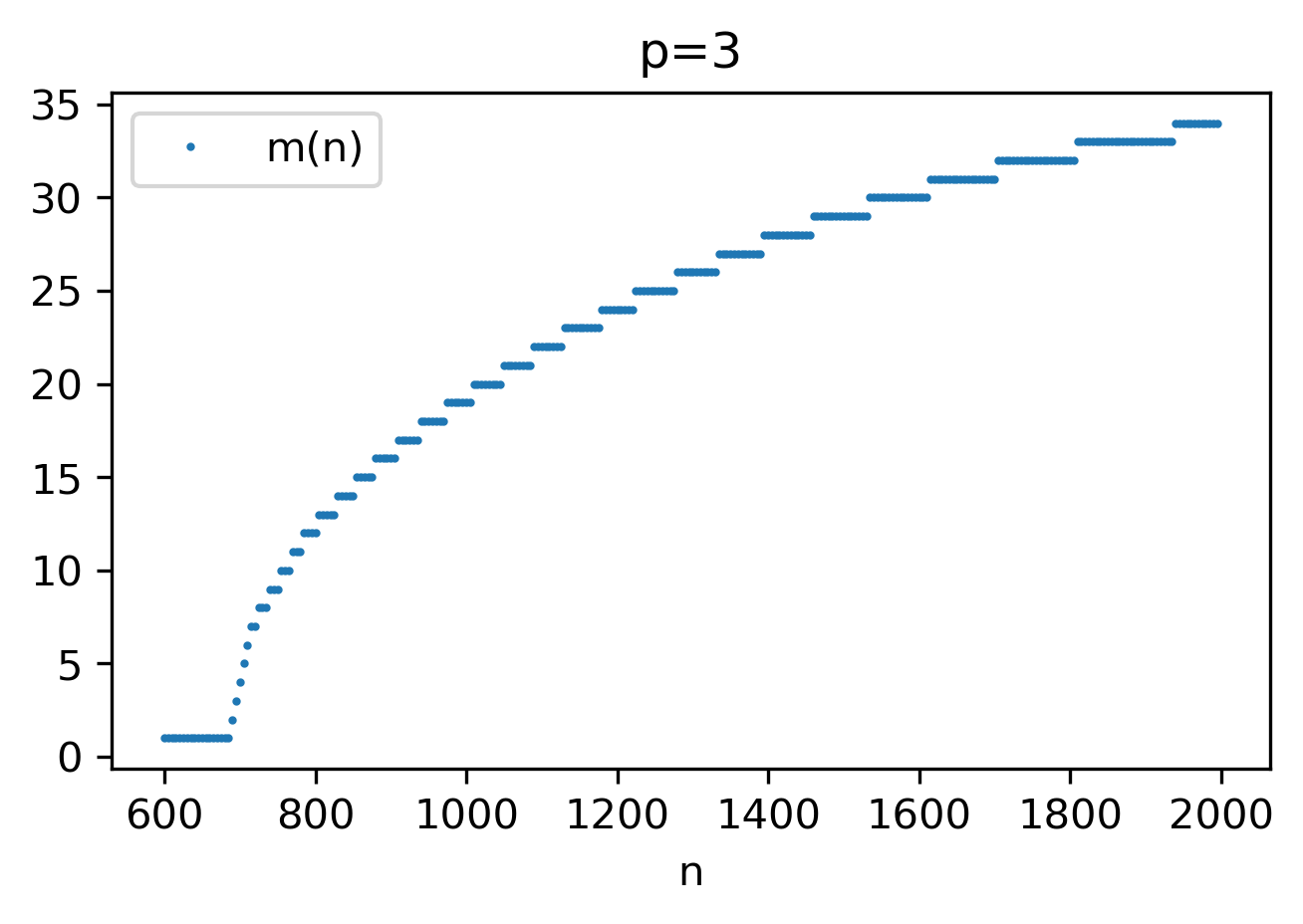}
\caption{The sequence $(m(n))_{n \geq 1}$}
\end{subfigure}%
\begin{subfigure}{.5\textwidth}
\centering
\includegraphics[width=\linewidth]{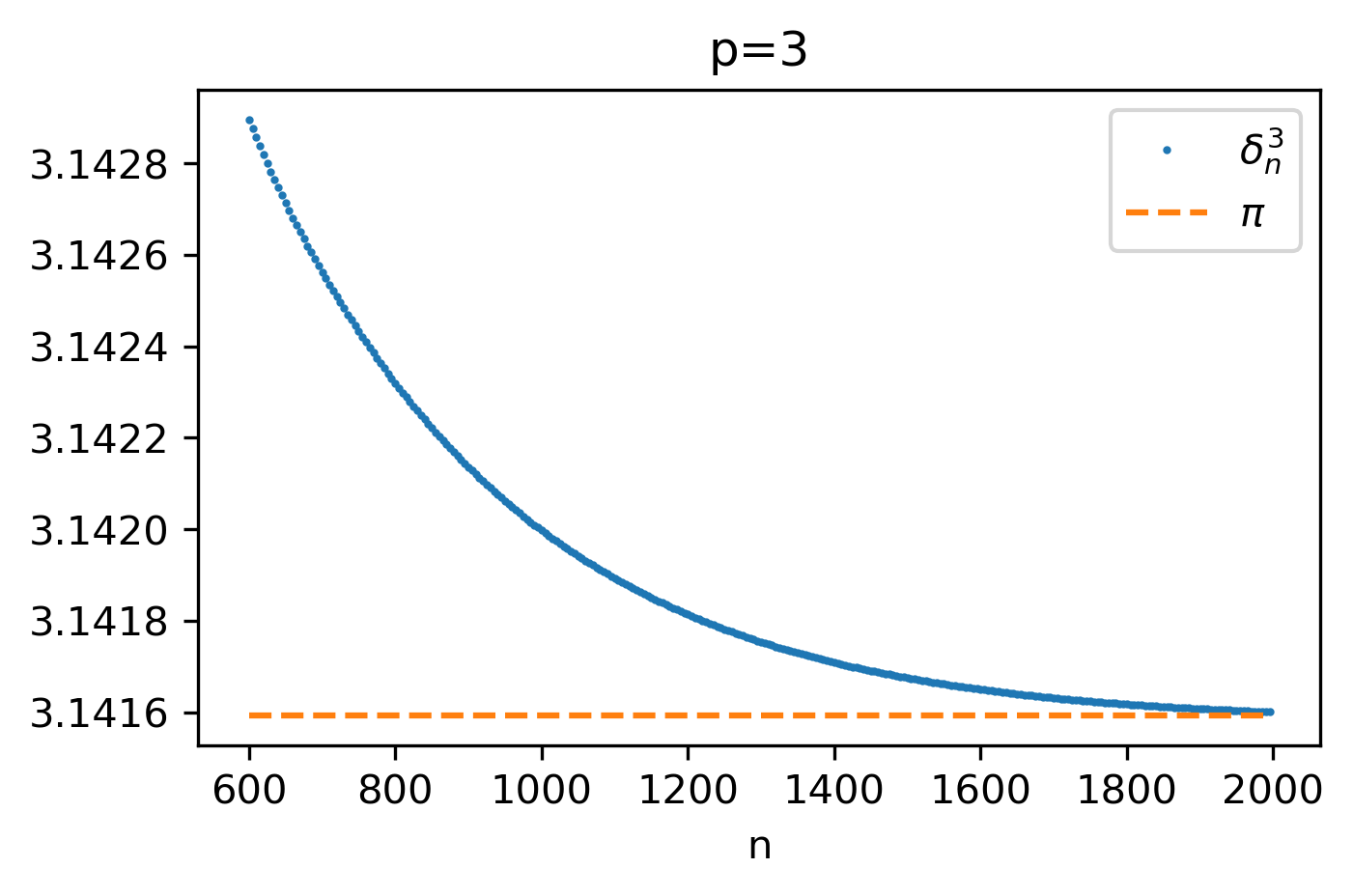}
\caption{The gap $\delta_n^p$}
\end{subfigure}
\begin{subfigure}{.5\textwidth}
\centering
\includegraphics[width=\linewidth]{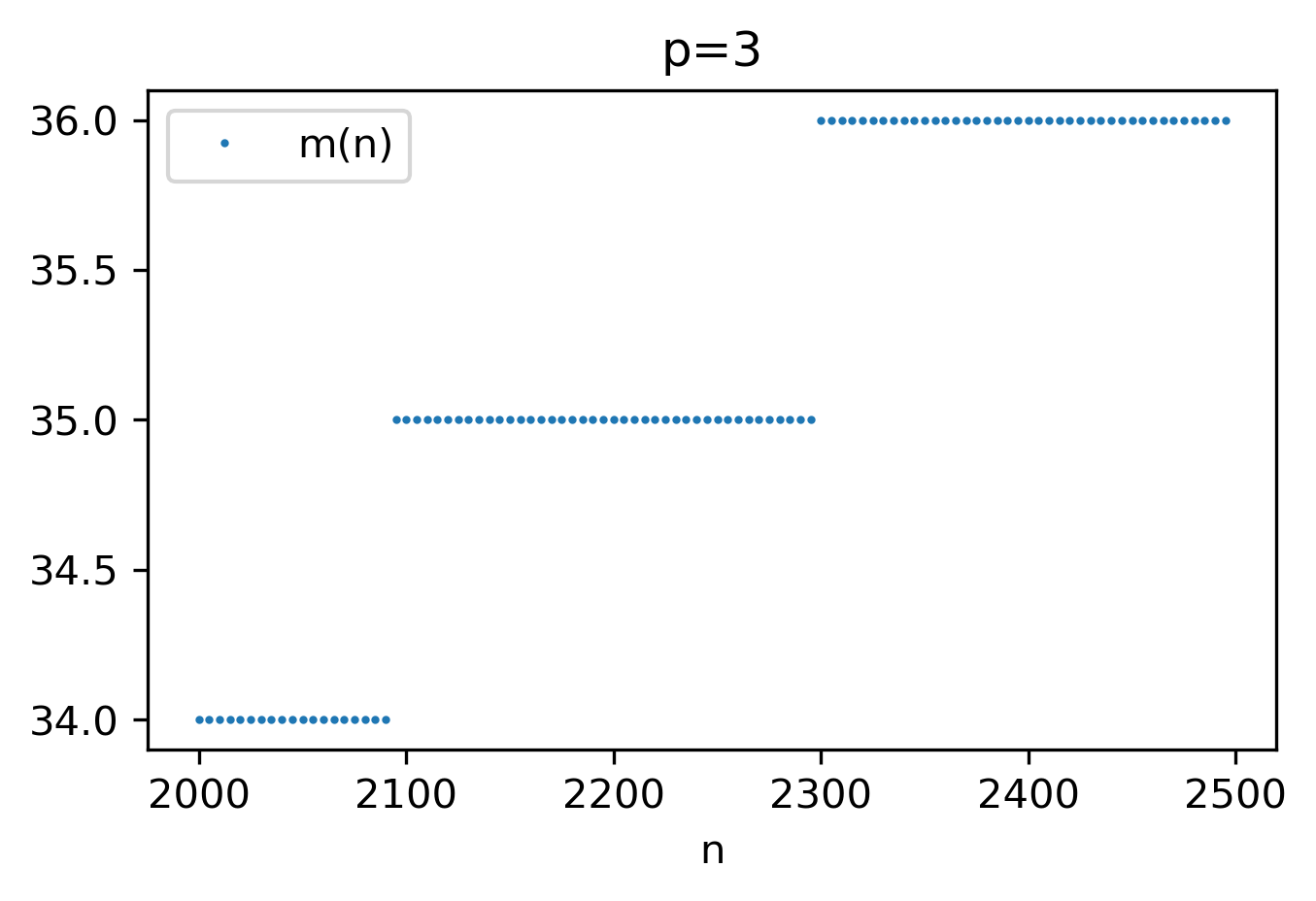}
\caption{The sequence $(m(n))_{n \geq 1}$ }
\end{subfigure}%
\begin{subfigure}{.5\textwidth}
\centering
\includegraphics[width=\linewidth]{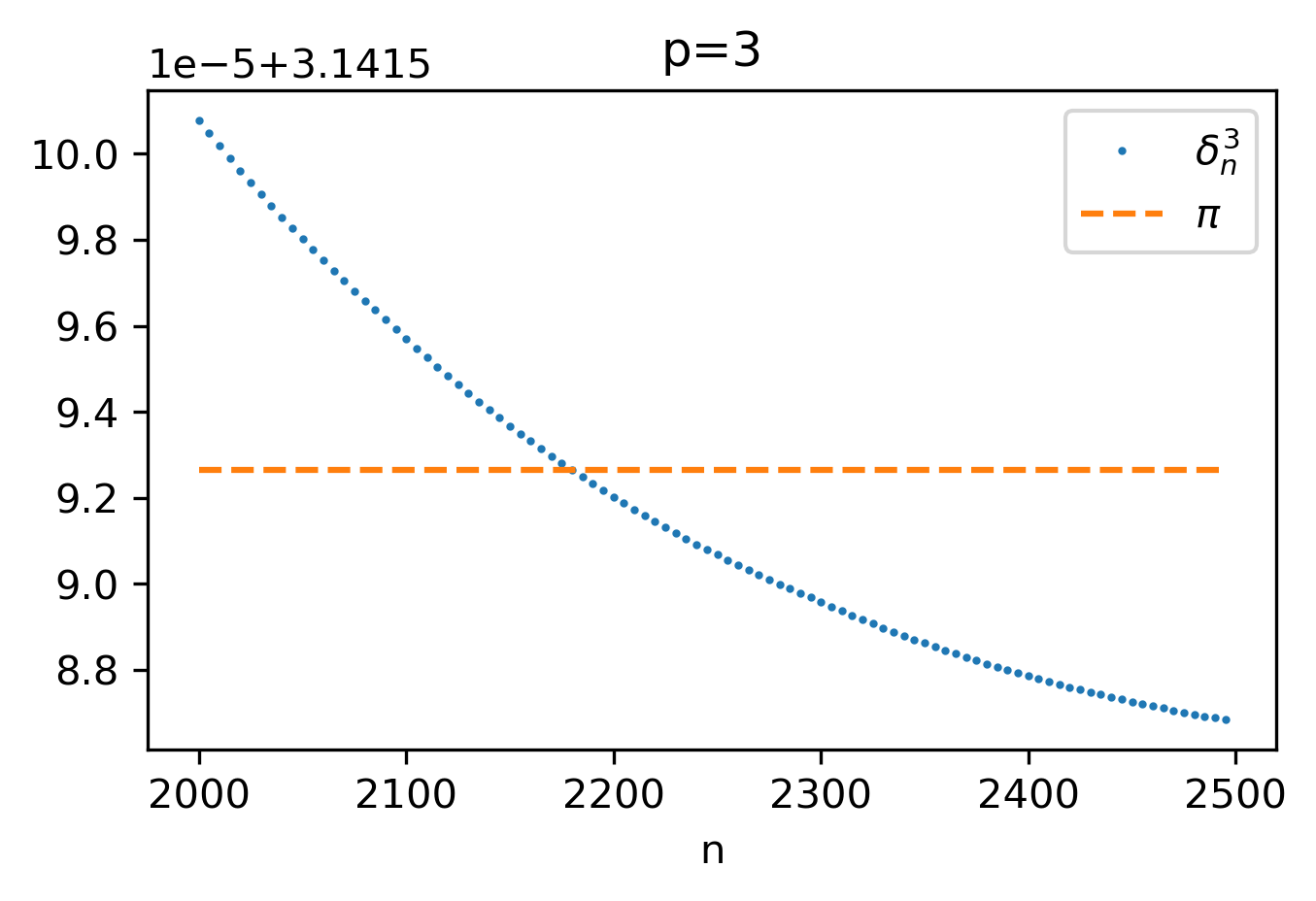}
\caption{The gap $\delta_n^p$}
\end{subfigure}
\caption{The graph of $(m(n))_{n \geq 1}$ (left) and the gap in function of $n$ (right). Parameter values: $p=3$, and $\phi=\phi_3$ with $\theta=0.01$}
\label{f16-f17-}
\end{figure}
\noindent  {\bf \em Test 8: study of the optimal gap condition ($\mathbf{p \geq 3}$):}  

\noindent  \hspace{0.5cm} In this test, we present numerical evidence showing that a simple modification to $\phi_3$ allows us to achieve a bounded sequence$(m(n))_{n \geq 1}$ and satisfy the optimal gap condition for all values of $p \in \{3, 4, \dots, 27\}$. Let
$$
\Phi_p(x)=(\phi_3(x))^{\frac{1}{p}},\quad \forall x\in [0,1].
$$

In Figures \ref{f15-f16}--\ref{f19-f20}, we can observe that for $p\in\{3,4,5\}$ the gap behavior is identical to that of the case $p=1$ with $\phi_3$ and $\theta=0.01$. This means that we achieve the optimal gap and observe the same pattern in the index $m(n)$ as in the case $p=1$. It is worth noting that this behavior is in accordance with Corollary \ref{cl2.5} since the number of square root eigenvalues follows the numerical convexity of our symbol. However, our numerical tests indicate that the distance between successive square root eigenvalues is increasing suggesting that the index at which the gap occurs is $m_0=1$.

Furthermore, for $p\in\{6,\dots ,27\}$ (not detailed here), our tests reveal the same gap behavior and sequence pattern $(m(n))_{n \geq 1}$ as observed in the case $p=1$.

Additionally, when $p=2$ and utilizing the reparametrization $\Phi_2$, we observe a behavior consistent with cases $p\in\{1,3,4,\dots,27\}$, as illustrated in Figure \ref{f21-f22}.
\begin{figure}[H]
\centering
\begin{subfigure}{.5\textwidth}
\centering
\includegraphics[width=\linewidth]{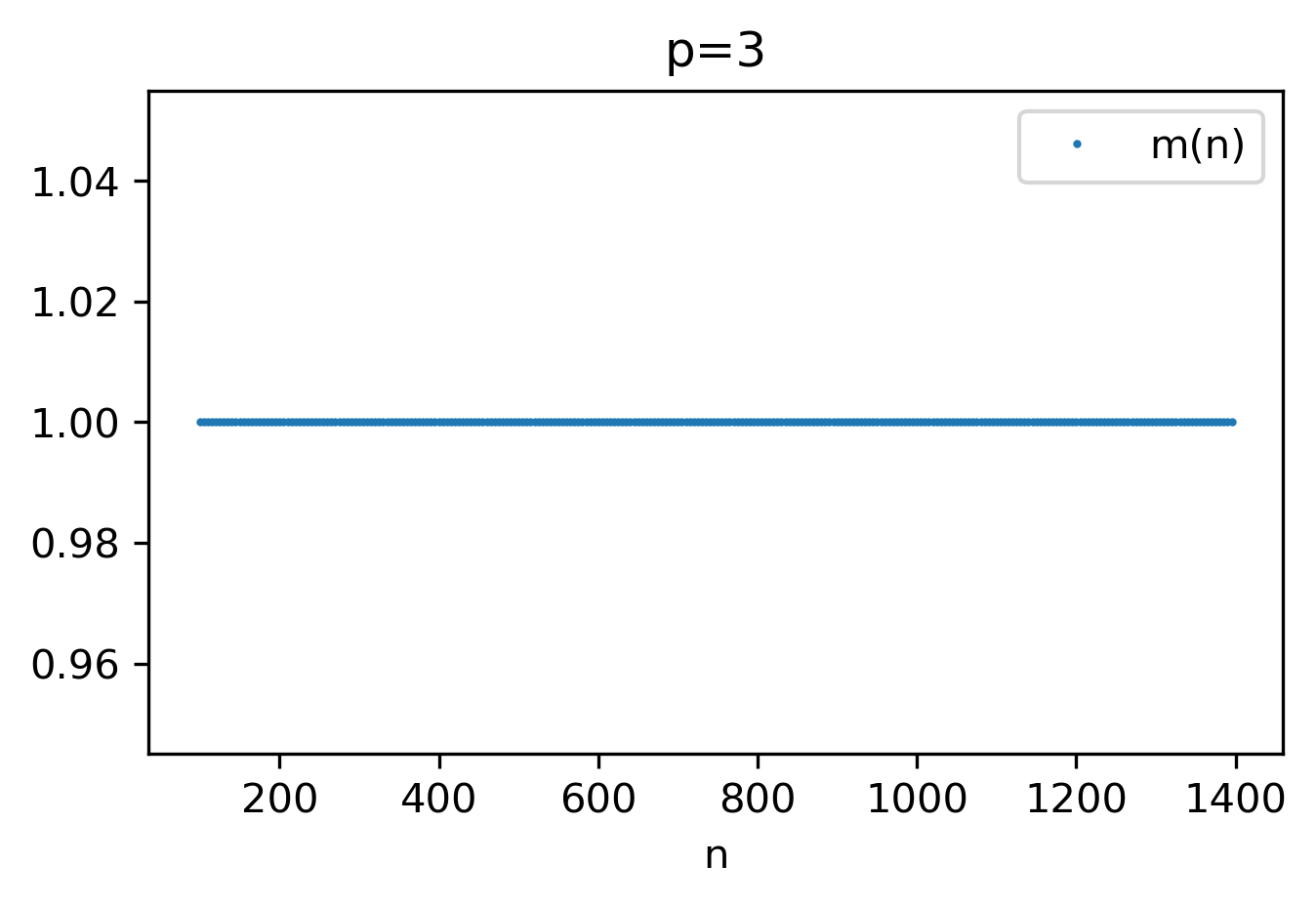}
\caption{The sequence $(m(n))_{n \geq 1}$}
\end{subfigure}%
\begin{subfigure}{.5\textwidth}
\centering
\includegraphics[width=\linewidth]{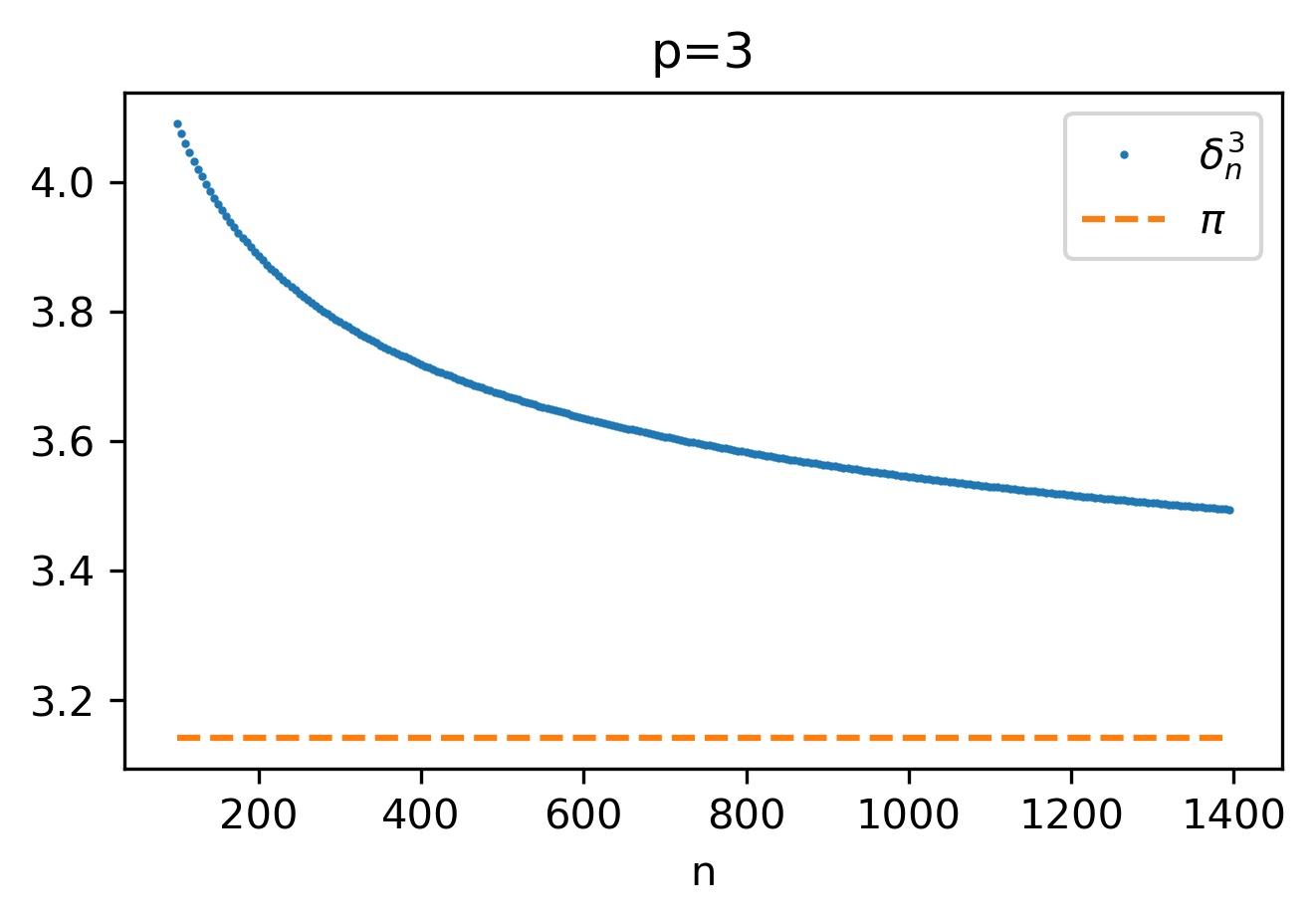}
\caption{The gap $\delta_n^p$}
\end{subfigure}
\caption{The graph of $(m(n))_{n \geq 1}$ (left) and the gap in function of $n$ (right). Parameter values $p=3$, and $\phi=\Phi_3$ with $\theta = 0.01$}
\label{f15-f16}
\end{figure}

\begin{figure}[H]
\centering
\begin{subfigure}{.5\textwidth}
\centering
\includegraphics[width=\linewidth]{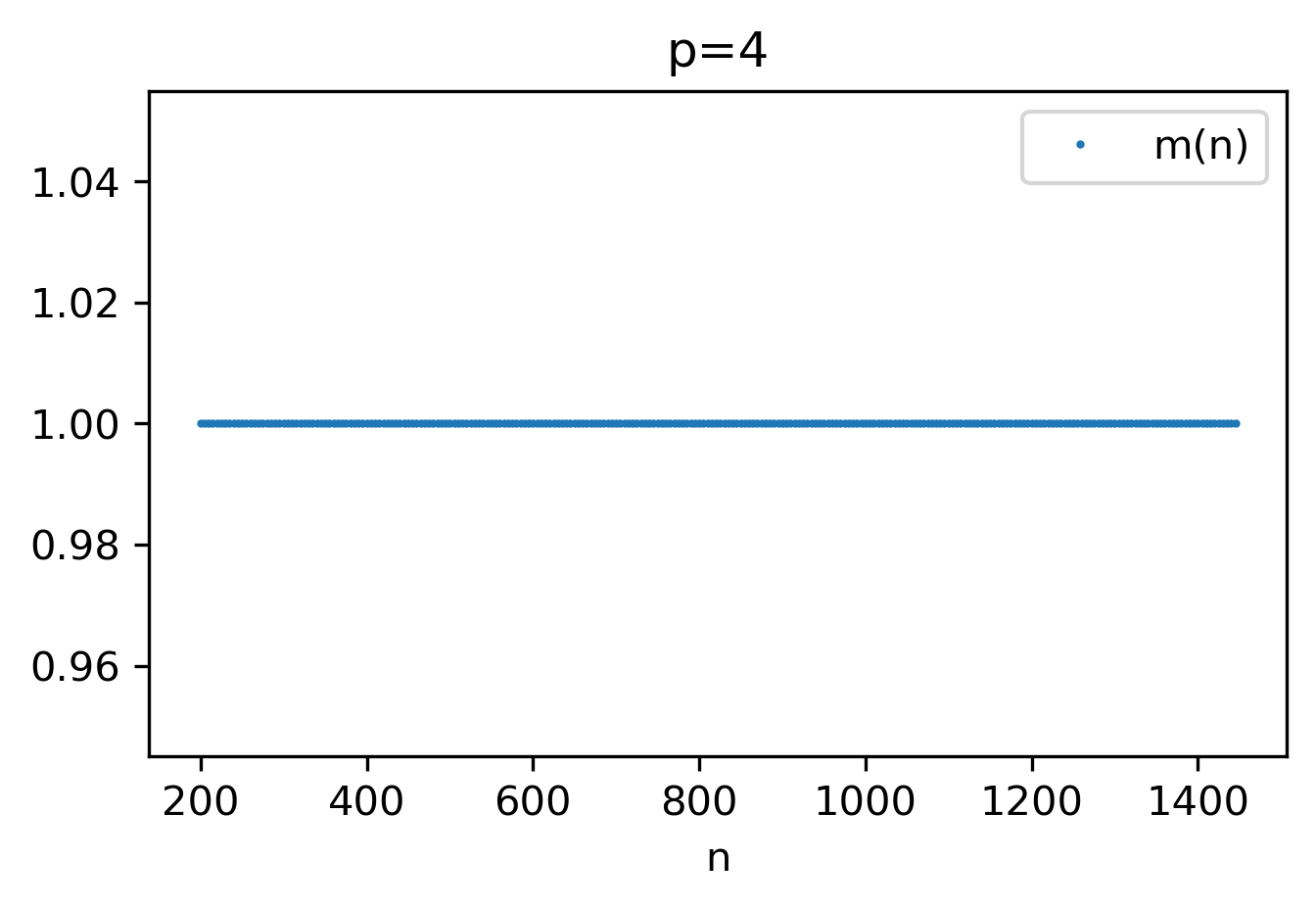}
\caption{The sequence $(m(n))_{n \geq 1}$}
\end{subfigure}%
\begin{subfigure}{.5\textwidth}
\centering
\includegraphics[width=\linewidth]{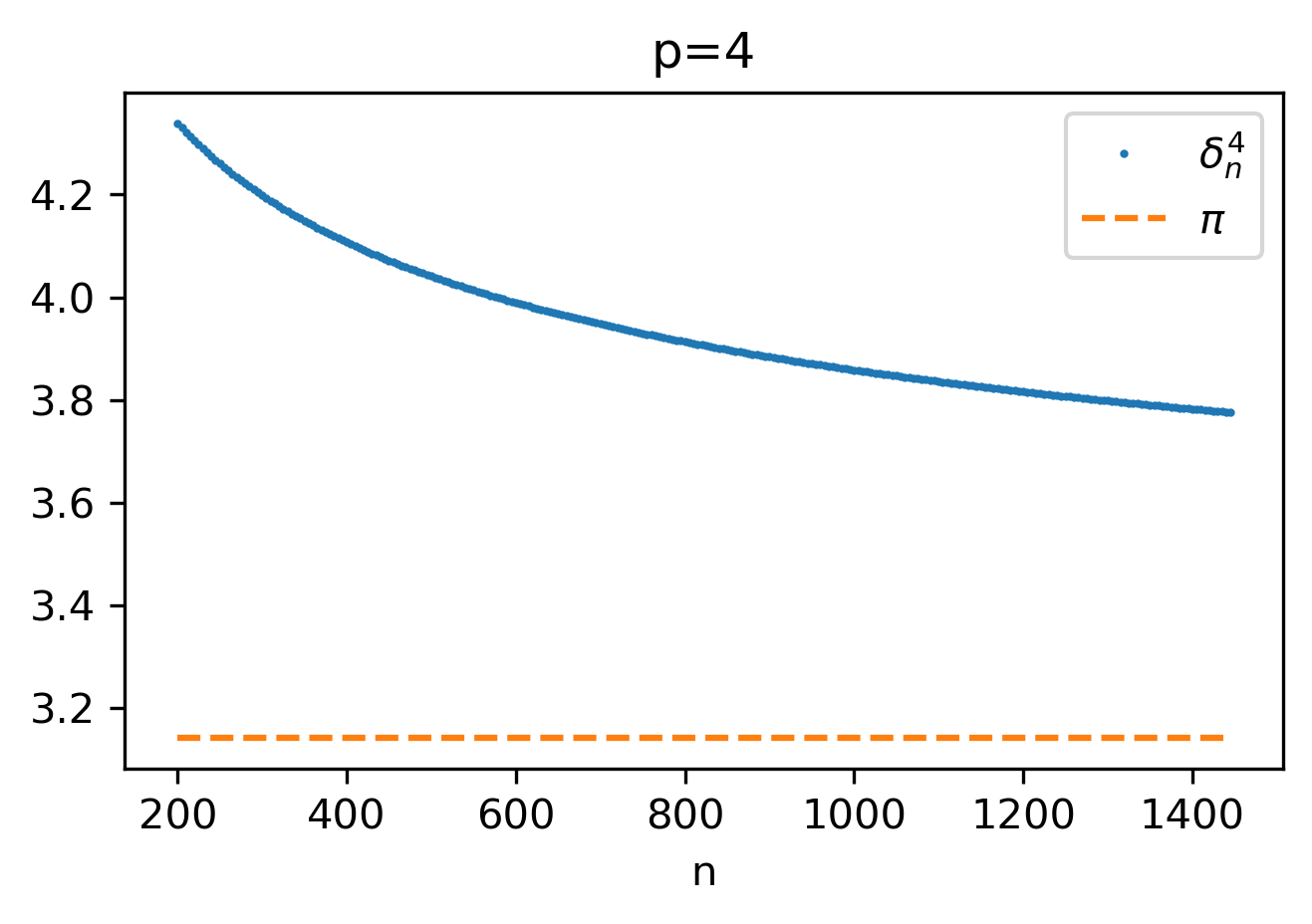}
\caption{The gap $\delta_n^p$}
\end{subfigure}
\caption{The graph of $(m(n))_{n \geq 1}$ (left) and the gap in function of $n$ (right). Parameter values $p=4$, and $\phi=\Phi_4$ with $\theta = 0.01$}
\label{f17-f18}
\end{figure}
\begin{figure}[H]
\centering
\begin{subfigure}{.5\textwidth}
\centering
\includegraphics[width=\linewidth]{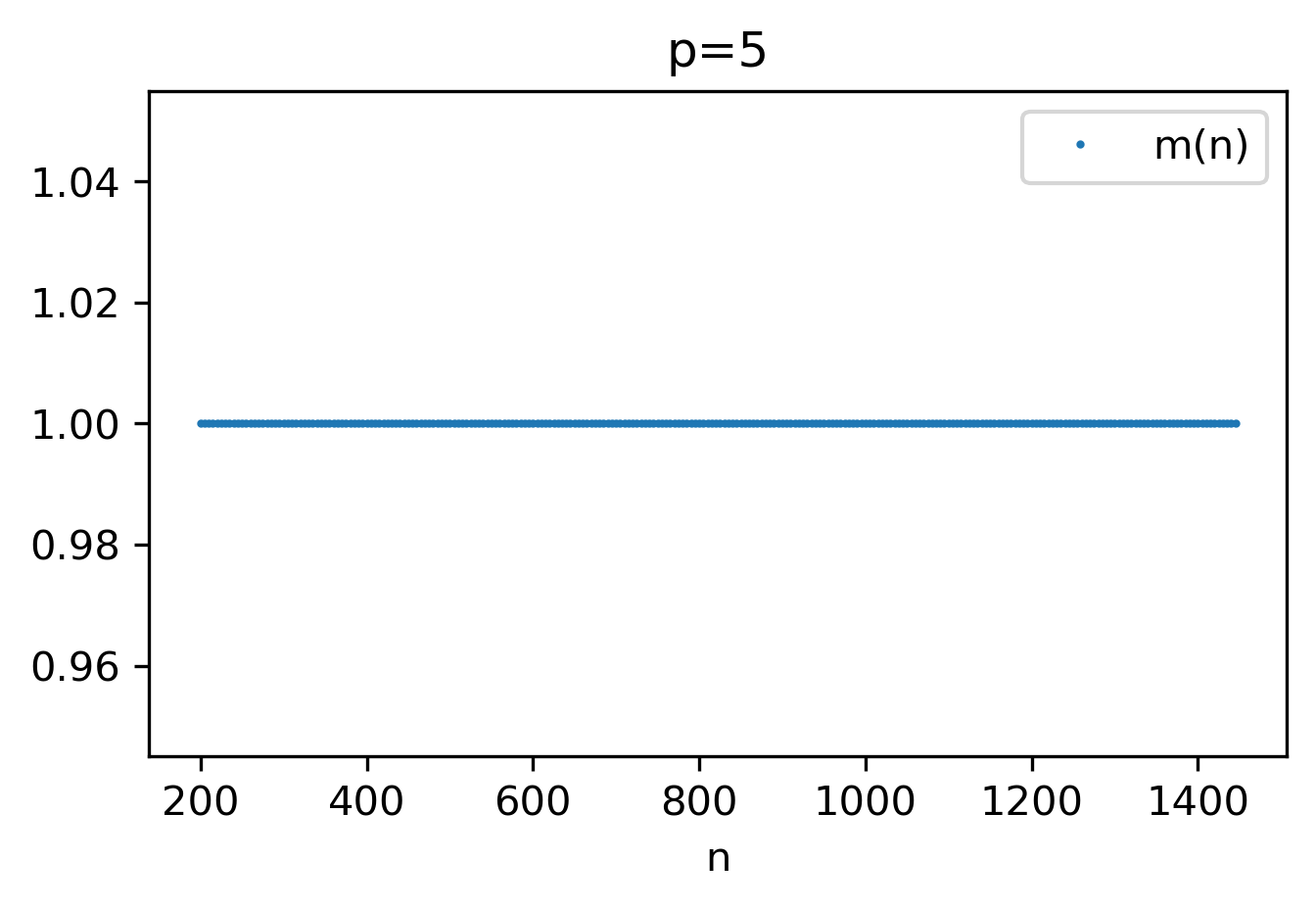}
\caption{The sequence $(m(n))_{n \geq 1}$}
\end{subfigure}%
\begin{subfigure}{.5\textwidth}
\centering
\includegraphics[width=\linewidth]{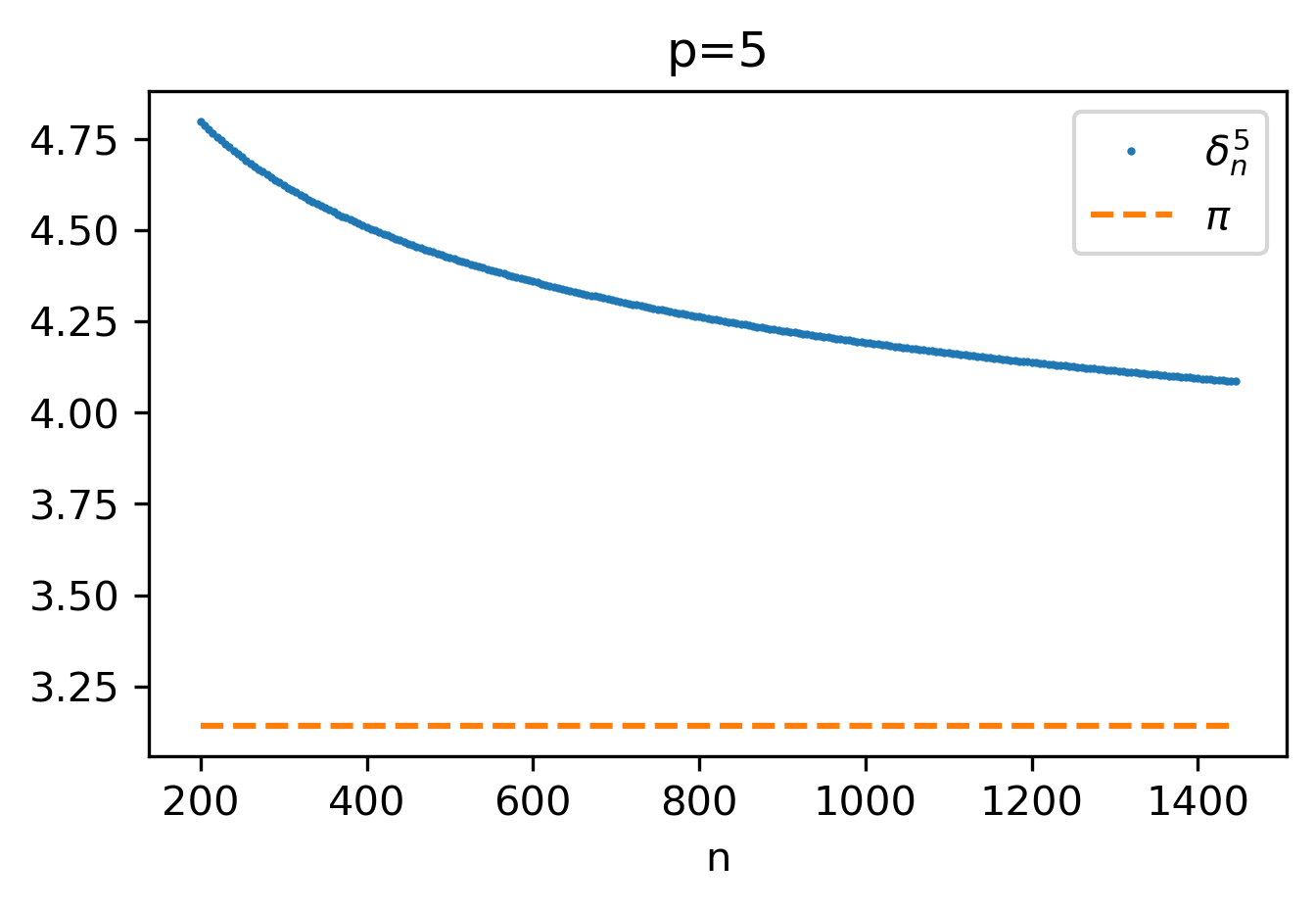}
\caption{The gap $\delta_n^p$}
\end{subfigure}
\caption{The graph of $(m(n))_{n \geq 1}$ (left) and the gap in function of $n$ (right). Parameter values $p=5$, and $\phi=\Phi_5$ with $\theta = 0.01$}
\label{f19-f20}
\end{figure}
\begin{figure}[H]
\centering
\begin{subfigure}{.5\textwidth}
\centering
\includegraphics[width=\linewidth]{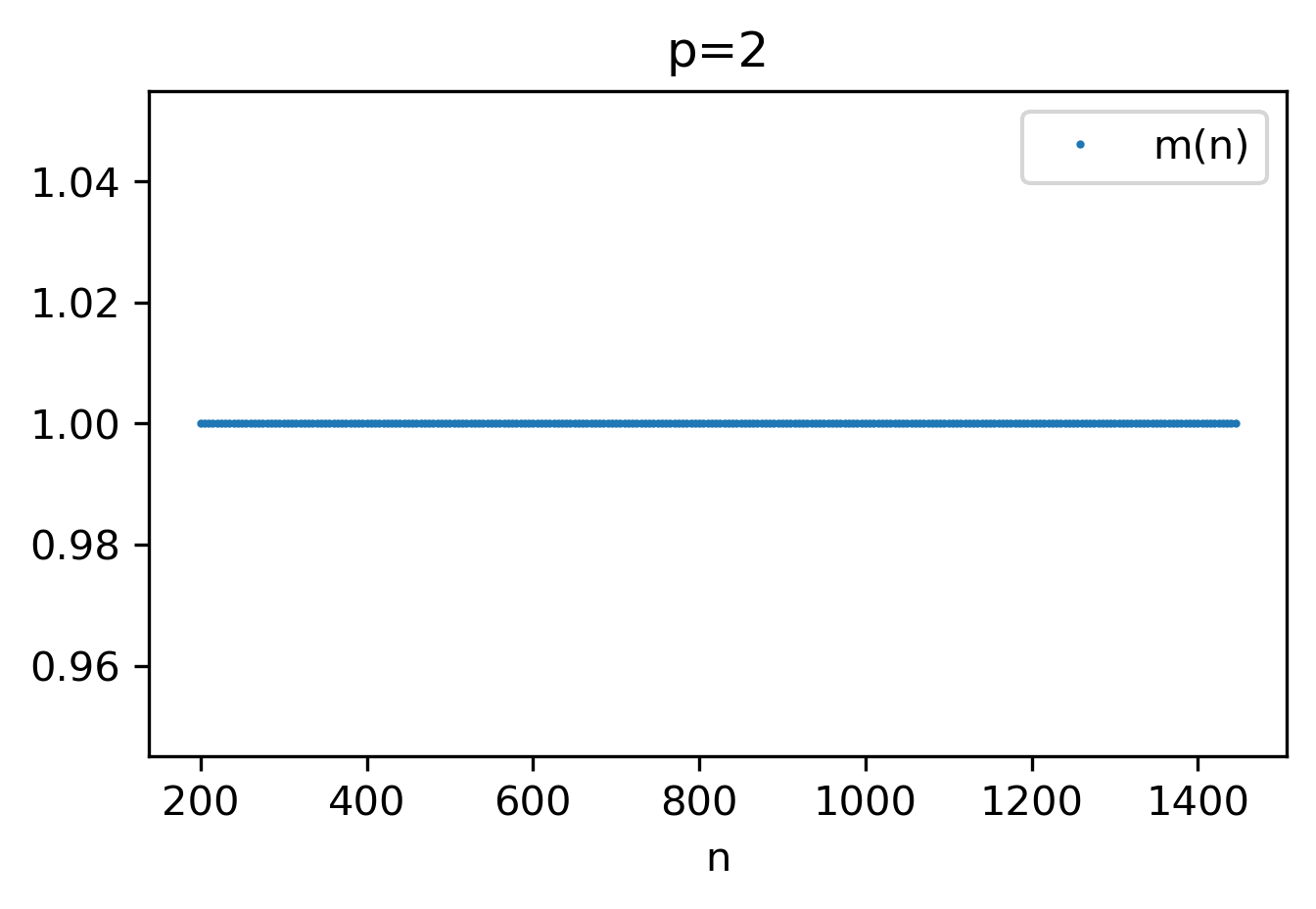}
\caption{The sequence $(m(n))_{n \geq 1}$}
\end{subfigure}%
\begin{subfigure}{.5\textwidth}
\centering
\includegraphics[width=\linewidth]{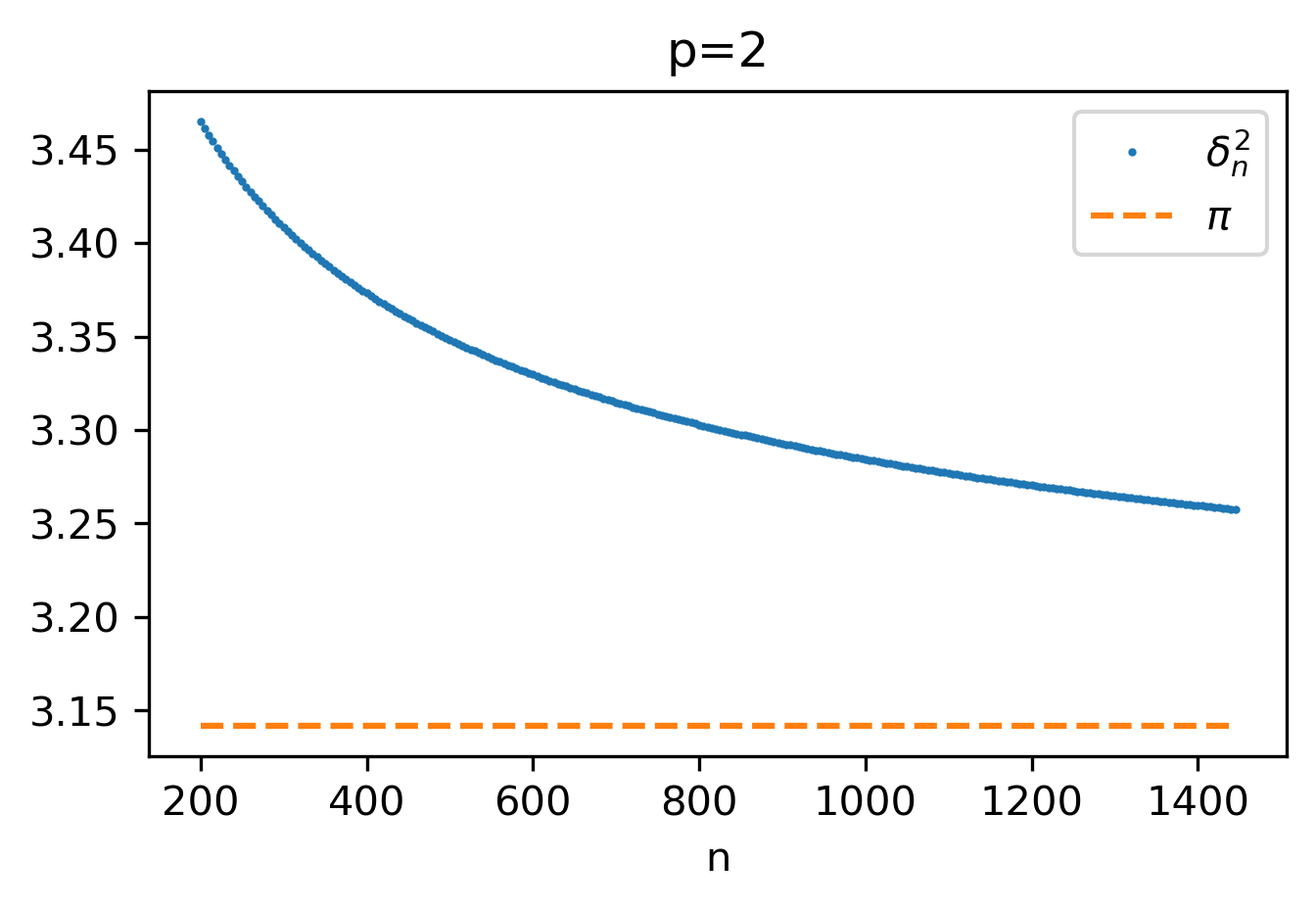}
\caption{The gap $\delta_n^p$}
\end{subfigure}
\caption{The graph of $(m(n))_{n \geq 1}$ (left) and the gap in function of $n$ (right). Parameter values $p=2$, and $\phi=\Phi_2$ with $\theta = 0.01$}
\label{f21-f22}
\end{figure} \vspace*{0.25cm}
\noindent  {\bf \em Test 9: study of the impact of $\theta$ on the behavior of the gap:} \vspace*{0.25cm} 

\noindent  \hspace{0.5cm} In the last test, we used small values of $\theta$ primarily to facilitate a comparison with the case when $p=1$. This choice was motivated by the existence of a mathematical proof demonstrating uniform stability specifically in this scenario. It's important to highlight that the values of $\theta$ are directly connected to the observability time discussed in \cite{ervedoza2016numerical}. In the following analysis, we will investigate the impact of the value of $\theta$ on the behavior of the gap and the index $m(n)$ when we use the reparameterization $\Phi_p$.

 Based on Figure \ref{f23-f24}, it is apparent that as $\theta$ increases, the convergence of the gap to $\pi$ occurs rapidly. Through various numerical tests, we have observed that, in general, the values of $\theta$ do not significantly impact the behavior of the gap, and we find that $m(n)=1$ for all $p\in \{1,\dots,27\}$.
\begin{figure}[H]
\centering
\begin{subfigure}{.5\textwidth}
\centering
\includegraphics[width=\linewidth]{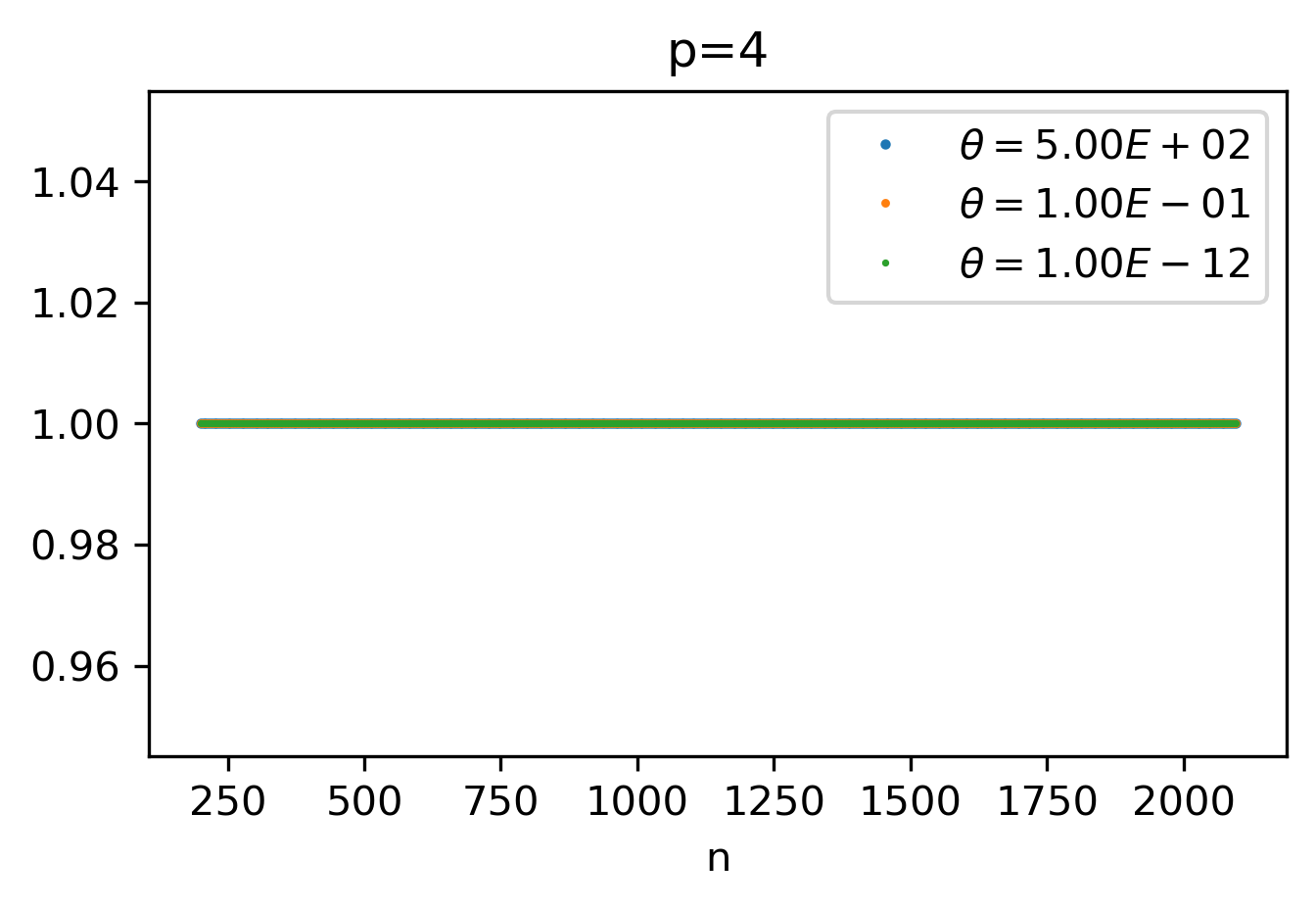}
\caption{The sequence $(m(n))_{n \geq 1}$}
\end{subfigure}%
\begin{subfigure}{.5\textwidth}
\centering
\includegraphics[width=\linewidth]{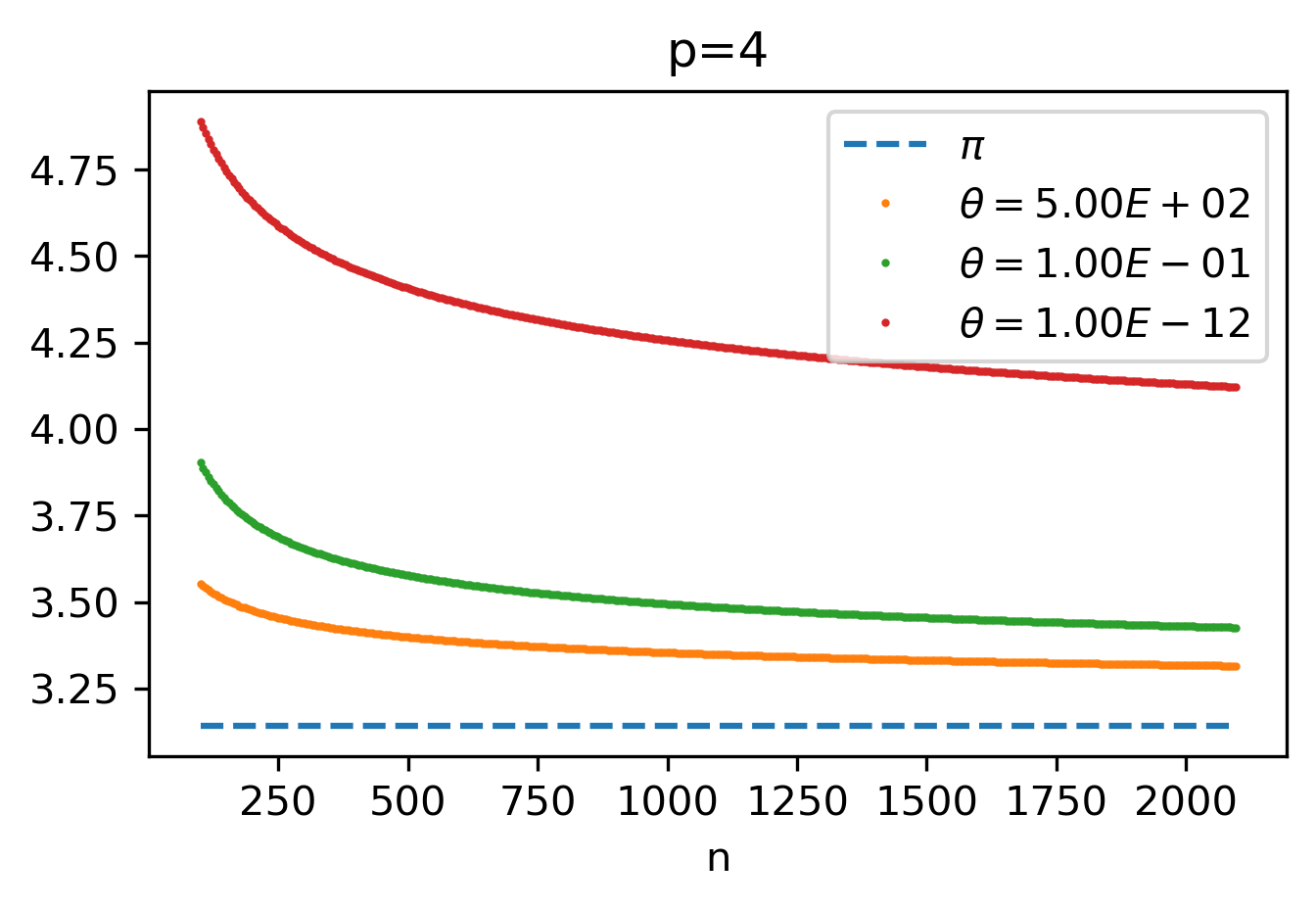}
\caption{The gap $\delta_n^p$}
\end{subfigure}
\caption{The graph of $(m(n))_{n \geq 1}$ (left) and the gap in function of $n$ (right). Parameter values $p=4$, and $\phi=\Phi_4$}
\label{f23-f24}
\end{figure}
\subsection{Summary of Numerical Tests} 
 \hspace{0.5cm} At the end of this section, we will summarize the numerical results to establish a framework for addressing mathematical questions. A comprehensive summary of all the conducted tests can be found in Table \ref{tab:summary-numerical-tests}.

In particular, in the case when $p\geq 2$ we have demonstrated numerically that there exists a significant set of  concave and convex reparametrizations for which we lose the gap condition for $p\geq 3$. For $p=2$, we achieve the gap with both convex and concave  reparametrization, but in general, not the optimal one except when we chose $\phi_3$ with  specific  values of $\theta$. 

After conducting numerous numerical tests, we observed that the reparametrization $\phi_3$ gives the gap for  some values of $\theta$ for all $p\in\{2,\dots,27\}$. However, it's essential to note that this gap is generally not optimal, and the index $m(n)$ is not bounded. Therefore, we set out to modify the function $\phi_3$  to construct  a reparametrization that  yields  the optimal gap.  We introduce the reparametrization $\Phi_p$ for this purpose. From our previous tests, it is evident that with this choice, we achieve the optimal gap condition, and the index at which the gap reaches its infimum is exactly $m_0 = 1$, mirroring the behavior observed in the case $p = 1$ with $\phi_3$ and $\theta=0.01$. In fact, the behavior of the gap and the index is consistent with the case $p=1$ when we chose $\phi_3$ as a reparametrization, for all $p\in\{2,\dots,27\}$  and for a large set of $\theta$, especially for $\theta\leq 0.01$. Furthermore, it is clear from the figures  that as $p$ increases, the rate of convergence of $\delta^p_n$ towards $\pi$ slows down.\\ 

In the paper  \cite{ervedoza2016numerical}, the reparametrization  $\phi_3$ is employed to establish uniform boundary observability of the wave equation for  $T>2\theta+2$, specifically for the case when $p=1$. However, it's worth noting that in our  analysis, we  observe that the optimal condition for uniform gap exists not only with $\phi_3$ or the other type of strictly concave reparametrization mentioned in \cite{ervedoza2016numerical}, but rather for a large set of concave or convex reparametrizations $\phi $ in $\mathbf{C}_{[0,1]}$. The remarkable aspect is the ability to use this reparametrization to achieve the optimal gap condition for all values of $p$ in $\{2,\dots,27\}$,  through a straightforward modification.

In the last numerical study, we focused on the effect of the $\theta $ on the behavior of the gap and the index $m(n)$, we showed numerically that the speed of the convergence of the gap depends on the values of $\theta$, as showed in the Figures  \ref{f23-f24}. Moreover, the behavior of the gap and $(m(n))$ is akin to the case  $p=1$ for both small values of $\theta$ and a wide range of larger values of $\theta$.
\begin{table}[H]
\centering 
\begin{tabular}{c|c|c|l|}
Test   & Objective                                                                                                                                                                           & Figures                                                                  & \multicolumn{1}{c|}{Observations}                                                                                                                                                                                                                                                                     \\ \hline
Test 1 & \begin{tabular}[c]{@{}c@{}}Examination of gap conditions \\ for linear $B$-splines\end{tabular}                                                                          & Tab. \ref{t1}, fig. \ref{f2}           & \begin{tabular}[c]{@{}l@{}}{\bf -} $(m(n))_{n \geq 1}$ is bounded ($m(n)=1$)\\ {\bf -} Optimal gap conditions observed\end{tabular}                                                                                                                          \\ \hline
Test 2 & \begin{tabular}[c]{@{}c@{}}Analysis of eigenvalue distribution  \\ for linear $B$-splines\end{tabular}                                                              & fig. \ref{fgap}                                                                     & \begin{tabular}[c]{@{}l@{}}The distance between square root eigenvalues \\ increases or decreases, in accordance with\\ Corollary \ref{cl2.5}.\end{tabular}                                                                                                               \\ \hline
Test 3 & \begin{tabular}[c]{@{}c@{}}Convergence study of the \\ symbol for linear $B$-splines\end{tabular}                                                                     & fig. \ref{f3}                                           & \begin{tabular}[c]{@{}l@{}}Convergence of discrete eigenvalues \\ towards the symbol.\end{tabular}                                                                                                                                                                                              \\ \hline
Test 4 & \begin{tabular}[c]{@{}c@{}}Investigation of gap conditions \\ for quadratic $B$-splines \\ using concave reparametrization $\phi_1$\end{tabular}                    & fig. \ref{f4-f5}                                        & \begin{tabular}[c]{@{}l@{}}{\bf -} $(m(n))_{n \geq 1}$ is not bounded, but equals $1$ a.e\\ {\bf -} Uniform gap conditions observed, \\ but not optimal\end{tabular}                                                                               \\ \hline
Test 5 & \begin{tabular}[c]{@{}c@{}}Analysis of gap conditions \\ for quadratic $B$-splines \\ using convex reparametrization $\phi_2$\end{tabular}                 & fig. \ref{f6-f7}                                        & \begin{tabular}[c]{@{}l@{}}{\bf -} $(m(n))_{n \geq 1}$ follows a linear pattern\\ {\bf -} Uniform gap conditions observed, \\ but not optimal\end{tabular}                                                                                                   \\ \hline
Test 6 & \begin{tabular}[c]{@{}c@{}}Study of optimal gap conditions \\ for quadratic $B$-splines \\ using concave reparametrization $\phi_3$\end{tabular} & fig. \ref{f8-f11}                                       & \begin{tabular}[c]{@{}l@{}}Similar behavior as when $p=1$:\\ {\bf -} $(m(n))_{n \geq 1}$ is bounded\\ {\bf -} Optimal gap conditions observed\end{tabular}                                                                                          \\ \hline
Test 7 & \begin{tabular}[c]{@{}c@{}}Examination of gap conditions \\ when $p \geq 3$\end{tabular}                                                                                & figs.  \ref{f12-f13-}--\ref{f16-f17-}, & \begin{tabular}[c]{@{}l@{}}{\bf -} The reparametrizations $\phi_1$ and $\phi_2$ yield \\ similar behavior  of $m(n)$ as when $p=2$ with $\phi_2$\\ {\bf -}We lose the gap condition in the case \\of $\phi_1$ and $\phi_2$\\ {\bf -} Different behavior from that of the case  $p=2$ \\  is observed  when we use  $\phi_3$\end{tabular} \\ \hline
Test 8 & \begin{tabular}[c]{@{}c@{}}Investigation of optimal gap \\ conditions when $p \geq 3$ \\ using the concave reparametrization $\Phi_p$\end{tabular}                     & figs. \ref{f15-f16}--\ref{f21-f22}     & \begin{tabular}[c]{@{}l@{}}Gap behavior is identical to that of \\ the case $p=1$ with $\phi_3$\end{tabular}                                                                                                                                                                                                    \\ \hline
Test 9 & \begin{tabular}[c]{@{}c@{}}Analysis of the impact of $\theta$\\  on gap behavior\end{tabular}                                                                               & fig. \ref{f23-f24}                                      & \begin{tabular}[c]{@{}l@{}}As $\theta$ increases, the convergence of the gap \\ toward the optimal value $\pi$ occurs rapidly.\end{tabular}                                                                                                                                                           \\ \hline
\end{tabular}
\caption{Summary of numerical tests}
\label{tab:summary-numerical-tests}
\end{table}

\section{Conclusions and  further comments}\label{sec:conclusions}
$\hspace{0.5cm}$We have shown that GLT theory can provide complex and novel information about
the distribution of eigenvalues and the impact of reparametrization on the spectrum. By applying our results, we have established a strong connection between GLT theory and control theory. In the numerical  section, we have demonstrated that the average gap condition is not equivalent to the uniform gap property. Furthermore, we have provided evidence that it is indeed possible to construct reparametrizations within the set  $\mathbf{C}_{[0,1]}$  that fulfill the boundedness condition for the index $m(n)$, enabling us to apply the Theorem \ref{cl2.4}.\\

The analysis presented in this study is related to the IgA discretization of the Laplace operator, utilizing regular B-splines. Throughout, our investigation of  eigenvalues distribution and gap behavior  reveals that  our analysis is independent of the specific discretization method and the particular operator under consideration. What remains crucial is the matrix's symbol resulting from a numerical discretization of any given operator. Consequently, we can extend the same study to finite difference discretization and higher-order Lagrangian finite element approximation.\\

The main challenge in this work revolves around establishing the boundedness of the sequence $(m(n))$. Unfortunately, our symbol analysis cannot provide insight into this condition,  as it depends on the choice of reparametrization. As illustrated in  our numerical analysis, achieving the boundedness of $(m(n))$  requires more than just ensuring convexity in the reparametrization. Rather, it necessitates a deeper understanding of how the chosen reparametrization impacts the distribution of eigenvalues, which means a control on the eigenvalues from the set $\mathbf{C}_{[0,1]}$. The only results in this context are Theorems \ref{tdis1} and \ref{dis1}, along with the Theorem \ref{dis2}. These theorems show that ordering  reparametrizations can affect the order of the  eigenvalues, but to establish the boundedness of $(m(n))$,  it is imperative to order the associated gaps.\\

Lastly, it is worth noting that all the findings presented in this paper, which encompass the improvement of the results in \cite{bianchi2021analysis}, along with our analysis of  eigenvalues distribution, can be extended to domains of higher dimensions (see \cite{serra2006glt}, \cite{barbarino2020block}).

\bibliographystyle{abbrv}
\bibliography{control-glt}

\end{document}